\documentclass[final]{amsart}
\usepackage{amsmath, amssymb, mathtools}
\usepackage{hyperref,enumerate}
\usepackage{cleveref}
\usepackage{graphicx}
\usepackage{comment}
\usepackage[ruled]{algorithm2e}
\usepackage[margin=1.25in]{geometry}

\usepackage{siunitx}
\usepackage{esint}

\DeclareMathOperator*{\argmin}{arg\,min}
\DeclareMathOperator*{\argmax}{arg\,max}
\DeclareMathOperator*{\esssup}{ess.sup}
\DeclareMathOperator*{\interior}{int}
\DeclareMathOperator{\Lip}{Lip}
\newcommand{\diff}{\textup{d}}
\DeclareMathOperator{\osc}{osc}

\newcommand{\wt}[1]{\widetilde{#1}}
\newcommand{\bv}[1]{\mathbf{#1}}

\def\l{\left}
\def\r{\right}

\newcommand{\pO}{\partial\Omega} 
\newcommand{\Oc}{\overline{\Omega}}
\newcommand{\Bc}{\overline{B}}
\newcommand{\mR}{\mathbb{R}}
\def\mRd{\mathbb{R}^d}
\def\Nh{\mathcal{N}_h}
\newcommand{\Nhi}[1]{\mathcal{N}_{h,#1}^I}
\newcommand{\Nhb}[1]{\mathcal{N}_{h,#1}^b}
\def\Nhj{\mathcal{N}_{h_j}}

\newcommand{\vertex}{{\mathtt{z}_h}}
\newcommand{\othervertex}{{\mathtt{z}_h'}}

\def\ve{\varepsilon}

\def\uve{u_{\frakh}}

\def\wtNhxi{\widetilde{\mathcal N}_\frakh(\vertex)}

\newcommand{\De}{\mathrm{D}}

\DeclareMathOperator{\dist}{dist}
\DeclareMathOperator{\diam}{diam}
\newcommand{\USC}{\mathrm{USC}}
\newcommand{\LSC}{\mathrm{LSC}}

\DeclareMathOperator{\trace}{tr}

\def\Th{\mathcal{T}_h}
\def\Oh{\Omega_h}
\def\Vh{\mathbb{V}_h}
\def\Xh{\mathbb{X}_h}
\def\mS{\mathbb{S}}
\def\St{\mathbb S_{\theta}}
\def\interp{\mathcal I_h}

\newcommand{\calC}{{\mathcal{C}}}

\newcommand{\frakh}{{\mathfrak{h}}}
\newcommand{\bT}{{\mathbb{T}}}

\newtheorem{Theorem}{Theorem}[section]
\newtheorem{Lemma}[Theorem]{Lemma}

\newtheorem{Corollary}[Theorem]{Corollary}
\theoremstyle{definition}
\newtheorem{Remark}[Theorem]{Remark}
\newtheorem{Definition}[Theorem]{Definition}

\newtheorem{Example}[Theorem]{Example}

\numberwithin{equation}{section}

\begin{document}

\title[Normalized infinity Laplacian]{Convergent, with rates, methods for normalized infinity Laplace, and related, equations}

\author[W.~Li]{Wenbo Li}
\address[W.~Li]{Department of Mathematics, University of Tennessee, Knoxville TN 37996 USA.}
\curraddr{Institute of Computational Mathematics and Scientific/Engineering Computing of the Chinese Academy of Sciences, Beijing 100190 China.}
\email{liwenbo@lsec.cc.ac.cn}

\author[A.J.~Salgado]{Abner J.~Salgado}
\address[A.J.~Salgado]{Department of Mathematics, University of Tennessee, Knoxville TN 37996 USA.}
\email{asalgad1@utk.edu}

\date{Draft version of \today}

\makeatletter
\@namedef{subjclassname@2020}{%
  \textup{2020} Mathematics Subject Classification}
\makeatother

\subjclass[2020]{Primary:
65N06,
65N12,
65N15.
Secondary: 35J94, 35J70, 35J87, 35D40.
}

\keywords{Normalized infinity Laplacian;
Optimal Lipschitz extensions;
Tug of war games;
Monotonicity;
Obstacle problems;
Viscosity solutions;
Degenerate elliptic equations;
Finsler norms.}

\begin{abstract}
We propose a monotone, and consistent numerical scheme for the approximation of the Dirichlet problem for the normalized Infinity Laplacian, which could be related to the family of so--called two--scale methods. We show that this method is convergent, and prove rates of convergence. These rates depend not only on the regularity of the solution, but also on whether or not the right hand side vanishes. Some extensions to this approach, like obstacle problems and symmetric Finsler norms are also considered.
\end{abstract}

\maketitle



\section{Introduction} \label{sec:Intro}

We begin our discussion by presenting two motivations for the type of problems we wish to study.

\subsection{Motivation: Optimal Lipschitz extensions} \label{sub:MotivateLinf}
The fundamental problem in the calculus of variations \cite{MR2361288} is the minimization of a functional $\mathcal{I}$, i.e., we wish to find
\[
  \min\left\{ \mathbb{I}[w]: w \in \mathcal{A} \right\},
\]
where $\mathcal{A}$ is the \emph{admissible set}, that is a subset of a function space over a domain $\Omega \subset \mR^d$ ($d\geq1$) and encodes, for instance, boundary behavior of the function. A simple example of this scenario is
\begin{equation}
\label{eq:minimizeW1p}
  p \in (1,\infty), \qquad \mathbb{I}_p[w] = \frac1{|\Omega|}\int_\Omega |\De w(x) |^p \diff x, \qquad \mathcal{A}_g^p = \left\{ w \in W^{1,p}(\Omega): w_{|\partial\Omega} = g \right\},
\end{equation}
whose Euler Lagrange equation gives rise to the boundary value problem (BVP)
\begin{equation}
\label{eq:pLaplaceBVP}
  \Delta_p u  = 0, \text{ in } \Omega, \qquad u_{|\partial\Omega} = g.
\end{equation}
Here we introduced the so--called $p$--Laplacian
\[
  \Delta_p w(x) =\De\cdot(|\De w(x)|^{p-2}\De w(x) )  =  |\De w(x)|^{p-2}\Delta w(x) + (p-2)|\De w(x)|^{p-4}\De w(x) \otimes \De w(x): \De^2 w(x),
\]
where $\otimes$ denotes tensor product, and $:$ is the Frobenius inner product, i.e.,
\[
  (\mathbf{a} \otimes \mathbf{b})_{i,j} = \mathbf{a}_i \mathbf{b}_j, \qquad \bv{A}:\bv{B} = \trace (\bv{A}^\intercal\bv{B}) = \sum_{i,j=1}^d \bv{A}_{i,j}\bv{B}_{i,j}.
\]
Since $\mathbb{I}_p$ is strictly convex ($p>1$), to minimize $\mathbb{I}_p$ it is sufficient to solve \eqref{eq:pLaplaceBVP}. Owing to the divergence structure of the $p$--Laplacian, \eqref{eq:pLaplaceBVP} is a variational problem and it can be tackled with the notion of weak solutions. Thus, the PDE \cite{MR3014456,MR3764541,MR3424623,MR2194585} and approximation theory \cite{MR1930132,MR1192966,MR2317830,MR2383205,MR2911397,MR2135783} for this problem is, to a very large extent, well developed.

Let us consider now, instead, the problem of minimizing the $L^\infty$--norm of the gradient,
\begin{equation}
\label{eq:minimizeLip}
  \mathbb{I}_\infty[w] = \| \De w \|_{L^\infty(\Omega;\mR^d)} = \esssup \{ |\De w(x)|: x \in \Omega \}
\end{equation}
over the admissible set $\mathcal{A}_g^\infty$. It is not completely evident now how to carry variations for $\mathbb{I}_\infty$, so obtaining the Euler Lagrange equation is not immediate. Using that, as $p \uparrow \infty$, we have $\mathbb{I}_p[w]^{1/p} \to \mathbb{I}_\infty[w]$ one may argue that the Euler Lagrange equation for \eqref{eq:minimizeLip} is the limit, as $p \uparrow \infty$, of \eqref{eq:pLaplaceBVP}. Divide $\Delta_p w$ by $(p-2)|\De w(x)|^{p-4}$ to obtain that, as $p \uparrow \infty$,
\begin{align*}
  \frac1{(p-2)|\De w|^{p-4}}\Delta_p w(x) &= \frac{|\De w(x)|^2}{p-2}\Delta w (x)+ \De w(x) \otimes \De w(x) : \De^2 w(x) \\
  &\to \Delta_\infty w (x)= \De w(x) \otimes \De w(x) : \De^2 w(x).
\end{align*}
The operator $\Delta_\infty$ is the \emph{$\infty$--Laplacian}. Heuristically, to minimize \eqref{eq:minimizeLip}, we solve the BVP 
\begin{equation}
\label{eq:ELinfLap}
  \Delta_\infty u  = 0, \text{ in } \Omega, \qquad u_{|\partial\Omega} = g.
\end{equation}
This derivation was originally discussed by G.~Aronsson \cite{MR0196551,MR0203541,MR0240690} (see also \cite{MR0237962,MR0217665}), who considered the problem of \emph{minimal Lipschitz extensions}: given $g \in C^{0,1}(\partial\Omega)$ find $u \in C^{0,1}(\bar\Omega)$ such that $u_{|\partial\Omega} = g$ and its Lipschitz constant is as small as possible. He arrived at \eqref{eq:ELinfLap} via the aforementioned approximation procedure. He showed existence of solutions and provided explicit examples that show the optimal expected regularity for solutions of \eqref{eq:ELinfLap}: $u \in C^{1,1/3}_{\mathrm{loc}}$. Uniqueness, stability, regularity, and further properties of solutions were not known to him.

The reason for these shortcomings is that, since the $\infty$--Laplacian does not have divergence structure, the correct way of interpreting \eqref{eq:ELinfLap} is via the theory of viscosity solutions \cite{MR1118699,MR2084272}, which did not exist at the time. This theory made it possible to obtain uniqueness, stability, and some properties of solutions \cite{MR3289084,MR3467690}. However, the regularity theory for \eqref{eq:ELinfLap} is limited: the best known results being interior $C^{1,\alpha}$ regularity \cite{MR2393071,MR3592681,MR1886623} and $C^1$ regularity in two dimensions \cite{MR2185662,MR3117143,MR3477071,MR3193978,MR3119202,MR3008417}. For further historical comments we refer to \cite[Section 8]{MR2408259}.

Although a full theory for problem \eqref{eq:ELinfLap} is lacking, we see that this problem appears as a prototype for the \emph{calculus of variations in $L^\infty$} \cite{MR3289084}. The minimal Lipschitz extension problem can also serve as a model for inpainting \cite{MR1669524,Casas1996}.

The ensuing applications justify the need for numerical schemes, which can be proved to converge, and some results exist in this direction. For instance, since solutions to \eqref{eq:ELinfLap} can be obtained as limits, when $p\uparrow\infty$, of solutions to \eqref{eq:pLaplaceBVP}, one idea is to discretize \eqref{eq:pLaplaceBVP} with standard techniques and then pass to the limit. This has been explored in \cite{MR3841500}, see also \cite{MR3617088,MR3569561}, where convergence of this method is established. The nature of this passage to the limit, and the limited available regularity of the solution to \eqref{eq:ELinfLap} makes obtaining rates of convergence very unlikely. The direct approximation of \eqref{eq:ELinfLap} is established in \cite{MR2137000,MR3061067} via wide stencil finite difference methods. To ascertain convergence the classical theory of Barles and Souganidis \cite{MR1115933} is invoked; see however \cite{MR3823876}. Rates of convergence for this method are not provided.

\subsection{Motivation: Tug of war games} \label{sub:TugOfWar}
A variant of the infinity Laplacian also appears in some \emph{tug of war games} \cite{MR2449057,MR3280579,RossiBook,MR2846345} as we now describe. Fix a step size $\varepsilon>0$, and place a token at $x_0 \in \Omega$. Two players, at stage $k \geq 0$, are allowed to choose points $x_k^I, x_k^{II} \in \bar B_\ve(x_k) \cap \bar\Omega$, respectively. A fair coin is tossed: heads means $x_{k+1} = x_k^I$, while tails gives $x_{k+1}=x_k^{II}$. If $x_{k+1} \in \Omega$ the game continues, but if $x_{k+1} \in \partial\Omega$, the game ends and Player II pays Player I the amount
\[
  P = g(x_{k+1}) + \frac{\varepsilon^2}2 \sum_{j=0}^k f(x_j),
\]
where the terminal payoff $g$, and running cost $f$ are given. Player I attempts to maximize $P$, while Player II to minimize it. The expected payoff (value) $u^\varepsilon$ of this game satisfies the relation
\begin{equation}
\label{eq:meanvalue}
  2u^\varepsilon(x) - \left( \sup_{y \in \Bc_\ve(x) \cap \bar\Omega} u^\varepsilon(y) + \inf_{y \in \Bc_\ve(x) \cap \bar\Omega} u^\varepsilon(y)\right) = \varepsilon^2 f(x),
\end{equation}
which in the (formal) limit $\varepsilon \downarrow0$ gives
\begin{equation}
\label{eq:BVPNinfLaplace}
  -\Delta_\infty^\diamond u = f, \ \text{ in } \Omega, \qquad u_{|\partial\Omega} = g.
\end{equation}
The operator $\Delta_\infty^\diamond$ is defined as
\[
  \Delta_\infty^\diamond w (x)= \frac1{|\De w(x)|^2} \Delta_\infty w(x) = \frac{\De w(x)}{|\De w(x)|} \otimes \frac{\De w(x)}{|\De w(x)|}:\De^2w(x), 
\]
and, for obvious reasons, is known as the \emph{normalized} $\infty$--Laplacian.

The authors of \cite{MR2846345} used \eqref{eq:meanvalue} as starting point to develop a semi--discrete scheme for \eqref{eq:BVPNinfLaplace}. Namely, they defined the operator
\begin{equation}
\label{eq:discInfLap}
  \Delta_{\infty,\varepsilon}^\diamond w(x) = \frac1{\varepsilon^2} \left( \sup_{y \in \Bc_\ve(x) \cap \bar\Omega} w(y) - 2w(x) + \inf_{y \in \Bc_\ve(x) \cap \bar\Omega} w(y)\right),
\end{equation}
and studied the problem
\begin{equation}
\label{eq:GameInfLapEps}
  -\Delta_{\infty,\varepsilon}^\diamond  u = f, \text{ in } \Omega, \qquad u_{|\partial\Omega} = g.
\end{equation}
Existence and uniqueness of solutions, as well as convergence with rates, as $\varepsilon \to 0$, are established for $f>0$.

\subsection{Goals and organization}

Despite the fact that the approximation theory for viscosity solutions is rather primitive \cite{MR3049920,MR3653852}, the purpose of this manuscript is to develop convergent, with rates, numerical schemes for \eqref{eq:ELinfLap}, \eqref{eq:BVPNinfLaplace}, and related problems. Let us describe how we organize our discussion in order to achieve these goals.

We begin by establishing the notation and main assumptions we shall operate under in \Cref{sec:Notation}. In \Cref{sub:InfLapVisco}, we also review the existing theory around \eqref{eq:BVPNinfLaplace}. The spatial discretization, and our numerical scheme \emph{per se}, are presented in \Cref{sub:Mesh} and \Cref{sub:TwoScale}, respectively. The analysis of the method begins in \Cref{sec:ExistenceAndSuch}, where we establish monotonicity and comparison principles. The comparison principles need to distinguish whether $f$ is of constant sign or $f\equiv0$. These properties expediently give us, as shown in \Cref{sub:ExistenceUniqueness}, existence, uniqueness, and stability of discrete solutions. The interior consistency of the discrete operator is studied next, in \Cref{sub:Consistency}. This, together with a barrier argument, and an approach \emph{\`a la} Barles and Souganidis \cite{MR1115933} give us convergence; see \Cref{sub:convergence}.

Having established convergence the next step is to provide, under realistic regularity assumptions, rates of convergence. This is the main goal of \Cref{sec:Rates}. Once again, we need to distinguish between the cases when $f$ is of constant sign or $f\equiv0$. In the first scenario we obtain a rate that is, at best, $\mathcal{O}(h^{1/3})$; whereas in the second $\mathcal{O}(h^{1/4})$. To the best of our knowledge, these are the first of their kind for this type of problems.

In \Cref{sec:Variants} we then turn our attention to some variants of our main theme. First, in \Cref{sub:PointClouds} a meshless discretization is briefly discussed. This only requires us to have a cloud of points that is, in a sense, sufficiently dense in the domain. Similar results as in the previous case can be obtained in this case. A variant of the tug of war game presented in \Cref{sub:TugOfWar} which leads to an obstacle problem for the normalized infinity Laplacian is then discussed in \Cref{sub:Obstacle}. A convergent numerical scheme is presented and analyzed. The rates, in this case, are the same as for the Dirichlet problem. We emphasize that these are the first of their kind. The last variation we consider is detailed in \Cref{sub:Finsler}, where we discuss the so--called Finsler infinity Laplacian. This is an operator that arises from a variant of our tug of war game, and can be thought of as an anisotropic version of our original operator. For its Dirichlet problem we sketch the construction of a convergent scheme.

\Cref{sec:Solution} details some practical aspects. We explain why the usually adopted approaches for the solution of the ensuing discrete problems may not work, and discuss an alternative. A convergent fixed point iteration is presented and analyzed. Finally, \Cref{sec:Numerics} presents some numerical experiments that illustrate our theory.


\section{Notation and preliminaries}\label{sec:Notation}

Throughout our work $\Omega \subset \mR^d$, with $d \geq 1$, will be a bounded domain whose boundary is at least continuous; see \cite[Definition 1.2.1.2]{MR775683}. For a point $z \in \mRd$ and $r>0$ by $B_r(z)$ and $\Bc_r(z)$ we denote, respectively, the open and closed balls centered at $z$ and of radius $r$. For a vector $\bv{v} \in \mR^d$, by $|\bv{v}|$ we denote its Euclidean norm. If $\bv{M} \in \mR^{d\times d}$, then $|\bv{M}|$ is the induced operator norm.

We will follow standard notation and terminology regarding function spaces. In particular, if $w : \Omega \to \mR$ is sufficiently smooth and $x \in \Omega$, we denote by $\De w(x) \in \mR^d$ and $\De^2 w(x) \in \mR^{d \times d}$ the gradient and Hessian, respectively, of $w$ at $x$.

The relation $A \lesssim B$ shall mean that $A \leq c B$ for a constant $c$, whose value may change at each occurrence, but does not depend on $A$, $B$, nor the discretization parameters. By $A \gtrsim B$ we mean $B \lesssim A$, and $A \approx B$ means $A \lesssim B \lesssim A$. The Landau symbol is $\mathcal O$.

We assume that our data, satisfies:
\begin{enumerate}[\textbf{RHS}.1]
  \item \label{Ass.RHS.cont} The right hand side $f \in C(\Omega) \cap L^\infty(\Omega)$.
  \item \label{Ass.RHS.f<>0}Either:
  \begin{enumerate}[\textbf{RHS}.2a]
    \item \label{Ass.RHS.sign} $\sup\{f(x): x \in \Omega\} <0$ or $\inf\{f(x): x \in \Omega\} >0$.
    \item \label{Ass.RHS.zero} $f \equiv 0$.
  \end{enumerate}
\end{enumerate}

In addition,
\begin{enumerate}[\textbf{BC}.1]
  \item \label{Ass.BC.cont} The boundary datum $g \in C(\partial\Omega)$.
  \item \label{Ass.BC.ext} For every $\ve>0$ we have an approximation $\wt{g}_\ve \in C(\Oc)$ to $g$, such that if for some $\alpha \in [0,1]$ we have $g \in C^{0,\alpha}(\partial\Omega)$, then $\wt{g}_\ve \in C^{0,\alpha}(\Oc)$; and, in addition, 
  \[
    \| g - \wt{g}_\ve \|_{L^\infty(\partial\Omega)} \lesssim \ve^\alpha.
  \]
\end{enumerate}
We comment that Assumption~\textbf{BC}.\ref{Ass.BC.ext} can be realized, for instance, via sup-- or inf--convolutions at level $\ve$; see \cite{MR934877},  \cite[Definition 2.1]{MR1017328}, and \cite[Proposition 7.11]{MR3653852}. Further assumptions will be introduced as needed.

\subsection{The normalized infinity Laplacian} \label{sub:InfLapVisco}
We note that for $x \in \Omega$, even if $w \in C^\infty(\Omega)$ the expression $\Delta_\infty^\diamond w(x)$ does not make sense whenever the gradient of $w$ vanishes at $x$. For this reason we define the upper and lower semicontinuous envelopes, respectively, of the mapping $x \mapsto \Delta_\infty^\diamond w(x)$ to be
\begin{equation}\label{eq:DefOfDeltaPlus}
  \Delta_{\infty}^+ w(x) = \begin{dcases}
    \Delta_\infty^\diamond w(x), & \De w(x) \neq \boldsymbol0, \\
    \lambda_{\max}(\De^2 w(x)), & \De w(x) = \boldsymbol0,
  \end{dcases}
\end{equation}
and
\begin{equation}
\label{eq:DefOfDeltaMinus}
  \Delta_{\infty}^- w(x) = \begin{dcases}
    \Delta_\infty^\diamond w(x),  & \De w(x) \neq \boldsymbol0, \\
    \lambda_{\min}(\De^2 w(x)), &  \De w(x) = \boldsymbol0.
  \end{dcases}
\end{equation}
Here, for a symmetric matrix $\bv{M}\in \mR^{d \times d}$, by $\lambda_{\max}(\bv{M})$ and $\lambda_{\min}(\bv{M)}$ we denote the largest and smallest eigenvalues of $\bv{M}$. We recall that these can be characterized as
\[
  \lambda_{\max}(\bv{M}) = \max_{|\bv{v}| = 1} \bv{v}^\intercal \bv{M v}, \qquad 
  \lambda_{\min}(\bv{M}) = \min_{|\bv{v}| = 1} \bv{v}^\intercal \bv{ M v}.
\]

The operators $\Delta_\infty$ and $\Delta_\infty^\diamond$ are degenerate elliptic in the sense of \cite{MR1118699}. The theory of viscosity solutions yields existence and uniqueness of solutions to \eqref{eq:ELinfLap}; see \cite{MR3289084}. Armed with the envelopes \eqref{eq:DefOfDeltaPlus} and \eqref{eq:DefOfDeltaMinus}, existence of solutions can also be asserted for \eqref{eq:BVPNinfLaplace}; see \cite{MR2846345,MR3467690}. Uniqueness can only be guaranteed when Assumption~\textbf{RHS}.\ref{Ass.RHS.f<>0} holds \cite{MR2372480}. Finally, if $f \equiv 0$, we have that the viscosity solutions to \eqref{eq:ELinfLap} and \eqref{eq:BVPNinfLaplace} coincide; see \cite[Remark 2.2]{MR2846345}.

\subsection{Space discretization}\label{sub:Mesh}

We now begin to describe our method for the approximation to the solution of \eqref{eq:BVPNinfLaplace}. We begin by introducing some notation. For $r>0$ we set
\[
  \Omega^{(r)} = \left\{ x \in \Omega : \dist(x,\pO) > r \right\}.
\]

Let  $\{\Th\}_{h>0}$ be a family of unstructured meshes that, for each $h>0$, consist of closed simplices $T$. This family is quasiuniform in the usual, finite element, sense \cite{MR1930132}. For each $h>0$ we set
\[
  h_T = \diam(T) , \quad h = \max_{T \in \Th}h_T, \quad \Oh = \interior \bigcup\left\{ T : T \in \Th\right\},
\]
where by $\interior A$ we denote the interior of the set $A \subset \mRd$. The set of nodes of $\Th$ is denoted by $\Nh$.

We need to quantify how $\Omega_h$ approximates $\Omega$. To do so we assume that, for all $h>0$,
\[
  \Omega^{(h)} \subset \Omega_h \subset \Omega.
\]
For $r>0$, the set $\Nhi{r} = \Nh \cap \Omega^{(2r)}$ is the set of interior nodes, and $\Nhb{r} = \Nh \setminus \Nhi{r}$ is the set of boundary nodes. Note that the condition above guarantees that, if $h$ is sufficiently small and $r \geq h$, then neither of these sets is empty.

Let $\Vh$ be the space of continuous piecewise linear functions subject to $\Th$, and $\interp : C(\Oc) \to \Vh$ be the Lagrange interpolant, i.e.,
\[
  \interp w(x) = \sum_{\vertex \in \Nh} w(\vertex) \widehat{\varphi}_\vertex(x), \qquad \forall w \in C(\Oc).
\]
Here $\{\widehat{\varphi}_\vertex\}_{\vertex \in \Nh} \subset \Vh$ is the canonical (hat) basis of $\Vh$, i.e., $\widehat{\varphi}_\vertex(\othervertex) = \delta_{\vertex,\othervertex}$ for all $\vertex, \othervertex \in \Nh$. We recall that, as a consequence, these functions are nonnegative and, moreover, form a partition of unity on $\Oc$.

\subsection{The numerical method}\label{sub:TwoScale}

Here and in what follows we will call the meshsize $h$ the fine scale. The coarse scale shall be given by $\varepsilon \in [h,\diam(\Omega)]$. Inspired by \eqref{eq:discInfLap} we could consider a fully discrete operator of the form
\[
  (w,\vertex) \mapsto -\Delta_{\infty,\varepsilon}^\diamond \interp w(\vertex), \qquad \forall w \in C(\Oc), \quad \forall \vertex \in \Nhi{\ve}.
\]
We notice, however, that even for a function $w_h \in \Vh$ the computation of $\max_{x \in \Bc_{\ve}(\vertex) } w_h(x)$ and $\min_{x \in \Bc_{\ve}(\vertex) } w_h(x)$ can be complicated in practice. To simplify this calculation we observe that, for $\ve$ sufficiently small,
\[
  \max_{x \in \Bc_{\ve}(\vertex) } w(x) \approx \max \left\{ w(\vertex), \max_{x \in \partial B_{\ve}(\vertex)} w(x)\right\}, \quad \forall w \in C^2(\Bc_\ve(\vertex)),
\]
and similar for the minimum. Indeed, if the maximum occurs at $x \in \interior B_\ve(\vertex)$, then we must have $\De w(x) = \boldsymbol{0}$, so that, as $\ve \downarrow 0$,
\[
  w(\vertex) = w(x) + \De w(x)^\intercal (\vertex-x) + \mathcal{O}(\ve^2) = w(x) + \mathcal{O}(\ve^2).
\]
As a final simplification, we discretize $\partial \Bc_{\ve}(\vertex)$. Let $\mS$ be the unit sphere in $\mRd$. For $\theta \leq 1$ we introduce a discretization $\St \subset \mS$: For any $\bv{v} \in \mS$ there is $\bv{v}_\theta \in \St$ such that
\begin{equation*}
  |\bv{v} - \bv{v}_{\theta} | \leq \theta.
\end{equation*}
We further assume that $\St$ is symmetric, i.e., if $\bv{v}_{\theta} \in \St$, then $-\bv{v}_{\theta} \in \St$ as well. We will then, instead of $\partial B_\ve(\vertex)$, only consider points 
of the form $\vertex + \ve \bv{v}_{\theta}$ for $\bv{v}_{\theta} \in \St$. 

With the previous simplifications at hand set $\frakh = (h,\ve,\theta)$. For $\vertex \in \Nhi{\ve}$ we define
\[
  \mathcal{N}_{\frakh}(\vertex) = \{\vertex\} \cup \{\vertex + \ve \bv{v}_\theta:  \bv{v}_\theta \in \St\},
\]
and, for $w\in C(\Oc)$,
\[
  S_{\frakh}^+ w(\vertex) = \frac{1}{\ve} \l( \max_{x \in \mathcal{N}_{\frakh}(\vertex)} w(x) - w(\vertex) \r),  \quad 
  S_{\frakh}^- w(\vertex) = \frac{1}{\ve} \l( w(\vertex) - \min_{x \in \mathcal{N}_{\frakh}(\vertex)} w(x)  \r).
\]
Finally, our fully discrete operator is defined, for $w \in C(\Oc)$, by
\begin{equation}\label{eq:def-fullydiscrete-op1}
  -\Delta_{\infty, \frakh}^\diamond w(\vertex) = -\frac1\ve \left( S_\frakh^+ \interp w(\vertex) - S_\frakh^- \interp w(\vertex) \right), \qquad \forall \vertex \in \Nhi{\ve}.
\end{equation}

For future reference we define, for $\vertex \in \Nh$, 
\[
  \wt{\mathcal{N}}_\frakh(\vertex) = \left\{ \vertex \right\} \cup \left\{ \othervertex \in \Nh : \exists \bv{v}_\theta \in \St, \widehat{\varphi}_{\othervertex}(\vertex + \ve \bv{v}_\theta) > 0 \right\} \subset \Nh.
\]
This set is such that, for $w_h \in \Vh$, $S_\frakh^\pm w_h(\vertex)$ is uniquely defined by the restriction of $w_h$ to $\wt{\mathcal{N}}_\frakh(\vertex)$.

Under Assumption~\textbf{BC}.\ref{Ass.BC.ext} we will now define our numerical scheme. Our numerical method reads: Find $u_\frakh \in \Vh$ such that
\begin{equation}\label{eq:inf-laplace-discrete}
  -\Delta_{\infty, \frakh}^\diamond u_\frakh (\vertex) = 
  f(\vertex), \quad \forall \vertex \in \Nhi{\ve}, \qquad u_\frakh(\vertex) = \wt{g}_\ve(\vertex), \quad \forall \vertex \in \Nhb{\ve}.
\end{equation}

Of importance in our analysis will be the concepts of discrete sub- and supersolution.

\begin{Definition}[subsolution]\label{def:Subsolution}
The function $w_h \in \Vh$ is a discrete subsolution (supersolution) of \eqref{eq:inf-laplace-discrete} if
\[
  -\Delta_{\infty,\frakh}^\diamond w_h(\vertex) \leq (\geq) f(\vertex), \quad \forall \vertex \in \Nhi{\ve}, \qquad w_h(\vertex) \leq (\geq) \wt{g}_\ve(\vertex), \quad \vertex \in \Nhb{\ve}.
\]
Therefore, a discrete solution of \eqref{eq:inf-laplace-discrete} is both a discrete sub- and supersolution.
\end{Definition}

\section{The discrete problem}\label{sec:ExistenceAndSuch}
We now begin to study the properties of our scheme. The definition of \eqref{eq:def-fullydiscrete-op1} immediately implies the monotonicity of the operator $-\Delta^\diamond_{\infty, \frakh}$ in the following sense.

\begin{Lemma}[monotonicity] \label{lemma:monotonicity}
Let $\vertex \in \Nhi{\ve}$ be an interior node and $w_h, v_h \in \Vh$. If
\begin{equation}
\label{eq:LocEqForMonotonicity}
  (w_h - v_h)(\vertex) \ge (w_h - v_h)(\othervertex), \quad \forall \othervertex \in \wt{\mathcal{N}}_\frakh(\vertex),
\end{equation}
then
\[
S_{\frakh}^+ w_h(\vertex) \le S_{\frakh}^+ v_h(\vertex), \quad 
S_{\frakh}^- w_h(\vertex) \ge S_{\frakh}^- v_h(\vertex),
\]
and
\begin{equation}
\label{eq:DiscInfLapIsMonotone}
  -\Delta^\diamond_{\infty, \frakh} w_h(\vertex) \ge -\Delta^\diamond_{\infty, \frakh} v_h(\vertex).
\end{equation}
\end{Lemma}
\begin{proof}
It simply follows by definition of the operators.

We observe, first of all, that \eqref{eq:LocEqForMonotonicity} implies the same condition over $x \in \mathcal{N}_{\frakh}(\vertex)$. Indeed, let $\psi_h = w_h - v_h \in \Vh$ and $x \in \mathcal{N}_{\frakh}(\vertex)$, so that
\[
  \psi_h(x) = \sum_{\othervertex \in \wt{\mathcal{N}}_\frakh(\vertex)} \psi_h(\othervertex) \widehat{\varphi}_{\othervertex}(x) \leq 
    \sum_{\othervertex \in \wt{\mathcal{N}}_\frakh(\vertex)} \psi_h(\vertex) \widehat{\varphi}_{\othervertex}(x) = \psi_h(\vertex) \sum_{\othervertex \in \wt{\mathcal{N}}_\frakh(\vertex)} \widehat{\varphi}_{\othervertex}(x) = \psi_h(\vertex),
\]
where we used the partition of unity property of the hat basis functions.

Let now $x_0 \in \mathcal{N}_{\frakh}(\vertex)$ satisfy
\[
  w_h(x_0) = \max_{x \in \mathcal{N}_\frakh(\vertex)} w_h(x),
\]
so that
\[
  \ve S_\frakh^+ v_h(\vertex)  = \max_{x \in \mathcal{N}_\frakh(\vertex)} v_h(x) - v_h(\vertex) \geq v_h(x_0) - v_h(\vertex) \geq w_h(x_0) - w_h(\vertex) = \ve S_\frakh^+ w_h(\vertex).
\]
Similarly, $S_\frakh^- v_h(\vertex) \leq S_\frakh^- w_h(\vertex)$. These two relations immediately imply \eqref{eq:DiscInfLapIsMonotone}.
\end{proof}

\subsection{Comparison}\label{sub:comparison}

The monotonicity presented in \Cref{lemma:monotonicity} leads to the following discrete comparison principle.

\begin{Theorem}[comparison I] \label{thm:simple-DCP}
Assume that $w_h,v_h \in \Vh$ satisfy
\begin{equation}\label{eq:assump-thm-discrete-comparison2}
  -\Delta^\diamond_{\infty, \frakh} w_h(\vertex) \le -\Delta^\diamond_{\infty, \frakh} v_h(\vertex) \quad \forall \vertex \in \Nhi{\ve},
\end{equation}
and, moreover, one of the following holds:
\begin{enumerate}[1.]
  \item For every $\vertex \in \Nhi{\ve}$, \eqref{eq:assump-thm-discrete-comparison2} holds with strict inequality or $-\Delta^\diamond_{\infty, \frakh} w_h(\vertex) < 0$.
  \item For every $\vertex \in \Nhi{\ve}$, \eqref{eq:assump-thm-discrete-comparison2} holds with strict inequality or $-\Delta^\diamond_{\infty, \frakh} v_h(\vertex) > 0$.
\end{enumerate}
Then,
\[
  \max_{\vertex \in \Nh} \left[ w_h(\vertex) - v_h(\vertex) \right] = \max_{\vertex \in \Nhb{\ve}} \left[ w_h(\vertex) - v_h(\vertex) \right].
\]
\end{Theorem}	
\begin{proof}
We first prove that the theorem is true if \eqref{eq:assump-thm-discrete-comparison2} holds with strict inequality for every $\vertex \in \Nhi{\ve}$. We assume by contradiction that
\[
  m = \max_{\vertex \in \Nh} \left[ w_h(\vertex) - v_h(\vertex) \right] > \max_{\vertex \in \Nhb{\ve}} \left[ w_h(\vertex) - v_h(\vertex) \right],
\]
and consider the set of points
\[
  E = \{ \vertex \in \Nh: (w_h - v_h)(\vertex) = m \} \subset \Nhi{\ve}.
\]
Clearly, the set $E$ is nonempty because of the definition of $m$. For any $\vertex \in E$, we have
\begin{equation}\label{eq:simple-DCP-proof0}
  w_h(\vertex) - v_h(\vertex) \ge w_h(\othervertex) - v_h(\othervertex)\quad \forall \othervertex \in \Nh. 
\end{equation}
By \Cref{lemma:monotonicity}, this implies that
\[
  -\Delta^\diamond_{\infty, \frakh} w_h(\vertex) \ge -\Delta^\diamond_{\infty, \frakh} v_h(\vertex),
\]
which is a contradiction to our assumption.

Next, consider the case that \eqref{eq:assump-thm-discrete-comparison2} does not hold with strict inequality. By symmetry, we only need to prove the result under the assumption that $-\Delta^\diamond_{\infty, \frakh} w_h(\vertex) < 0$ for all $\vertex \in \Nhi{\ve}$.
Then, for any constant $\beta > 0$ that is small enough, we have
\[
  -\Delta^\diamond_{\infty, \frakh} \l( (1+\beta)w_h \r)(\vertex) < -\Delta^\diamond_{\infty, \frakh} v_h(\vertex) ,
\]
for any $\vertex \in \Nhi{\ve}$. Applying the comparison principle we just proved to the functions $(1+\beta)w_h$ and $v_h$, we obtain that
\[
  \max_{\vertex \in \Nh} \l[ (1+\beta) w_h(\vertex) - v_h(\vertex) \r] = \max_{\vertex \in \Nhb{\ve}} \l[ (1+\beta) w_h(\vertex) - v_h(\vertex) \r].
\]
Letting $\beta \downarrow 0$ finishes our proof.
\end{proof}

The comparison of \Cref{thm:simple-DCP} will be enough to study the case when the right hand side $f$ satisfies Assumption~\textbf{RHS}.\ref{Ass.RHS.sign}. However, when Assumption~\textbf{RHS}.\ref{Ass.RHS.zero} holds, we need a refinement. If this is the case, we require the set of vertices, $\Nh$, to satisfy an additional condition, regarding its connectivity.
\begin{enumerate}[\textbf{M}.1]
  \item \label{Ass.M.connect} For any $\emptyset \neq S \subset \Nhi{\ve}$ there are $\vertex \in S$ and $\othervertex \in \Nh \setminus S$ so that
  \[
    \othervertex \in \wt{\mathcal{N}}_\frakh(\vertex).
  \]
\end{enumerate}
We comment that Assumption~\textbf{M}.\ref{Ass.M.connect} can be satisfied provided $h, \ve$ are sufficiently small compared to $\diam(\Omega)$, and $\ve $ sufficiently big compared to $h$.

With this condition at hand we can now prove another comparison principle. The following result is a discrete counterpart of \cite[Corollary 2.8]{MR2846345}.

\begin{Theorem}[comparison II]\label{thm:DCP}
Let $w_h, v_h \in \Vh$ satisfy
\begin{equation}\label{eq:assump-thm-discrete-comparison}
  -\Delta^\diamond_{\infty, \frakh} w_h(\vertex) \le -\Delta^\diamond_{\infty, \frakh} v_h(\vertex), \quad \forall \vertex \in \Nhi{\ve}.
\end{equation}
Further, assume that $-\Delta^\diamond_{\infty, \frakh} w_h(\vertex) \le 0$ or $-\Delta^\diamond_{\infty, \frakh} v_h(\vertex) \ge 0$ for every $\vertex \in \Nhi{\ve}$. If Assumption~\textbf{M}.\ref{Ass.M.connect} holds, then
\[
  \max_{\vertex \in \Nh} \left[ w_h(\vertex) - v_h(\vertex) \right] = \max_{\vertex \in \Nhb{\ve}}\left[ w_h(\vertex) - v_h(\vertex) \right].
\]
\end{Theorem}
\begin{proof}
The proof proceeds in the same way as \cite[Proof of Theorem 2.7]{MR2846345}. By symmetry, we only need to prove the theorem under the assumption $-\Delta^\diamond_{\infty, \frakh} w_h(\vertex) \le 0$. We assume by contradiction that
\[
  m = \max_{\vertex \in \Nh} \left[ w_h(\vertex) - v_h(\vertex) \right] > \max_{\vertex \in \Nhb{\ve}} \left[ w_h(\vertex) - v_h(\vertex) \right].
\]
Consider the set of points
\[
  E = \{ \vertex \in \Nh: (w_h - v_h)(\vertex) = m \} \subset \Nhi{\ve}.
\]
Clearly, the set $E$ is nonempty because of the definition of $m$. For any $\vertex \in E$, we have
\begin{equation}\label{eq:discrete-comparison-proof0}
  w_h(\vertex) - v_h(\vertex) \ge w_h(\othervertex) - v_h(\othervertex), \quad \forall \othervertex \in \Nh. 
\end{equation}
By \Cref{lemma:monotonicity}, this implies that
\[
  S_{\frakh}^+ w_h(\vertex) \le S_{\frakh}^+ v_h(\vertex), \qquad
  S_{\frakh}^- w_h(\vertex) \ge S_{\frakh}^- v_h(\vertex).
\]
However, from \eqref{eq:assump-thm-discrete-comparison} and the definition of $-\Delta^\diamond_{\infty, \frakh}$, we have
\[
  S_{\frakh}^- w_h(\vertex) - S_{\frakh}^+ w_h(\vertex) \geq S_{\frakh}^- v_h(\vertex) - S_{\frakh}^+ v_h(\vertex).
\]
Therefore we must have
\begin{equation}\label{eq:discrete-comparison-proof1}
  S_\frakh^+ w_h(\vertex) = S_\frakh^+ v_h(\vertex), \; \qquad \;
  S_\frakh^- w_h(\vertex) = S_\frakh^- v_h(\vertex), \qquad \forall \vertex \in E.
\end{equation}
Furthermore, let $l = \max_{\vertex \in E} w_h(\vertex)$ and consider the set
\[
  F = \{ \vertex \in E: w_h(\vertex) = l\}.
\]

We claim that $S_{\frakh}^+ w_h(\vertex) = 0$ for any $\vertex \in F$. We argue by contradiction and suppose that there exists such a point $\vertex \in F$ with $S_{\frakh}^+ w_h(\vertex) > 0$. Then, either
there exists $\othervertex \in \wtNhxi$ with
\begin{equation}\label{eq:discrete-comparison-proof2}
  w_h(\othervertex) - w_h(\vertex) = \ve  S_{\frakh}^+ w_h(\vertex) > 0,
\end{equation}
or there is $\bv{v}_\theta \in \St$ such that
\begin{equation}\label{eq:discrete-comparison-proof3}
w_h(\vertex + \ve \bv{v}_\theta) - w_h(\vertex) = \sum_{\othervertex \in \wtNhxi \cap T} \widehat{\varphi}_{\othervertex} (\vertex + \ve \bv{v}_\theta ) \l( w_h(\othervertex) - w_h(\vertex) \r) = \ve  S_{\frakh}^+ w_h(\vertex) > 0,
\end{equation}
where $T$ is a simplex containing $\vertex + \ve \bv{v}_\theta$. If \eqref{eq:discrete-comparison-proof2} holds, then $w_h(\othervertex) > w_h(\vertex) = l$ and thus $\othervertex \notin E$. This implies that
\[
w_h(\vertex) - v_h(\vertex) > w_h(\othervertex) - v_h(\othervertex),
\]
which leads to
\begin{equation}\label{eq:discrete-comparison-proof-contradict1}
\ve S_{\frakh}^+ w_h(\vertex) = w_h(\othervertex) - w_h(\vertex) < v_h(\othervertex) - v_h(\vertex) \le 	\ve S_{\frakh}^+ v_h(\vertex),
\end{equation}
but this contradicts \eqref{eq:discrete-comparison-proof1}. If \eqref{eq:discrete-comparison-proof3} holds, then choose $\othervertex  \in T \cap \wtNhxi$ with
\[
  \widehat{\varphi}_\othervertex (\vertex + \ve \bv{v}_\theta) \l( w_h(\othervertex) - w_h(\vertex) \r) > 0,
\]
i.e.,
\[
\widehat{\varphi}_\othervertex (\vertex + \ve \bv{v}_\theta)  > 0, \quad w_h(\othervertex) > w_h(\vertex) = l.
\]
So we, once again, deduce that $\othervertex \notin E$ and
\[
  w_h(\vertex) - v_h(\vertex) > w_h(\othervertex) - v_h(\othervertex).
\] 
The previous inequality, combined with \eqref{eq:discrete-comparison-proof0}, gives
\begin{equation}\label{eq:discrete-comparison-proof-contradict2}
\begin{aligned}
  \ve S_{\frakh}^+ w_h(\vertex) &= \sum_{\othervertex \in T \cap \wtNhxi} \widehat{\varphi}_\othervertex (\vertex + \ve \bv{v}_\theta) \l( w_h(\othervertex) - w_h(\vertex) \r) \\
&< \sum_{\othervertex \in T \cap \wtNhxi} \widehat{\varphi}_\othervertex (\vertex + \ve \bv{v}_\theta)  \l( v_h(\othervertex) - v_h(\vertex) \r) \le \ve S_{\frakh}^+ v_h(\vertex),
\end{aligned}
\end{equation}
which contradicts \eqref{eq:discrete-comparison-proof1}.
This proves the claim that $S_{\frakh}^+ w_h(\vertex) = 0$ for any $\vertex \in F$.

Thanks to the assumption that $-\Delta^\diamond_{\infty, \frakh} w_h(\vertex) \le 0$, we have
\[
\ve S_{\frakh}^- w_h(\vertex) \le \ve S_{\frakh}^+ w_h(\vertex).
\]
This leads to
\[
\ve S_{\frakh}^- w_h(\vertex) \le \ve S_{\frakh}^+ w_h(\vertex) = 0, \quad \forall \vertex \in F.
\]
So for any $\vertex \in F$, we have $\ve S_{\frakh}^- w_h(\vertex) = \ve S_{\frakh}^+ w_h(\vertex) = 0$. This implies that
\begin{equation}\label{eq:discrete-comparison-proof4}
  w_h(\vertex + \ve \bv{v}_\theta) = w_h(\vertex), \quad \forall \bv{v}_\theta \in \St.
\end{equation}
We further claim that $\wt{\mathcal{N}}_\frakh(\vertex) \subset E$. In fact, if there exists some $\othervertex \in \wt{\mathcal{N}}_\frakh(\vertex) \setminus E$, we argue in the same way as we did in \eqref{eq:discrete-comparison-proof-contradict1} or \eqref{eq:discrete-comparison-proof-contradict2} to get a contradiction. From $\wt{\mathcal{N}}_\frakh(\vertex) \subset E$, we see that $w_h(\othervertex) \le w_h(\vertex) = l$ for any $\othervertex \in \wt{\mathcal{N}}_\frakh(\vertex)$. Combine this with \eqref{eq:discrete-comparison-proof4} to obtain that
\[
  w_h(\othervertex) = w_h(\vertex) = l, \quad \forall \othervertex \in \wt{\mathcal{N}}_\frakh(\vertex)
\]
and thus $\wt{\mathcal{N}}_\frakh(\vertex) \subset F \subset \Nhi{\ve}$. To summarize, we have shown that if $\vertex \in F$, then we necessarily must have that $\wtNhxi \subset F$.

To conclude we observe that, by Assumption~\textbf{M}.\ref{Ass.M.connect}, we can find $\vertex \in F$ and $\othervertex \in \Nh \setminus F$, for which
\[
\othervertex \in \wt{\mathcal{N}}_\frakh(\vertex),
\]
which is a contradiction.
\end{proof}

\subsection{Existence, uniqueness, and stability}\label{sub:ExistenceUniqueness}

We are now in position to study the existence, uniqueness, and stability of $u_\frakh \in \Vh$, the solution of \eqref{eq:inf-laplace-discrete}. We begin with an \emph{a priori}, that is stability, estimate.

\begin{Lemma}[stability]\label{lemma:stability}
Assume \textbf{RHS}.\ref{Ass.RHS.cont}--\ref{Ass.RHS.f<>0}, and \textbf{BC}.\ref{Ass.BC.cont}--\ref{Ass.BC.ext}. Moreover, if \textbf{RHS}.\ref{Ass.RHS.zero} suppose, in addition, that for all $h>0$, the mesh $\Th$ is such that \textbf{M}.\ref{Ass.M.connect} holds. If $u_\frakh \in \Vh$ solves \eqref{eq:inf-laplace-discrete}, then we have
\[
  \Vert \uve \Vert_{L^{\infty}(\Oh)} \le C \Vert f \Vert_{L^{\infty}(\Oh)} + \max_{\vertex \in \Nhb{\ve}} |\wt{g}_\ve(\vertex)|,
\]
where the constant $C$ depends only on $\Omega$, but is independent of the parameters $\frakh=(h,\ve,\theta)$ that define our scheme.
\end{Lemma}
\begin{proof}	
Without loss of generality we may assume that $0 \in \Omega$. Let us define the barrier function $\varphi(x) = \bv{p}^\intercal x - \frac12 |x|^2 + A$, where $\bv{p}\in \mRd$ and $A \in \mR$ are to be specified. We choose the vector $\bv{p}$ to guarantee that, for any $\vertex \in \Nhi{\ve}$,
\begin{equation}\label{eq:lemma-stab-proof}
  \max_{x \in \mathcal{N}_\frakh(\vertex) } \varphi(x)
  = \max\left\{  \max_{\bv{v}_\theta \in \St} \varphi(\vertex  + \ve \bv{v}_\theta ), \varphi(\vertex) \right\}
  =  \max_{\bv{v}_\theta \in \St} \varphi(\vertex  + \ve \bv{v}_\theta ). 
\end{equation}
To achieve this, it suffices to have
\[
  \max_{\bv{v}_\theta \in \St} \varphi(\vertex  + \ve \bv{v}_\theta ) > \varphi(\vertex).
\]	
Since, for $x \in \mathcal{N}_\frakh(\vertex)$, 
\[
\begin{aligned}
  \varphi(x) & = \bv{p}^\intercal (x - \vertex) + \bv{p}^\intercal \vertex - \frac12 \l( |x-\vertex|^2 + 2(x - \vertex)^\intercal \vertex + |\vertex|^2 \r) + A \\
  & = (\bv{p} - \vertex)^\intercal (x - \vertex) - \frac12 |x - \vertex|^2  + \l( \bv{p}^\intercal \vertex - \frac12 |\vertex|^2 + A\r),
\end{aligned}
\]
we then have
\[
  \varphi(\vertex+ \ve \bv{v}_\theta ) - \varphi(\vertex ) = \ve(\bv{p} - \vertex)^\intercal \bv{v}_\theta - \frac12 \ve^2.
\]
By the definition of $\St$, there is $\bar{\bv{v}}_\theta \in \St$ such that
\[
  \l|\bar{\bv{v}}_\theta - \frac{\bv{p} - \vertex }{|\bv{p}-\vertex |} \r| \le \theta \le 1.
\]
With this choice we then have
\[
\ve(\bv{p} - \vertex)^\intercal \bar{\bv{v}}_\theta - \frac12 \ve^2
\ge \frac{\ve}{2} |\bv{p}-\vertex| - \frac12 \ve^2.
\]
Consequently, if we are able to choose $\bv{p}$ such that $|\bv{p} - \vertex| > \ve$, for all $\vertex \in \Nhi{\ve}$, we will have shown that 
\[
  \max_{\bv{v}_\theta \in \St} \varphi(\vertex + \ve \bv{v}_\theta) > \varphi(\vertex)
\]
To this aim, one could simply choose $\bv{p} \in \mRd$ such that for any $x \in \Omega$
\[
|\bv{p}| \ge 2\diam(\Omega) \ge |x| + \ve,
\]
where we, trivially, assumed that $\ve \le \diam(\Omega)$. Once $\bv{p}$ is chosen, we let $A$ be sufficiently big so that $\varphi \geq 0$ on $\Oc$.

Now, owing to \eqref{eq:lemma-stab-proof}, for any $\vertex \in \Nhi{\ve}$ there exists $\bv{v}_{\theta} \in \St$ such that
\[
  \varphi(\vertex  + \ve \bv{v}_{\theta}) = \max_{x \in \mathcal{N}_{\frakh}(\vertex) } \varphi(x).
\]
Since we assumed that $\St$ is symmetric, then  $-\bv{v}_{\theta} \in \St$, and we have
\[
  \varphi(\vertex - \ve \bv{v}_{\theta}) \ge \min_{x \in \mathcal{N}_{\frakh}(\vertex) } \varphi(x),
\]
and thus
\[
  \begin{aligned}
    &2\varphi(\vertex) - \max_{x \in \mathcal{N}_{\frakh}(\vertex) } \varphi(x) - \min_{x \in \mathcal{N}_{\frakh}(\vertex) } \varphi(x)
    = 2\varphi(\vertex) - \varphi(\vertex + \ve \bv{v}_{\theta}) - \min_{x \in \mathcal{N}_{\frakh}(\vertex) } \varphi(x) \\
    &\ge 2\varphi(\vertex) - \varphi(\vertex + \ve \bv{v}_{\theta}) - \varphi(\vertex - \ve \bv{v}_{\theta})
    = \ve^2.
  \end{aligned}
\]
Combine this with the concavity of $\varphi$ to arrive at
\[
\begin{aligned}
-\Delta_{\infty, \frakh}^\diamond \varphi (\vertex) &= \frac{1}{\ve^2} \l( 2\varphi(\vertex) - \max_{x \in \mathcal{N}_{\frakh}(\vertex) } \interp\varphi(x) - \min_{x \in \mathcal{N}_{\frakh}(\vertex) } \interp \varphi(x) \r) \\
& \ge \frac{1}{\ve^2} \l( 2\varphi(\vertex ) - \max_{x \in \mathcal{N}_{\frakh}(\vertex) } \varphi(x) - \min_{x \in \mathcal{N}_{\frakh}(\vertex) }  \varphi(x) \r)
\ge \frac{1}{\ve^2} \ve^2 = 1.
\end{aligned}
\]
This implies that the function
\[
  w_h = \Vert f \Vert_{L^{\infty}(\Oh)} \interp \varphi + \max_{\vertex \in \Nhb{\ve}} \wt{g}_\ve(\vertex) \in \Vh,
\]
satisfies
\[
  -\Delta_{\infty, \frakh}^\diamond w_h (\vertex) \geq f(\vertex), \quad \forall \vertex \in \Nhi{\ve}, \qquad w_h(\vertex) \geq \wt{g}_\ve(\vertex), \quad \forall \vertex \in \Nhb{\ve},
\]
i.e., it is a discrete supersolution. If \textbf{RHS}.\ref{Ass.RHS.sign} holds then we invoke \Cref{thm:simple-DCP}, otherwise, i.e., when \textbf{RHS}.\ref{Ass.RHS.zero} holds, by \Cref{thm:DCP}, we have
\[
  \max_{\vertex \in \Nh} \l( \uve(\vertex) - w_h(\vertex) \r) = \max_{\vertex  \in \Nhb{\ve}} \l( \uve(\vertex) -w_h(\vertex) \r)  = 
\max_{\vertex  \in \Nhb{\ve}} \l( \wt{g}_\ve(\vertex) -w_h(\vertex) \r) \le 0.
\]
Therefore,
\begin{align*}
\max_{\vertex \in \Nh} \uve(\vertex) &\le \Vert f \Vert_{L^{\infty}(\Oh)} \max_{\vertex \in \Nh} \varphi(\vertex)
+  \max_{\vertex \in \Nhb{\ve}} \wt{g}_\ve(\vertex) \\
&\le C  \Vert f \Vert_{L^{\infty}(\Oh)} + 
\max_{\vertex \in \Nhb{\ve}} \wt{g}_\ve(\vertex) ,
\end{align*}
where the constant $C$ depends only on the domain $\Omega$. This gives an upper bound for $\uve$.
Similarly, we can obtain a lower bound by using $-\Vert f \Vert_{L^{\infty}(\Oh)} \interp \varphi + \min_{\vertex \in \Nhb{\ve}} \wt{g}_\ve(\vertex)$ as a discrete subsolution. Combining the lower and upper bounds together finishes the proof.
\end{proof}

Let us now show that \eqref{eq:inf-laplace-discrete} has a unique solution.

\begin{Lemma}[existence, uniqueness]\label{lemma:ExistUnique}
Assume \textbf{RHS}.\ref{Ass.RHS.cont}--\ref{Ass.RHS.f<>0}, and \textbf{BC}.\ref{Ass.BC.cont}--\ref{Ass.BC.ext}. Moreover, if \textbf{RHS}.\ref{Ass.RHS.zero} holds, suppose that for all $h>0$ the mesh $\Th$ satisfies \textbf{M}.\ref{Ass.M.connect}. For any choice of parameters $\frakh=(h,\ve,\theta)$, there exists a unique $\uve \in \Vh$ that solves \eqref{eq:inf-laplace-discrete}.
\end{Lemma}
\begin{proof}
Uniqueness follows, when \textbf{RHS}.\ref{Ass.RHS.sign} holds, immediately from \Cref{thm:simple-DCP}. Similarly, when \textbf{RHS}.\ref{Ass.RHS.zero} is valid, we invoke \Cref{thm:DCP} to get uniqueness. To show existence we employ a, somewhat standard, discrete version of Perron's method \cite[Section 6.1]{MR2777537}. Consider the set of discrete subsolutions
\[
  W_h = \left\{ w_h \in \Vh : \; -\Delta^\diamond_{\infty, \frakh} w_h(\vertex) \le f(\vertex) \; \forall \vertex \in \Nhi{\ve}; \; w_h(\vertex) \leq \wt{g}_\ve(\vertex) \; \forall \vertex \in \Nhb{\ve} \right\}.
\]
From the proof of \Cref{lemma:stability}, we see that the set $W_h$ is nonempty and bounded from above.
Define $v_h \in \Vh$ as
\[
  v_h(\vertex) = \sup_{w_h \in W_h} w_h(\vertex), \qquad \forall \vertex \in \Nh.
\]
A standard argument based on discrete monotonicity implies that $v_h$ is also a discrete subsolution, i.e., $v_h \in W_h$. In addition, it satisfies
\begin{equation}\label{eq:existence-proof}
-\Delta^\diamond_{\infty, \frakh} v_h(\vertex) = f(\vertex), \quad \forall \vertex \in \Nhi{\ve}, \qquad v_h(\vertex) = \wt{g}_\ve(\vertex), \quad \forall \vertex \in \Nhb{\ve},
\end{equation}
i.e., $v_h$ is a discrete solution. To see this assume, for instance, that the equation is not satisfied at some $\vertex \in \Nhi{\ve}$. Then, we can consider the function
\[
  \wt{v}_h = v_h + \beta \widehat{\varphi}_{\vertex},
\]
with some positive $\beta > 0$. For $\beta > 0$ small enough we have that $\wt{v}_h \in W_h$ since, by monotonicity,
\[
  -\Delta^\diamond_{\infty, \frakh} \wt{v}_h(\othervertex) \le -\Delta^\diamond_{\infty, \frakh} v_h(\othervertex) \le f(\othervertex),
\]
for any $\othervertex \neq \vertex$. This contradicts with our definition of $v_h$ since $\wt{v}_h(\vertex) > v_h(\vertex)$ and $\wt{v}_h \in W_h$. Similarly, if there is $\vertex \in \Nhb{\ve}$ where $v_h(\vertex) < \wt{g}_\ve(\vertex)$ we can again define
\[
  \wt{v}_h = v_h + \beta \widehat{\varphi}_{\vertex},
\]
which for $\beta>0$ but small enough is a subsolution. Once again we reach a contradiction, since $\wt{v}_h \in W_h$ and $\wt{v}_h(\vertex) > v_h(\vertex)$. Consequently, \eqref{eq:existence-proof} holds and $v_h$ is a discrete solution.
\end{proof}

\subsection{Interior consistency}\label{sub:Consistency}

Let us now study the interior consistency of scheme \eqref{eq:inf-laplace-discrete}. We must consider two cases, depending on whether the gradient of our function vanishes or not.

\begin{Lemma}[consistency I]\label{lemma:consistency}
Let $\vertex \in \Nhi{\ve}$ and $\varphi \in C^3(\overline{B}_{\ve}(\vertex))$ be such that $\De\varphi(\vertex) \neq \boldsymbol0$. There exists a constant $C > 0$ depending only on $\varphi$ such that
\[
  \l| -\Delta_{\infty}^\diamond \varphi (\vertex) + \Delta_{\infty, \frakh}^\diamond \varphi (\vertex) \r| \le
  C \l( \ve + \theta + (h/\ve)^2 \r).
\]
\end{Lemma}
\begin{proof}
The proof follows that of \cite[Lemma 4.2]{MR2846345}. It suffices to show
\begin{equation}\label{eq:proof-consist-one-direction}
-\Delta_{\infty}^\diamond \varphi (\vertex) \le -\Delta_{\infty, \frakh}^\diamond \varphi (\vertex) 
+ C \l( \ve + \theta + (h/\ve)^2 \r).
\end{equation}
The other direction can be proved in a similar way. Now, to prove \eqref{eq:proof-consist-one-direction}, let
\[
  \bv{v} =  \frac{\De\varphi(\vertex)}{|\De\varphi(\vertex)|}, \quad
  \bv{v}_{\theta} \in \argmin_{\bv{v}_{\theta} \in \St} |\bv{v} - \bv{v}_{\theta}|, \quad
  \bv{w} \in \argmax_{ \bv{v} \in \St \cup \{\boldsymbol0\} : \vertex + \ve \bv{v} \in \mathcal{N}_{\frakh}(\vertex)} \varphi(\vertex + \ve \bv{v}),
\]
then $|\bv{v}| = |\bv{v}_{\theta}| = 1$ and $|\bv{w}| \leq 1$. By definition of $\St$, we have $|\bv{v}_{\theta} - \bv{v}| \le \theta$ and, in addition,
\[
  0 \le (\bv{v} - \bv{v}_{\theta})^\intercal \bv{v} = \frac12 |\bv{v} - \bv{v}_{\theta}|^2 \le \frac12 \theta^2.
\]
From the choice of $\bv{w}, \bv{v}_{\theta}$, we apply Taylor expansion and obtain
\[
\begin{aligned}
  0 &\le \varphi(\vertex + \ve \bv{w}) -  \varphi(\vertex + \ve \bv{v}_{\theta}) \\
  &\le \ve  \De \varphi(\vertex)^\intercal(\bv{w} - \bv{v}_{\theta}) + C \ve^2 \Vert \De^2 \varphi \Vert_{L^{\infty}(B_{\ve}(\vertex);\mR^{d \times d})}|\bv{w} - \bv{v}_{\theta}| \\
  &= \ve |\De \varphi(\vertex)| (\bv{w} - \bv{v}_{\theta})^\intercal \bv{v} + C \ve^2 \Vert \De^2 \varphi \Vert_{L^{\infty}(B_{\ve}(\vertex);\mR^{d \times d})} |\bv{w} - \bv{v}_{\theta}| \\
  &\le \ve |\De \varphi(\vertex)| (\bv{w} - \bv{v})^\intercal \bv{v} + C \ve^2 \Vert \De^2 \varphi \Vert_{L^{\infty}(B_{\ve}(\vertex);\mR^{d \times d})} |\bv{w} - \bv{v}| \\
  & \quad + \ve |D\varphi(\vertex)| (\bv{v} - \bv{v}_{\theta})^\intercal \bv{v} + C \ve^2 \Vert \De^2 \varphi \Vert_{L^{\infty}(B_{\ve}(\vertex);\mR^{d \times d})}  |\bv{v}_{\theta} - \bv{v}| \\
  &\le \ve |\De\varphi(\vertex)| (\bv{w} - \bv{v})^\intercal \bv{v} + C \ve^2 \Vert \De^2 \varphi \Vert_{L^{\infty}(B_{\ve}(\vertex);\mR^{d \times d})} |\bv{w} - \bv{v}|
  + \ve \theta^2 |\De\varphi(\vertex)| \\ &+ C \ve^2 \theta \Vert \De^2 \varphi \Vert_{L^{\infty}(B_{\ve}(\vertex);\mR^{d \times d})}.
\end{aligned}
\]
Dividing by $\ve|\De\varphi(\vertex)|$ implies
\[
  (\bv{v} - \bv{w})^\intercal \bv{v} \le  |\De\varphi(\vertex)|^{-1} \l( C \ve|\bv{w} - \bv{v}| + \ve\theta \r) \Vert \De^2 \varphi \Vert_{L^{\infty}(B_{\ve}(\vertex);\mR^{d \times d})}
+ \theta^2.
\]
Combining with the fact that, since $|\bv{w}| \leq 1 = |\bv{v}|$,
\[
  (\bv{v} - \bv{w})^\intercal \bv{v}  = \frac12 |\bv{v} - \bv{w}|^2 + \frac12 \l( |\bv{v}|^2 - |\bv{w}|^2 \r) \ge \frac12 |\bv{v} - \bv{w}|^2
\]
we obtain
\[
\frac12 |\bv{v} - \bv{w}|^2 \le |\De\varphi(\vertex)|^{-1} \l( C \ve|\bv{w} - \bv{v}| + \ve\theta \r) \Vert \De^2 \varphi \Vert_{L^{\infty}(B_{\ve}(\vertex);\mR^{d \times d})}
+ \theta^2.
\]
which leads to
\[
  |\bv{v} - \bv{w}| \lesssim \ve |\De \varphi(\vertex)|^{-1} + \theta \l( \Vert \De^2 \varphi \Vert_{L^{\infty}(B_{\ve}(\vertex);\mR^{d \times d})} + 1\r).
\]
Therefore using Taylor expansion again we see that
\[
\begin{aligned}
&\varphi(\vertex + \ve \bv{w}) + \varphi(\vertex - \ve \bv{w}) - \l( \varphi(\vertex + \ve \bv{v}) + \varphi(\vertex - \ve \bv{v}) \r) \lesssim  \ve^2 |\bv{w} - \bv{v}| \Vert \De^2 \varphi \Vert_{L^{\infty}(\Omega;\mR^{d \times d})} \\
&\lesssim \ve^3 |\De \varphi(\vertex)|^{-1}\Vert D^2 \varphi \Vert_{L^{\infty}(\Omega;\mR^{d \times d})}  + \ve^2 \theta \l( \Vert \De^2 \varphi \Vert_{L^{\infty}(\Omega;\mR^{d \times d})}^2 + 1\r).
\end{aligned}
\]	
Combine this with the facts that
\[
\l| -\Delta_{\infty}^\diamond \varphi(\vertex) - \frac{1}{\ve^2} \l( 2\varphi(\vertex) - \varphi(\vertex + \ve \bv{v}) - \varphi(\vertex - \ve \bv{v}) \r) \r| \lesssim \ve \Vert \De^3 \varphi \Vert_{L^{\infty}(B_{\ve}(\vertex);\mR^{d \times d \times d})},
\]
and that, since $\mathcal{N}_{\frakh}(\vertex)$ is symmetric,
\begin{multline*}
2\varphi(\vertex) - \max_{x \in \mathcal{N}_{\frakh}(\vertex)} \varphi(x) - \min_{x \in \mathcal{N}_{\frakh}(\vertex)} \varphi(x) = 2 \varphi(\vertex) - \varphi(\vertex + \ve \bv{w}) - \min_{x \in \mathcal{N}_{\frakh}(\vertex)}  \varphi(x) \\
\ge 2 \varphi(\vertex) - \varphi(\vertex + \ve \bv{w}) - \varphi(\vertex - \ve \bv{w}),
\end{multline*}
we arrive at
\begin{equation}\label{eq:proof-consist}
\begin{aligned}
&-\Delta_{\infty}^\diamond \varphi(\vertex) - \frac{1}{\ve^2} \l( 2\varphi(\vertex) - \!\!\! \max_{x \in \mathcal{N}_{\frakh}(\vertex)} \!\!\! \varphi(x) - \!\!\! \min_{x \in \mathcal{N}_{\frakh}(\vertex)} \!\!\! \varphi(x) \r) \\
&\lesssim \ve \Vert \De^3 \varphi \Vert_{L^{\infty}(B_{\ve}(\vertex);\mR^{d \times d \times d})}
+ \ve |\De\varphi(\vertex)|^{-1} \Vert \De^2 \varphi \Vert_{L^{\infty}(\Omega;\mR^{d \times d})} + \theta \l( \Vert \De^2 \varphi \Vert^2_{L^{\infty}(\Omega;\mR^{d \times d})} + 1\r).
\end{aligned}
\end{equation}
Finally, recalling the definition of $-\Delta^\diamond_{\infty, \frakh}$, we have
\[
\begin{aligned}
\l| \l( 2\varphi(\vertex) - \max_{x \in \mathcal{N}_{\frakh}(\vertex)}  \varphi(x) -  \min_{x \in \mathcal{N}_{\frakh}(\vertex)}  \varphi(x) \r) + \Delta^\diamond_{\infty, \frakh}\varphi(\vertex) \r| & \le 2 \Vert \varphi - \interp \varphi \Vert_{L^{\infty}(B_{\ve}(\vertex))} \\
&\lesssim h^2 \Vert \De^2 \varphi \Vert_{L^{\infty}(B_{\ve}(\vertex);\mR^{d \times d})},
\end{aligned}
\]
which combined with  \eqref{eq:proof-consist} implies the desired result, i.e.,
\[
-\Delta_{\infty}^\diamond \varphi(\vertex) + \Delta^\diamond_{\infty, \frakh} \varphi(\vertex)
\le C \l( \ve + \theta + \frac{h^2}{\ve^2} \r).
\]
The constant $C$ depends on $|\De\varphi(\vertex)|, \Vert \De^2 \varphi \Vert^2_{L^{\infty}(\Omega;\mR^{d \times d})}, \Vert \De^3 \varphi \Vert^2_{L^{\infty}(B_{\ve}(\vertex);\mR^{d \times d \times d})}$.
\end{proof}

\Cref{lemma:consistency} controls the consistency error at an interior vertex $\vertex$ in terms of the discretization parameters provided that $\De \varphi(\vertex) \neq \boldsymbol0$. When $\De \varphi(\vertex) = \boldsymbol0$, the derivation of a consistency error is not standard. For simplicity, we assume that $\varphi$ is a quadratic function and obtain the following estimate between the discrete operator $-\Delta^\diamond_{\infty, \frakh}$ and the lower and upper semicontinuous envelope operators $-\Delta_{\infty}^+, -\Delta_{\infty}^+$.

\begin{Lemma}[consistency II]\label{lemma:consistency-p0}
Let $\vertex \in \Nhi{\ve}$, and $\varphi$ be a quadratic polynomial. If $\varphi$ has no strict local maximum in $B_\ve(\vertex)$, then
\[
  -\Delta^\diamond_{\infty, \frakh} \varphi (\vertex)  \ge
  -\lambda_{\max}(\De^2 \varphi) - C\left( \frac{h^2}{\ve^2} + \theta^2 \right),
\]
where the constant $C$ depends on the dimension $d$, $\De^2 \varphi$ and the shape regularity constant of the mesh $\Th$.
\end{Lemma}
\begin{proof} Without loss of generality, we can assume that $\varphi(x) = \bv{p}^\intercal x + \frac12 x^\intercal \bv{M} x$.
Since $\varphi$ has no strict local maximum in $B_\ve(\vertex)$, we have
\[
\max_{x \in \Bc_{\ve}(\vertex)} \varphi(x) = \varphi(\vertex + \ve \bv{v}),
\]
for some $|\bv{v}| = 1$. Let $\bv{v}_{\theta} \in \St$ be such that $|\bv{v} - \bv{v}_{\theta}| \le \theta$, and 
we claim that
\begin{equation}\label{eq:lemma-consist-p0-proof0}
  \varphi(\vertex + \ve \bv{v}_{\theta}) \ge \varphi(\vertex) -\frac34 \ve^2 \theta^2 \vert \bv{M} \vert. 
\end{equation}
Without loss of generality, we could further assume that $\vertex = 0$. This implies that $\varphi(\vertex) = 0$ and we only need to show that
\[
  \varphi(\ve \bv{v}_{\theta}) \ge -\frac{3\ve^2 \theta^2}4 \vert \bv{M} \vert.
\]
From the choice of $\bv{v}$, we have
\begin{equation}\label{eq:lemma-consist-p0-proof1}
\De \varphi(\ve \bv{v}) = \bv{p} + \ve \bv{Mv} = \lambda \bv{v},
\end{equation}
for some $\lambda \ge 0$. Moreover,
\begin{equation}\label{eq:lemma-consist-p0-proof2}
0 = \varphi(0) \le \varphi(\ve \bv{v}) =
  \bv{p}^\intercal \ve \bv{v} + \frac12 (\ve \bv{v})^\intercal \bv{M} (\ve \bv{v}) 
  = (\lambda \bv{v} - \ve \bv{Mv})^\intercal \ve \bv{v} + \frac{\ve^2}2 \bv{v}^\intercal\bv{Mv}
  = \lambda \ve - \frac{\ve^2}2  \bv{v}^\intercal \bv{Mv}.
\end{equation}
Now write $\varphi(\ve \bv{v}_{\theta})$ as
\[
\begin{aligned}
&\varphi(\ve \bv{v}_{\theta}) = 
\varphi(\ve \bv{v}) + \ve\De\varphi(\ve \bv{v})^\intercal(\bv{v}_{\theta} - \bv{v}) + \frac{\ve^2}{2} (\bv{v}_{\theta} - \bv{v})^\intercal\bv{M} (\bv{v}_{\theta} - \bv{v}) \\
&= \l( \lambda \ve - \frac{\ve^2}2 \bv{v}^\intercal\bv{Mv} \r) + \ve(\bv{p} + \ve \bv{Mv})^\intercal(\bv{v}_{\theta} - \bv{v}) + \frac{\ve^2}{2} (\bv{v}_{\theta} - \bv{v})^\intercal \bv{M} (\bv{v}_{\theta} - \bv{v}) \\
&= \l( \lambda \ve - \frac{\ve^2}2 \bv{v}^\intercal\bv{Mv} \r) + \lambda \ve \bv{v}^\intercal (\bv{v}_{\theta} - \bv{v})
+ \frac{\ve^2}{2} (\bv{v}_{\theta} - \bv{v})^\intercal \bv{M} (\bv{v}_{\theta} - \bv{v}) \\
&= \l( \lambda \ve - \frac{\ve^2}2 \bv{v}^\intercal\bv{Mv} \r) - \frac{\lambda \ve}2 |\bv{v}_{\theta} - \bv{v}|^2
+ \frac{\ve^2}{2} (\bv{v}_{\theta} - \bv{v})^\intercal \bv{M} (\bv{v}_{\theta} - \bv{v}) \\
&\geq \l( \lambda \ve - \frac{\ve^2}2 \bv{v}^\intercal\bv{Mv} \r) - \frac{\lambda \ve \theta^2}2 
- \frac{\ve^2\theta^2}{2} |\bv{M} | \\
&= \left( 1 - \frac12 \theta^2 \right) \l( \lambda \ve - \frac{\ve^2}2 \bv{v}^\intercal\bv{Mv} \r) - \frac{\ve^2 \theta^2}4 \bv{v}^\intercal\bv{Mv}  - \frac{\ve^2 \theta^2}2 \vert \bv{M} \vert  \\
& \ge - \frac{\ve^2 \theta^2}4 \bv{v}^\intercal\bv{Mv}  - \frac{\ve^2 \theta^2}2 \vert \bv{M} \vert
\geq - \frac{3\ve^2 \theta^2}4 \vert \bv{M} \vert,
\end{aligned}
\]
where we used \eqref{eq:lemma-consist-p0-proof1}, \eqref{eq:lemma-consist-p0-proof2}, and the fact that, since $\theta \le 1$, we have $1 - \theta^2/2 \ge 0$.

Now, choose $\bv{w}_{\theta}$ that satisfies
\[
  \max_{\bv{w} \in \St} \varphi(\vertex + \ve \bv{w}) = \varphi(\vertex + \ve \bv{w}_{\theta}).
\]
Thanks to \eqref{eq:lemma-consist-p0-proof0}, we have
\[
  \varphi(\vertex + \ve \bv{w}_{\theta}) \ge \max_{x \in \mathcal{N}_{\frakh}(\vertex)} \varphi(x) -\frac{3 \ve^2 \theta^2}4 \vert \bv{M} \vert.
\]
Therefore, since $\mathcal{N}_\frakh(\vertex)$ is symmetric,
\[
\begin{aligned}
-\Delta_{\infty, \frakh}^\diamond \varphi(\vertex) 
&= \frac{1}{\ve^2} \l( 2\varphi(\vertex) - \min_{x \in \mathcal{N}_\frakh(\vertex)} \interp \varphi(x) - \max_{x \in \mathcal{N}_\frakh(\vertex)} \interp \varphi(x)   \r) \\
&\ge \frac{1}{\ve^2} \l( 2\varphi(\vertex) - \min_{x \in \mathcal{N}_\frakh(\vertex)} \varphi(x) - \max_{x \in \mathcal{N}_\frakh(\vertex)} \varphi(x) - Ch^2  \r) \\
&\ge \frac{1}{\ve^2} \l( 2\varphi(\vertex) - \varphi(\vertex - \ve \bv{w}_{\theta}) - \varphi(\vertex + \ve \bv{w}_{\theta}) - C h^2 -  C\ve^2 \theta^2 \r) \\
& = -\bv{w}_{\theta}^\intercal\bv{M w}_{\theta} - Ch^2 \ve^{-2} - C \theta^2 \ge  -\lambda_{\max}(\bv{M}) - Ch^2 \ve^{-2} - C \theta^2.
\end{aligned}
\]
This finishes the proof of the lemma.
\end{proof}

Let $\{ \frakh_j = (h_j, \ve_j,\theta_j) \}_{j=1}^\infty$ be such that $\frakh_j \to (0,0,0)$ as $j \uparrow \infty$. Let $w_j \in \mathbb{V}_{h_j}$. We define the upper and lower semicontinuous envelopes to be
\begin{equation}
  \label{eq:def-uo-uu}
  \begin{aligned}
    \overline{w}(x) &= \limsup_{j \uparrow \infty, \mathcal{N}_{h_j} \ni \mathtt{z}_{h_j} \to x} w_j(\mathtt{z}_{h_j}), \\
    \underline{w}(x) &= \liminf_{j \uparrow \infty, \mathcal{N}_{h_j} \ni \mathtt{z}_{h_j} \to x} w_j(\mathtt{z}_{h_j}).
  \end{aligned}
\end{equation}
By construction $\overline{w} \in \USC(\Oc)$ and $\underline{w} \in \LSC(\Oc)$.

With the aid of interior consistency we can show that the lower and upper semicontinuous envelopes of a sequence of discrete solutions to \eqref{eq:inf-laplace-discrete} are a sub- and supersolution, respectively. 

\begin{Lemma}[envelopes]\label{lemma:consistency-combined}
Assume the right hand side $f$ satisfies Assumptions \textbf{RHS}.\ref{Ass.RHS.cont}--\ref{Ass.RHS.f<>0}. Assume, in addition, that if \textbf{RHS}.\ref{Ass.RHS.zero} holds, then then for all $h>0$ the mesh $\Th$ satisfies Assumption~\textbf{M}.\ref{Ass.M.connect}. Let $\{ \frakh_j = (h_j, \ve_j,\theta_j) \}_{j=1}^\infty$ be a sequence of discretization parameters such that $\frakh_j \to (0,0,0)$ as $j \uparrow \infty $ and, in addition,
\begin{equation}
\label{eq:RelationsForConvergence}
  \frac{h_j}{\ve_j} \to 0.
\end{equation}
Let $u_{\frakh_j} \in \mathbb{V}_{h_j}$ be the sequence of discrete solutions to \eqref{eq:inf-laplace-discrete}. Then, the function $\overline{u}$, defined in \eqref{eq:def-uo-uu}, is a subsolution, i.e., in the viscosity sense we have
\begin{equation}\label{eq:thm-convergence-subsol}
  -\Delta_{\infty}^+ \overline{u} \le f \quad \text{ in } \Omega.
\end{equation}
Similarly, $\underline{u}$, defined in \eqref{eq:def-uo-uu}, is a supersolution, i.e., in the viscosity sense we have
\[
  -\Delta_{\infty}^- \underline{u} \ge f \quad \text{ in } \Omega.
\]
\end{Lemma}
\begin{proof}
By symmetry, it suffices to show \eqref{eq:thm-convergence-subsol}, as the proof of the fact that $\underline{u}$ is a supersolution follows in the same way.

We need to prove that for a quadratic polynomial $\varphi$ such that $\overline{u} - \varphi$ attains a strict local maximum at $x_0 \in \Omega$, then
\[
-\Delta_{\infty}^+ \varphi(x_0) \le f(x_0).
\]
Without loss of generality, we may assume that $x_0 = 0$ and $\overline{u}(0) = \varphi(0) = 0$. Therefore, $\varphi(x) = \bv{p}^\intercal x + \frac12 x^\intercal \bv{M}x$ where $\bv{p} = \De\varphi(0)$ and $\bv{M} = \De^2 \varphi(0)$. Since the maximum is strict, there is $\delta_0 > 0$ such that $\overline{u}(0) - \varphi(0) > \overline{u}(x) - \varphi(x)$ for any $x \in \Bc_{\delta_0}(0) \setminus \{0\}$. 
By the definition of $\overline{u}$, for any $0<\delta\leq \delta_0$ and large enough $j$ we can select $\mathtt{z}_{h_j} \in \mathcal{N}_{h_j,\ve_j}^I \cap \Bc_{\delta}(0)$ such that
\[
  u_{\frakh_j}(\mathtt{z}_{h_j}) - \mathcal{I}_{h_j}\varphi(\mathtt{z}_{h_j}) = \max_{\mathcal{N}_{h_j} \cap \Bc_{\delta}(0)} (u_{\frakh_j} - \mathcal{I}_{h_j}\varphi) ,
\]
and $\mathtt{z}_{h_j} \to 0$ as $j \uparrow \infty$. Using the monotonicity of $-\Delta_{\infty, \frakh_j}^\diamond$, presented in \Cref{lemma:monotonicity}, we have
\begin{equation}\label{eq:proof-lem-convergence}
-\Delta^\diamond_{\infty, \frakh_j} \varphi(\mathtt{z}_{h_j}) \le -\Delta^\diamond_{\infty, \frakh_j} u_{\frakh_j}(\mathtt{z}_{h_j}) \le f(\mathtt{z}_{h_j}).
\end{equation}
To conclude the proof, we will discuss three cases separately. 

\emph{Case 1:} $\bv{p} \neq \boldsymbol0$. By \Cref{lemma:consistency} and the upper semicontinuity of $\Delta_{\infty}^+$, we obtain that
\[
\begin{aligned}
-\Delta_{\infty}^+ \varphi(0) &\leq \lim_{j \uparrow \infty} -\Delta_{\infty}^+ \varphi(\mathtt{z}_{h_j}) \le \limsup_{j \uparrow \infty} \l( -\Delta^\diamond_{\infty, \frakh_j} \varphi (\mathtt{z}_{h_j}) + C \l( \ve_j + \theta_j + (h_j/\ve_j)^2 \r) \r) \\
&\le \limsup_{j \uparrow \infty} f(\mathtt{z}_{h_j}) = f(0),
\end{aligned} 
\]
because of \eqref{eq:RelationsForConvergence}. We remark that by a standard perturbation argument, we also have $-\Delta_{\infty}^+ \varphi(0) \le f(0)$ if $\overline{u} - \varphi$ attains a non-strict local maximum at $0 \in \Omega$.

\emph{Case 2:} $\bv{p} = \boldsymbol0$, and $\bv{M}$ is not strictly negative definite. The assumptions imply that $\varphi$ has no strict local maximum in $B_\ve(\mathtt{z}_{h_j})$. We apply \Cref{lemma:consistency-p0} and obtain
\[
\begin{aligned}
-\Delta_{\infty}^+ \varphi(0) &= -\lambda_{\max}(\bv{M}) \le \limsup_{j \uparrow \infty} \l( -\Delta^\diamond_{\infty, \frakh_j} \varphi (\mathtt{z}_{h_j}) + C \l( \theta_j^2 + (h_j/\ve_j)^2 \r) \r) \\
&\le \limsup_{j \uparrow \infty} f(\mathtt{z}_{h_j}) = f(0)
\end{aligned} 
\]
because of \eqref{eq:RelationsForConvergence}.

\emph{Case 3:} $\bv{p} = \boldsymbol0$ and $\bv{M}$ is strictly negative definite. Let $0 \neq y \in \mRd$ and define $\varphi_y(x) = \tfrac12 (x-y)^\intercal \bv{M} (x-y)$. For $|y|$ sufficiently small, we know that $\overline{u} - \varphi_y$ attains a local maximum at $x_y \in B_\delta(0)$. Since $\bv{M}$ is strictly negative definite, we have $\varphi(y) = \varphi_y(0) < 0 =\varphi(0) = \varphi_y(y)$ and
\[
\overline{u}(x_y) - \varphi_y(x_y) \geq
\overline{u}(0) - \varphi_y(0)
> \overline{u}(0) - \varphi(0) 
> \overline{u}(y) - \varphi(y)
> \overline{u}(y) - \varphi_y(y).
\]

This shows that $x_y \neq y$. Therefore, $\De \varphi_y( x_y) \neq \boldsymbol0$. Consequently, by using the results we have obtained for Case 1, we have
\[
-\lambda_{\max}(\bv{M}) \le -\Delta_{\infty}^+ \varphi(x_y) \le f(x_y).
\]
Letting $y \to 0$, we have $x_y \to 0$ and 
\[
-\Delta_{\infty}^+ \varphi(0) = -\lambda_{\max}(\bv{M}) \le \lim_{y \to 0} f(x_y) = f(0),
\]
because of $f \in C(\Oc)$. 
\end{proof}

\begin{Remark}[consistency]
Notice that, in the proof of \Cref{lemma:consistency-combined}, the case when $\bv{p}=\boldsymbol0$ and $\bv{M}$ is strictly negative definite required a nonstandard proof. This is due to the rather surprising fact that, in this case our scheme is not consistent.
To see this it suffices to consider the semidiscrete scheme \eqref{eq:discInfLap}. Further discretization, i.e., using \eqref{eq:def-fullydiscrete-op1}, will only add consistency terms depending on $\frakh$. Let $\varphi(x) = -\tfrac12|x|^2$ and  assume $x=0$ is an interior point, so that $B_\ve(x) \subset \Omega$. We have
\[
  \varphi(0) = 0, \qquad \max_{x \in \Bc_\ve(0) } \varphi(x) = 0, \qquad \min_{x \in \Bc_\ve(0) } \varphi(x) = -\frac12 \max_{x \in \Bc_\ve(0) } |x|^2 = -\frac12 \ve^2
\]
so that
\[
  -\Delta_{\infty,\varepsilon}^\diamond \varphi(0) = \frac12 < -\Delta_\infty^+ \varphi(0) = 1,
\]
independently of the value of $\ve$. The strategy that we follow, and that greatly simplifies that of \cite[Theorem 2.11]{MR2846345}, is to move away from such a point to one where the consistency issue does not arise.
\end{Remark}

\subsection{Boundary behavior and convergence}\label{sub:convergence}

To prove convergence of discrete solutions we need, in addition to interior consistency, to control the behavior of discrete solutions $\uve$ near $\partial\Omega$. As before, we will assume that we have a sequence of discretization parameters $\{\frakh_j = (h_j, \ve_j, \theta_j) \}_{j=1}^{\infty}$ satisfying that
\begin{equation}\label{eq:para-condition}
  \lim_{j \uparrow \infty} \left[ h_j + \ve_j + \theta_j + \frac{h_j}{\ve_j} + \frac{\theta_j}{\ve_j} \right] = 0.
\end{equation}
Let us prove that $\overline{u}$ and $\underline{u}$, defined in \eqref{eq:def-uo-uu},  coincide with $g$ at the boundary $\partial \Omega$.

\begin{Lemma}[boundary behavior]\label{lemma:bdry}
Let all the conditions of \Cref{lemma:consistency-combined} be fulfilled. Assume, in addition, that the boundary datum $g$ satisfies Assumptions \textbf{BC}.\ref{Ass.BC.cont}--\ref{Ass.BC.ext}. Let $\{\frakh_j\}_{j=1}^\infty$ be a sequence of discretization parameters that satisfies \eqref{eq:para-condition}. The functions $\overline{u}, \underline{u}$, defined as in \eqref{eq:def-uo-uu} through the family $u_{\frakh_j} \in \mathbb{V}_{h_j}$ of discrete solutions, satisfy $\overline{u}(x) = \underline{u}(x) = g(x)$ for all $x \in \partial \Omega$.
\end{Lemma}
\begin{proof} By symmetry it suffices to prove that $\overline{u}(x) = g(x)$. From the fact that $u_{\frakh_j}(\mathtt{z}_{h_j}) = \wt{g}_\ve(\mathtt{z}_{h_j})$ for all $\mathtt{z}_{h_j} \in \mathcal{N}_{h_j,\ve_j}^b$ we have
\[
  \overline{u}(x) \ge \lim_{\ve_j \to 0} \wt{g}_{\ve_j}(x) = g(x), \quad \forall x \in \partial \Omega.
\]
So, it only remains to prove that $\overline{u}(x) \le g(x)$ for any $x \in \partial \Omega$. 

First, we claim that it is sufficient to prove this for a smooth $\wt{g}_{\ve_j}$. To see this we employ the fact that, for any $\delta>0$, we have a function $g_{s,\delta} \in C^{\infty}(\Oc)$ such that $\|\wt{g}_{\ve_j} - g_{s,\delta} \|_{L^\infty(\Omega)} \le \delta$. Let the functions $\overline{u}_s$ be associated with the data $(f,g_{s,\delta})$. If we know that
$\overline{u}_s(x) \le g_{s,\delta}(x)$ for $x \in \pO$, then 
\[
\overline{u}(x) \le \overline{u}_s(x) + \delta \le g_{s,\delta}(x) + \ve \le \wt{g}_{\ve_j}(x) + 2\delta,
\]
where the first inequality follows from $\wt{g}_{\ve_j} \le g_{s,\delta} + \delta$ and the comparison principle of \Cref{thm:simple-DCP} or \Cref{thm:DCP}, depending on how \textbf{RHS}.\ref{Ass.RHS.f<>0} is satisfied. Since $\delta$ is arbitrary, this implies that $\overline{u}(x) \le g(x)$ for any $x \in \partial \Omega$.

Thus, assuming that $g$ is smooth, we prove $\overline{u}(x_0) \le g(x_0)$ for any $x_0 \in \pO$ satisfying the exterior ball condition, i.e., there are $y_0 \in \mRd$ and $0 < R < 1$ such that
\[
x_0 \in \partial B_R(y_0), \quad \interior B_R(y_0) \cap \Omega = \emptyset.
\]
We may write $y_0 = x_0 + R \bv{w}$ for some $|\bv{w}| = 1$. The exterior ball condition implies that
\begin{equation}\label{eq:bdry-behavior-proof0}
  |(x-x_0) - R\bv{w}| \ge R \quad \implies \quad \bv{w}^\intercal (x - x_0) \le \frac{|x - x_0|^2}{2R},
\end{equation}
for any $x \in \Omega$. Let $z_0 = x_0 + (R/2) \bv{w}$ and consider the following barrier function
\[
\varphi(x) = q\l( |x - z_0|\r) = a + b|x - z_0| - \frac{c}{2} |x - z_0|^2,
\]
where $q$ is a quadratic function to be determined. Fix $c = \sup_{x \in \Omega} f(x) + 1$ and assume that $b > c$, we claim that for $j$ large enough,
\[
-\Delta_{\infty, \frakh_j}^\diamond \varphi (\mathtt{z}_{h_j})  \ge \sup_{x \in \Omega} f(x) \ge f(\mathtt{z}_{h_j}),
\]
for all $\mathtt{z}_{h_j} \in \mathcal{N}^I_{h_j,\ve_j}$. To see this, notice that a simple calculation leads to
\begin{equation}\label{eq:bdry-behavior-proof1}
-\Delta_{\infty}^\diamond \varphi (\mathtt{z}_{h_j}) = c.
\end{equation}
Since $\varphi$ is smooth away from $z_0$, we obtain from \Cref{lemma:consistency} that
\[
\l| \Delta_{\infty}^\diamond \varphi (\mathtt{z}_{h_j}) - \Delta_{\infty, \frakh_j}^\diamond \varphi (\mathtt{z}_{h_j}) \r| \le
C_R \l( \ve_j + \theta_j + (h_j/\ve_j)^2 \r),
\]
for any $\mathtt{z}_{h_j} \in \mathcal{N}^I_{h_j,\ve_j}$ where $C_R$ is a constant depending on $R$. Combining this with \eqref{eq:bdry-behavior-proof1} and using the assumption in \eqref{eq:para-condition} we see that
\[
-\Delta_{\infty, \frakh_j}^\diamond \varphi (\mathtt{z}_{h_j}) \ge -\Delta_{\infty}^\diamond \varphi (\mathtt{z}_{h_j}) - C_R \l( \ve_j + \theta_j + (h_j/\ve_j)^2 \r) \ge c - 1 \ge \sup_{x \in \Omega} f(x),
\]
for $j$ large enough and any $\mathtt{z}_{h_j} \in \mathcal{N}^I_{h_j,\ve_j}$. Now we choose the constants $a$ and $b$ to guarantee that
\begin{equation}\label{eq:bdry-behavior-proof2}
\varphi(x) \ge \wt{g}_{\ve_j}(x),
\end{equation}
for any $x \in \Omega$. Since $\wt{g}_{\ve_j}$ is smooth, we let $L$ be the Lipschitz constant of $\wt{g}_{\ve_j}$. It then suffices to require that
\begin{equation}\label{eq:bdry-behavior-proof3}
  \varphi(x) \ge \wt{g}_{\ve_j}(x_0) + L|x - x_0| \ge \wt{g}_{\ve_j}(x).
\end{equation}
Let, for some $b_0 > 0$, 
\[
b = b_0 + c \l( \diam(\Omega) + \frac{R}{2} \r).
\]
Then the function $q(t) = a + bt - (c/2)t^2$ satisfies
\[
b \ge q'(t) = b - ct \ge b_0,
\]
for $t \in [0, \diam(\Omega) + \frac{R}{2}]$. Recall that by \eqref{eq:bdry-behavior-proof0} we have for any $x \in \Omega$
\[
\begin{aligned}
|x - z_0|^2 &= |(x_0 - z_0) + (x - x_0)|^2 = |x_0 - z_0|^2 + |x - x_0|^2 + 2 (x_0 - z_0)^\intercal (x - x_0) \\ 
&= \frac{R^2}{4} + |x - x_0|^2 - R\bv{w}^\intercal (x - x_0)
\ge \frac{R^2}{4} + |x - x_0|^2 - \frac{|x - x_0|^2}{2}
= \frac{R^2}{4} + \frac{|x - x_0|^2}{2}.
\end{aligned}
\]
This implies that
\[
\begin{aligned}
\varphi(x) - \varphi(x_0) = q(|x - z_0|) - q(|x_0 - z_0|)
\ge b_0 \l(\sqrt{\frac{R^2}{4} + \frac{|x - x_0|^2}{2}} - \frac{R}{2} \r).
\end{aligned}
\]
We choose the constant $a$ such that $\varphi(x_0) = \wt{g}_{\ve_j}(x_0) + L \beta$
for some $\beta > 0$. From the inequality above, a sufficient condition for \eqref{eq:bdry-behavior-proof3} to be accomplished is
\begin{equation}\label{eq:bdry-behavior-proof4}
L \beta + b_0 \l(\sqrt{\frac{R^2}{4} + \frac{|x - x_0|^2}{2}} - \frac{R}{2} \r)
\ge L |x - x_0|,
\end{equation}
which clearly holds for any $x$ with $|x - x_0| \le \beta$. In addition, it is enough to have
\[
b_0 \sqrt{\frac{R^2}{4} + \frac{|x - x_0|^2}{2}} \ge L |x - x_0| + b_0 \frac{R}{2}.
\]
Simple calculations reveal that it suffices to require
\[
b_0 \ge \max\l\{ 2L, \; \frac{4LR}{|x-x_0|} \r\}.
\]
Therefore, if we have
\begin{equation}\label{eq:bdry-behavior-proof5}
b_0 \ge \max\l\{ 2L, \; \frac{4L}{\beta} \r\},
\end{equation}
then \eqref{eq:bdry-behavior-proof4} and \eqref{eq:bdry-behavior-proof2} mush hold because of $0 < R < 1$.

Once \eqref{eq:bdry-behavior-proof2} is true, by the comparison principle of \Cref{thm:simple-DCP} or \Cref{thm:DCP} we see that $u_{\frakh_j}(\mathtt{z}_{h_j}) \le \varphi(\mathtt{z}_{h_j})$ for any $\mathtt{z}_{h_j} \in \mathcal{N}_{h_j}$ and thus
\begin{equation*}
\overline{u}(x) = \limsup_{j \uparrow \infty, \Nhj \ni \mathtt{z}_{h_j} \to x} u_{\frakh_j}(\mathtt{z}_{h_j}) \le \limsup_{j \uparrow \infty, \Nhj \ni \mathtt{z}_{h_j} \to x} \varphi(\mathtt{z}_{h_j}) \le \varphi(x), \quad \forall x \in B_\beta(x_0).
\end{equation*}
This simply shows that
\[
\overline{u}(x_0) \le 
g(x_0) + L \beta.
\]
Since $\beta > 0$ can be chosen arbitrarily small with $b_0$ chosen afterwards, we have $\overline{u}(x_0) \le g(x_0)$ for any $x_0 \in \pO$ satisfying the exterior ball condition.

Finally, let us consider the point $x \in \pO$ where the exterior ball condition may not hold. Noticing that from the previous discussions, if the exterior ball condition is satisfied at $x_0 \in \pO$ $x \in B_\beta(x_0)\setminus \{x_0\}$, then
\[
\begin{aligned}
\overline{u}(x) &\le \varphi(x) \le \varphi(x_0) + b|x - x_0|
= g(x_0) + L \beta + b|x - x_0| \\
&\le g(x) + L|x - x_0| + L \beta + b |x - x_0|.
\end{aligned}
\]
For any $\delta > 0$, we may choose $\beta$ such that $L \beta < \delta/2$. Then we choose $b_0$ satisfying \eqref{eq:bdry-behavior-proof5} and obtain the parameter $b$. Since $\pO$ is assumed continuous and compact, there always exist an $x_0 \in \pO$ satisfying the exterior ball condition with $|x - x_0|<\tfrac\delta{2(L+b)}$. Consequently, 
\[
\overline{u}(x) \le g(x) + \delta,
\]
for any $\delta > 0$. This proves that $\overline{u}(x) \le g(x)$ for any $x \in \pO$ and finishes the proof.
\end{proof}

Gathering all our previous results we can assert convergence of our numerical scheme.

\begin{Theorem}[convergence]\label{thm:convergence}
Assume that the right hand side $f$ satisfies Assumptions \textbf{RHS}.\ref{Ass.RHS.cont}--\ref{Ass.RHS.f<>0}. Assume that the boundary datum $g$ satisfies Assumptions \textbf{BC}.\ref{Ass.BC.cont}--\ref{Ass.BC.ext}. Let $\{\frakh_j\}_{j=1}^\infty$ be a sequence of discretization parameters that satisfies \eqref{eq:para-condition}; and, if \textbf{RHS}.\ref{Ass.RHS.zero} holds, it additionally satisfies \textbf{M}.\ref{Ass.M.connect} for all $j \in \mathbb{N}$. In this framework we have that the sequence $\{u_{\frakh_j} \in \mathbb{V}_{h_j}\}_{j=1}^\infty$ of solutions to \eqref{eq:inf-laplace-discrete} converges uniformly, as $j \uparrow \infty$, to $u$, the solution of \eqref{eq:BVPNinfLaplace}.
\end{Theorem}
\begin{proof}
The proof follows under the framework of \cite{MR1115933}. To be specific, since our approximation scheme \eqref{eq:inf-laplace-discrete} is monotone (\Cref{lemma:monotonicity}), stable (\Cref{lemma:stability}) and consistent (\Cref{lemma:consistency} and \Cref{lemma:consistency-p0}), the argument in \cite{MR1115933} implies that $\overline{u}$ and $\underline{u}$ are a subsolution and supersolution to the continuous problem \eqref{eq:BVPNinfLaplace}, respectively; see \Cref{lemma:consistency-combined}.

Notice, in addition, that \Cref{lemma:bdry} implies that the Dirichlet boundary condition is attained in the classical sense for $\overline{u}$ and $\underline{u}$. Recall that \eqref{eq:BVPNinfLaplace} admits a comparison principle for Dirichlet boundary conditions in the classical sense: see \cite[Theorem 3.3]{MR2475319}, \cite{MR2592977}, and the discussions after \cite[Theorem 2.18]{MR2846345}, we thus have $\overline{u} \le \underline{u}$. This yields, since $\overline{u} \ge \underline{u}$ by definition, that $\overline{u} = \underline{u}$ and the uniform convergence of $u_{\frakh_j}$ to the solution $u$.
\end{proof}

\section{Rates of convergence}\label{sec:Rates}

In this section, we prove convergence rates for solutions of \eqref{eq:inf-laplace-discrete}. To be able to do this, additional regularity on the solution of \eqref{eq:BVPNinfLaplace} must be required. However, we recall that, as we mentioned in Section~\ref{sec:Intro}, the regularity theory for \eqref{eq:BVPNinfLaplace} is far from complete. Nevertheless, we have that, if there is $\alpha \in (0,1]$ for which $f \in C^{0,\alpha}(\Omega)$ and $g \in C^{0,\alpha}(\pO)$, then $u \in C^{0,\alpha}(\Oc)$; see the references mentioned in the Introduction, and \cite[Proposition 6.4]{MR2846345} for the case \textbf{RHS}.\ref{Ass.RHS.sign}, and Propositions 3.11 and 6.7 of \cite{MR2846345} for \textbf{RHS}.\ref{Ass.RHS.zero}.

In this section we will tacitly assume that \textbf{RHS}.\ref{Ass.RHS.cont}--\ref{Ass.RHS.f<>0}, \textbf{BC}.\ref{Ass.BC.cont}--\ref{Ass.BC.ext} are valid. We will also assume that, if \textbf{RHS}.\ref{Ass.RHS.zero}, then the meshes satisfy \textbf{M}.\ref{Ass.M.connect}. As a final assumption, we posit that there is $\alpha \in (0,1]$ such that $f \in C^{0,\alpha}(\Oc)$, $g \in C^{0,\alpha}(\pO)$, and $u \in C^{0,\alpha}(\Oc)$. 

We immediately note a consequence of this assumed regularity, which will be used repeatedly. For any $\vertex \in \Nhb{\ve}$, by our definition of $\Nhb{\ve}$, there exists 
$x \in \pO$ satisfying $|x - \vertex| \le 2\ve$. Then from the regularity assumptions on $g$ and $u$ we get, for any $z \in \overline{\Omega}$,
\begin{equation*}
  \begin{aligned}
    | \wt{g}_{\ve}(\vertex) - u(z) | &\le | \wt{g}_{\ve}(\vertex) - \wt{g}_{\ve}(x) | + 
    |\wt{g}_{\ve}(x) - g(x)| + |u(x) - u(z)| \\
    &\le (2\ve)^{\alpha} |\wt{g}_{\ve}|_{C^{0,\alpha}(\Oc)} + \Vert \wt{g}_{\ve} - g \Vert_{L^{\infty}(\Omega)} + |x - z|^{\alpha} |u|_{C^{0,\alpha}(\Oc)}.
  \end{aligned}
\end{equation*}
If, in addition, $|\vertex - z| \lesssim \ve$, then $|x - z| \lesssim \ve$ and the inequality above implies that
\begin{equation}\label{eq:bdry-gve-u}
\begin{aligned}
| \wt{g}_{\ve}(\vertex) - u(z) | \lesssim \ve^{\alpha} \l( |\wt{g}_{\ve}|_{C^{0,\alpha}(\Oc)} + |u|_{C^{0,\alpha}(\Oc)} \r)  + \Vert \wt{g}_{\ve} - g \Vert_{L^{\infty}(\Omega)} \lesssim \ve^{\alpha},
\end{aligned}
\end{equation}
where we used \textbf{BC}.\ref{Ass.BC.ext}. This shows that the boundary condition we enforce for the discrete problem induces an error of no more than $\mathcal{O}(\ve^{\alpha})$.

Let us, for convenience, provide here some notation. Given a function $w: \Omega \to \mR$, we define its local Lipschitz constant at $x \in \Omega$ as
\begin{equation}\label{eq:local-Lip}
  L(w, x) = \lim_{r \downarrow 0} \Lip(w, B_r(x)), \qquad \Lip(w,E) = \sup_{x,y \in E:  x \ne y} \frac{ |w(x)- w(y)|}{|x-y|}.
\end{equation}
We also define the operators $S_{\ve}^+,\  S_{\ve}^-$ as
\begin{equation}\label{eq:def-semi-S}
  S_{\ve}^+ w(x) = \frac1\ve \left( \max_{x' \in \overline{B}_\ve(x)} w(x') - w(x) \right), \quad
  S_{\ve}^- w(x) = \frac1\ve \left( w(x) - \min_{x' \in \overline{B}_\ve(x)} w(x') \right).
\end{equation}

We are now ready to prove rates of convergence. We split our discussions depending on how \textbf{RHS}.\ref{Ass.RHS.f<>0} is fulfilled. 

\subsection{Convergence rates for the inhomogeneous problem} \label{sub:Ratesfsign}

\begin{Theorem}[error estimate for the inhomogeneous problem]\label{Thm:inhomo-error}
Let $u$ be the viscosity solution of \eqref{eq:BVPNinfLaplace} and $\uve$ be the solution of \eqref{eq:inf-laplace-discrete}. Under our running assumptions suppose that \textbf{RHS}.\ref{Ass.RHS.sign} holds, and  that $\frakh$ is sufficiently small, and such that
\[
  \frac{ (\ve\theta)^\alpha + h^\alpha}{\ve^2},
\]
is small enough. Then
\[
  \Vert u-\uve \Vert_{L^{\infty}(\Omega_h)} \lesssim \ve^{\alpha} +
\frac{\theta^{\alpha}}{\ve^{2-\alpha}} + \frac{h^{\alpha}}{\ve^2},
\]
where the implied constant depends on the dimension $d$, the domain $\Omega$, the shape regularity of the mesh $\Th$, $\min_{x \in \Oc} |f(x)|$, and the $C^{0,\alpha}$ norms of the data $f$, $g$ and the solution $u$.
\end{Theorem}
\begin{proof}
For convenience, we assume that $\min_{x \in \Oc}f(x) > 0$. The proof for the case when right hand is strictly negative follows in a similar way.

Let $u$ be the solution to the continuous problem \eqref{eq:BVPNinfLaplace}. To derive error estimates, we consider the function $u^{\ve}$ defined as
\begin{equation}
\label{eq:DefOfUSuperEps}
  u^{\ve}(x) = \max_{y \in \overline{B}_\ve(x)} u(y), \qquad x \in \Omega^{(\ve)}.
\end{equation}
For $\vertex \in \Nhi{\ve}$ let the points $y \in \overline{B}_\ve(\vertex)$ and $y' \in \overline{B}_\ve(y)$ be defined by
\[
  u(y) = u^{\ve}(\vertex), \qquad u(y') = u^{\ve}(y).
\]
Since $\vertex \in \Nhi{\ve} = \Nh \cap \Omega^{(2 \ve)}$, we have $y, y' \in \Omega$.

According to \cite[Lemma 5.2]{MR2846345}, $u$ is locally Lipschitz and
\[
  S_{\ve}^+ u(\vertex) -  \frac{\ve}{2} f^{\ve}(\vertex) \le L(u, y) \le S_{\ve}^+ u(y) + \frac{\ve}{2} f^{\ve}(y).
\]
This implies that
\[
  \ve S_{\ve}^- u^{\ve}(\vertex) = u^{\ve}(\vertex) - \min_{x \in \overline{B}_\ve(\vertex)} u^{\ve}(x) \le u(y) - u(\vertex) = \ve S_{\ve}^+ u(\vertex)  \le \ve \l( L(u, y) + \frac{\ve}{2} f^{\ve}(\vertex) \r),
\]
and
\[
  \ve S_{\ve}^+ u^{\ve}(\vertex) =  \max_{x \in \overline{B}_\ve(\vertex)} u^{\ve}(x) - u^{\ve}(\vertex) \ge u^{\ve}(y) - u(y) = 
  \ve S_{\ve}^+ u(y) \ge \ve \l( L(u, y) - \frac{\ve}{2} f^{\ve}(y) \r).
\]
This implies that
\begin{equation}\label{eq:thm-error1-proof1}
  S_{\ve}^- u^{\ve}(\vertex) - S_{\ve}^+ u^{\ve}(\vertex) \le
  \frac{\ve}{2} \l( f^{\ve}(\vertex) + f^{\ve}(y) \r) \le \ve f(\vertex) + C \ve^{1+\alpha} |f|_{C^{0,\alpha}(\overline{B}_{2\ve}(\vertex))}.
\end{equation}
Since $\mathcal{N}_{\frakh}(\vertex) \subset \overline{B}_\ve(\vertex)$, we have
\begin{equation}\label{eq:thm-error1-proof2}
\begin{aligned}
  & \left(  S_\frakh^- u^{\ve}(\vertex) - S_\frakh^+ u^{\ve}(\vertex) \right) - \left( S_{\ve}^- u^{\ve}(\vertex) - S_{\ve}^+ u^{\ve}(\vertex) \right) \\
  & = \frac1\ve \left( \max_{x \in \overline{B}_\ve(\vertex)} u^{\ve}(x) - \max_{x \in \mathcal{N}_{\frakh}(\vertex)} u^{\ve}(x) - \min_{x \in \mathcal{N}_{\frakh}(\vertex)} u^{\ve}(x) + \min_{x \in \overline{B}_\ve(\vertex)} u^{\ve}(x) \right) \\
  & \le \frac1\ve \left( \max_{x \in \overline{B}_\ve(\vertex)} u^{\ve}(x) - \max_{x \in \mathcal{N}_{\frakh}(\vertex)} u^{\ve}(x) \right).
\end{aligned}
\end{equation}
To control the right hand side, we notice that
\[
  \overline{B}_{2\ve}(\vertex) = \bigcup_{y \in \partial \overline{B}_\ve(\vertex)} \overline{B}_\ve(y), \quad
  \max_{x \in \overline{B}_\ve(\vertex)} u^{\ve}(x) = \max_{x \in \overline{B}_{2\ve}(\vertex)} u(x) = \max_{y \in \partial \overline{B}_\ve(\vertex)} \max_{x \in \overline{B}_\ve(y)} u(x),
\]
which guarantees that for some $\vertex + \ve \bv{v} \in \partial \overline{B}_\ve(\vertex)$ with $|\bv{v}| = 1$, 
\[
  \max_{x \in \overline{B}_\ve(\vertex)} u^{\ve}(x) = \max_{x \in \overline{B}_\ve(\vertex + \ve \bv{v})} u(x) = u^{\ve}(\vertex + \ve \bv{v}).
\]
By definition of $\St$, there exists $\bv{v}_{\theta} \in \St$ such that $|\bv{v}_{\theta} - \bv{v}| \le \theta$.	
Since $\vertex + \ve \bv{v}_{\theta} \in \mathcal{N}_{\frakh}(\vertex)$,
\begin{equation}\label{eq:thm-error1-proof3}
  \max_{x \in \overline{B}_\ve(\vertex)} u^{\ve}(x) - \max_{x \in \mathcal{N}_{\frakh}(\vertex)} u^{\ve}(x)
  \le u^{\ve}(\vertex + \ve \bv{v}) - u^{\ve}(\vertex + \ve \bv{v}_{\theta})
  \lesssim (\ve \theta)^{\alpha} |u|_{C^{0,\alpha}(\overline{B}_{2\ve}(\vertex))}.
\end{equation}

Recall that 
\[
  -\Delta_{\infty, \frakh}^\diamond w(\vertex) = -\frac1\ve \left( S_\frakh^+ \interp w(\vertex) - S_\frakh^- \interp w(\vertex) \right),
\]
we also have
\[
  \left| -\Delta_{\infty, \frakh}^\diamond u^{\ve}(\vertex) - \frac1\ve \left(  S_\frakh^- u^{\ve}(\vertex) - S_\frakh^+ u^{\ve}(\vertex) \right) \right|
  \lesssim \frac1{\ve^2} \Vert u^{\ve} - \interp u^{\ve} \Vert_{L^{\infty}(B_\ve(\vertex))}
  \lesssim  \frac{h^{\alpha}}{\ve^2} |u|_{C^{0,\alpha}(\overline{B}_{2\ve}(\vertex))}.
\]
Combine now \eqref{eq:thm-error1-proof1}---\eqref{eq:thm-error1-proof2}---\eqref{eq:thm-error1-proof3} to obtain
\[
-\Delta_{\infty, \frakh}^\diamond u^{\ve}(\vertex) \le  f(\vertex) + C\ve^{\alpha} |f|_{C^{0,\alpha}(\overline{B}_{2\ve}(\vertex))} + C \frac{(\ve \theta)^{\alpha}}{\ve^2}  |u|_{C^{0,\alpha}(\overline{B}_{2\ve}(\vertex))} + C \frac{h^{\alpha}}{\ve^2}  |u|_{C^{0,\alpha}(\overline{B}_{2\ve}(\vertex))},
\]
leading to
\[
  -\Delta_{\infty, \frakh}^\diamond u^{\ve}(\vertex) \le f(\vertex) + C\ve^{\alpha} |f|_{C^{0,\alpha}(\Oc)} + C \left( \frac{\theta^{\alpha}}{\ve^{2-\alpha}} 
  + \frac{h^{\alpha}}{\ve^2} \right)
  |u|_{C^{0,\alpha}(\Oc)}.
\]
Choose
\begin{equation}\label{eq:inhomo-error-proof-I1}
  I_1 = \frac{C}{ \min_{x \in \Oc} f(x) } \l( \ve^{\alpha} |f|_{C^{0,\alpha}(\Oc)} +  \left( \frac{\theta^{\alpha}}{\ve^{2-\alpha}} 
+ \frac{h^{\alpha}}{\ve^2} \right) |u|_{C^{0,\alpha}(\Oc)} \r).
\end{equation}
Observe that, for sufficiently small $\frakh = (h,\ve,\theta)$, we have $I_1 < 1$. Consequently, 
\[
\begin{aligned}
  -\Delta_{\infty, \frakh}^\diamond[ (1 - I_1 ) u^{\ve} ](\vertex) 
  &\le (1 - I_1) \l( f(\vertex) + C\ve^{\alpha} |f|_{C^{0,\alpha}(\Oc)} + C \left( \frac{\theta^{\alpha}}{\ve^{2-\alpha}} 
  + \frac{h^{\alpha}}{\ve^2} \right)
  |u|_{C^{0,\alpha}(\Oc)} \r) \\
  &\le f(\vertex) = -\Delta_{\infty, \frakh}^\diamond \uve (\vertex).
\end{aligned}
\]
By the discrete comparison principle of \Cref{thm:simple-DCP} we have, for all $\vertex \in \Nhi{\ve}$,
\[
\begin{aligned}
  \uve(\vertex) &\ge (1 - I_1 ) u^{\ve}(\vertex) + \min_{\othervertex \in \Nhb{\ve}} [\wt{g}_\ve(\othervertex) - (1 - I_1 ) u^{\ve}(\othervertex) ] \\
  &\ge u^{\ve}(\vertex) - I_1 \max_{x \in \overline{\Omega}} u(x) + I_1
  \min_{\othervertex \in \Nhb{\ve}} u^{\ve} (\othervertex)
  + \min_{\othervertex \in \Nhb{\ve}} (\wt{g}_\ve(\othervertex) - u^{\ve}(\othervertex) ) \\
  &\ge u(\vertex) - I_1 \osc_{\Omega} u - C \l( |\wt{g}_\ve|_{C^{0,\alpha}(\Oc)} + |u|_{C^{0,\alpha}(\Oc)} \r) \ve^{\alpha} - \Vert g - \wt{g}_{\ve} \Vert_{L^{\infty}(\pO)},
\end{aligned}
\]
where $\osc_{\Omega} u = \max_{x \in \Oc} u(x) - \min_{x \in \Oc} u(x)$ and we used \eqref{eq:bdry-gve-u} in the derivation.
We can obtain a similar upper bound for $\uve$ and hence
\[
\Vert u-\uve \Vert_{L^{\infty}(\Omega_h)} \lesssim \ve^{\alpha} + \frac{\theta^{\alpha}}{\ve^{2-\alpha}} + \frac{h^{\alpha}}{\ve^2}.
\]
This proves the result.
\end{proof}

By properly scaling all discretization parameters, the previous result allows us to obtain explicit rates of convergence. The following result is the first explicit rate of convergence for problems involving the normalized $\infty$--Laplacian.

\begin{Corollary}[convergence rates ($f \not\equiv0$)] \label{cor:RatesInhom}
Under the same setting and assumptions as in \Cref{Thm:inhomo-error}, we can choose 
\[
  \ve \approx h^{\frac{\alpha}{2+\alpha}}, \qquad \ve\theta \approx h,
\]
to obtain
\[
  \Vert u - \uve \Vert_{L^{\infty}(\Omega_h)} \lesssim   
  h^{\frac{\alpha^2}{2+\alpha}}.
\]
In particular, if $u \in C^{0,1}(\Oc)$, we have the error estimate
\[
  \| u - \uve \|_{L^\infty(\Omega_h)} \lesssim h^{1/3}.
\]
\end{Corollary}
\begin{proof}
Use the indicated scalings.
\end{proof}

\subsection{Convergence rates for the homogeneous problem}\label{sub:Ratesfzero}

Let us now obtain convergence rates in the case that Assumption \textbf{RHS}.\ref{Ass.RHS.zero} holds. First of all, to understand why different arguments are needed in this case consider two problems of the form \eqref{eq:BVPNinfLaplace}, with the same boundary condition but different right hand sides:
\[
  -\Delta^\diamond_{\infty} u_1(x) = f_1(x), \qquad 
  -\Delta^\diamond_{\infty} u_2(x) = f_2(x), \qquad 
  \forall x \in \Omega.
\]
If $f_1(x), f_2(x) \ge f_0>0$ for all $x \in \Oc$, then we have the stability result
\[
\Vert u_1 - u_2 \Vert_{L^{\infty}(\Omega)} \lesssim \dfrac{1}{f_0} \Vert f_1 - f_2 \Vert_{L^{\infty}(\Omega)};
\]
see \cite[Proposition 6.3]{MR2846345}.
However, in general, if $f_1, f_2$ are not bounded away from zero, it is not possible to have a stability estimate of the form
\[
\Vert u_1 - u_2 \Vert_{L^{\infty}(\Omega)} \lesssim \Vert f_1 - f_2 \Vert_{L^{\infty}(\Omega)}.
\]
In fact, if this could be proved, then the error estimates that we prove below can be improved. The optimal case would also give  $\mathcal{O}(h^{1/3})$.

\begin{Theorem}[error estimate for the homogeneous problem]\label{Thm:homo-error}
Let $u$ be the viscosity solution of \eqref{eq:BVPNinfLaplace} and $\uve$ be the solution of \eqref{eq:inf-laplace-discrete}. Under our running assumptions suppose that \textbf{RHS}.\ref{Ass.RHS.zero} and, for all $h>0$, \textbf{M}.\ref{Ass.M.connect} holds. If $\frakh$ is sufficiently small, such that $\ve - \ve \theta - h > 0$, and
\[
  (2h + \ve \theta)/\ve^2,
\]
can be made sufficiently small, then we have
\[
\Vert u- \uve \Vert_{L^{\infty}(\Omega_h)} \lesssim \ve^{\alpha} 
+ \frac{\sqrt{2h + \ve \theta}}{\ve},
\]
where the implied constant depends only on the dimension $d$, the domain $\Omega$, the shape regularity of the family $\{\Th\}_{h>0}$, and the $C^{0,\alpha}$ norms of the data $g$ and the solution $u$.
\end{Theorem}
\begin{proof}	
Similar to the proof of \Cref{Thm:inhomo-error}, we aim to construct a discrete subsolution. To achieve this, first, we employ the approximation introduced in \cite[Theorem 2.1]{MR2372480}; see also the proof of Theorem 2.19 in \cite{MR2846345}. This result asserts, for $\gamma>0$, the existence of $v_{\gamma} \in C(\overline{\Omega}) \cap W^{1,\infty}_{\mathrm{loc}}(\Omega)$ with the following properties:
\[
  -\Delta_{\infty}^{\diamond} v_{\gamma} \le 0,  \text{ in } \Omega, \quad
    v_{\gamma} = g,  \text{ on } \partial \Omega, \quad
    L(v_{\gamma},\cdot) \ge \gamma,  \text{ in } \Omega, \quad
    u - \gamma \diam(\Omega) \le v_{\gamma} \le u,  \text{ in } \Omega,
\]
where $L(v_{\gamma}, \cdot)$ is the local Lipschitz constant defined in \eqref{eq:local-Lip}. Let $m = \min_{x \in \Oc} u(x)$, and $ M = \max_{x \in \Oc} u(x)$. We have
\[
  m_{\gamma} = m - \gamma \diam(\Omega) \le v_{\gamma} \le M.
\]

Define the function
\[
  v_{\gamma}^{\ve + h}(x) = \max_{y \in \overline{B}_{\ve+h}(x)} v_{\gamma}(y), \quad x \in \Omega^{(\ve+h)},
\]
and a discrete subsolution of the form $\uve^{-} = q(v_{\gamma}^{\ve + h})$, where $q(t) = t + a (t - m_{\gamma})^2$ is a quadratic function with a parameter $a > 0$ to be determined. Nevertheless, since $a>0$ this implies that $q'(t) \ge 0$ for $t \in [m_{\gamma}, M]$. Consider a vertex $\vertex \in \Nhi{\ve}$, and let $y \in \overline{B}_{\ve+h}(\vertex,)$ be such that
\[
  v_{\gamma}(y) = v_{\gamma}^{\ve+h}(\vertex) = \max_{x \in \overline{B}_{\ve+h}(\vertex)} v_{\gamma}(x).
\]
Recall that
\[
  S_{\frakh}^- \interp w(\vertex) = \frac{1}{\ve} \l( w(\vertex) - \min_{x \in \mathcal{N}_{\frakh}(\vertex)} \interp w(x) \r).
\]
Consider a point $x \in \mathcal{N}_{\frakh}(\vertex)$ and an element $T \ni x$.
Then for any vertex $\othervertex \in T \cap \Nh$, we have
\[
  |\vertex - \othervertex| \le |\vertex - x| + |x -  \othervertex| \le \ve + h,
\]
which implies that
\[
  v_{\gamma}^{\ve + h}(\othervertex) \ge v_{\gamma}(\vertex).
\]
Using the monotonicity of the quadratic function $q$ we see that
\begin{equation*}
  \begin{aligned}
  S_{\frakh}^- \interp \uve^{-}(\vertex)
  &= \frac{1}{\ve} \l( q(v_{\gamma}^{\ve + h}(\vertex)) - \min_{x \in \mathcal{N}_{\frakh}(\vertex)} \interp q(v_{\gamma}^{\ve + h}(x)) \r)
  \le \frac{1}{\ve} \l( q(v_{\gamma}(y)) - q(v_{\gamma}(\vertex)) \r).
  \end{aligned}
\end{equation*}	
For convenience, let $b = S_{\ve + h}^+ v_{\gamma}(\vertex)$ where the operator is defined in \eqref{eq:def-semi-S}, then
\[
  v_{\gamma}(y) = v_{\gamma}(\vertex) + (\ve + h) b,
\]
and thus 
\begin{equation}\label{eq:homo-error-proof0}
  \ve S_{\frakh}^- \interp \uve^{-}(\vertex) \le 
  q(v_{\gamma}(y)) - q(v_{\gamma}(y) - (\ve + h)b).
\end{equation}	
By \cite[Lemma 5.2]{MR2846345}, we have
\begin{equation*}
  S_{\delta}^+ v_{\gamma}(y) \ge L(v_{\gamma}, y) \ge b = S_{\ve + h}^+ v_{\gamma}(\vertex) \ge L(v_{\gamma}, \vertex) \ge \gamma,
\end{equation*}
for any $\delta > 0$. Let $\delta = \ve - \ve \theta - h > 0$ and $y' \in \overline{B}_\delta(y)$ such that
\[
  v_{\gamma}(y') = \max_{x \in \overline{B}_\delta(y)} v_{\gamma}(x) = v_{\gamma}(y) + \delta S_{\delta}^+ v_{\gamma}(y) 
\ge v_{\gamma}(y) + \delta b.
\]
Since
\[
  |y' - \vertex| \le |y' - y| + |y - \vertex| \le \ve - \ve \theta - h + \ve + h = 2\ve - \ve \theta,
\]
we claim that there exists $\bv{v}_{\theta} \in \St$ such that $x' = \vertex + \ve \bv{v}_{\theta} \in \mathcal{N}_{\frakh}(\vertex)$ satisfies $|x' - y'| \le \ve$. To see this, let $\bv{v} \in \mS$ be such that $y' - \vertex = |y' - \vertex| \bv{v}$ and choose $\bv{v}_{\theta} \in \St$ such that $|\bv{v}_{\theta} - \bv{v}| \le \theta$. If $|y' - \vertex| \ge \ve \theta$, then we have
\[
|x' - y'| \le |x' - (\vertex + \ve \bv{v})| + |(\vertex + \ve \bv{v}) - y'|
\le \ve |\bv{v}_{\theta} - \bv{v}| + | \ve - |y' - \vertex| |
\le \ve \theta + (\ve - \ve \theta) = \ve.
\] 
On the other hand, if $|y' - \vertex| < \ve \theta$, then 
\[
\begin{aligned}
  |x' - y'| &\le |x' - (\vertex + |y' - \vertex| \bv{v}_{\theta})| + |(\vertex + |y' - \vertex| \bv{v}_{\theta}) - y'| \\
  &\le |\ve - |y' - \vertex| | + |y' - \vertex| | \bv{v}_{\theta} - \bv{v} | 
  \le \ve - |y' - \vertex| + |y' - \vertex| \theta \le \ve,
\end{aligned}
\]
because of $\theta \le 1$.

Let now $T \in \Th$ be such that $x' \in T$. Then for any vertex $\othervertex \in T$,
\[
  |y' - \othervertex| \le |y' - x'| + |x' - \othervertex| \le \ve + h \le \ve + h.
\]
This implies that
\[
  v_{\gamma}^{\ve+h}(\othervertex) \ge v_{\gamma}(y') \ge v_{\gamma}(y) + \delta b.
\]
By the monotonicity of $q$, we also have
\[
  q(v_{\gamma}^{\ve+h}(\othervertex)) \ge q(v_{\gamma}(y) + \delta b),
\]
which implies that
\[
  \uve^{-}(x') = \interp q(v_{\gamma}^{\ve+h}(x')) \ge  q(v_{\gamma}(y) + \delta b).
\]
Hence,
\begin{equation}\label{eq:homo-error-proof1}
  \ve S_\frakh^+ \interp \uve^{-}(\vertex) \ge \uve^{-}(x') - \uve^{-}(\vertex) \ge q(v_{\gamma}(y) + \delta b) - q(v_{\gamma}(y)).
\end{equation}

Let us now choose the parameters $a$ and $\gamma$ to guarantee that
\[
  -\Delta_{\infty, \frakh}^\diamond \uve^{-}(\vertex) = -\frac1\ve \left( S_\frakh^+ \interp \uve^{-}(\vertex) - S_\frakh^- \interp \uve^{-}(\vertex) \right) \le 0,
\]
i.e., $S_\frakh^+ \interp \uve^{-}(\vertex) \ge S_\frakh^- \interp \uve^{-}(\vertex)$. Upon combining \eqref{eq:homo-error-proof0} and \eqref{eq:homo-error-proof1} we see that, to achieve this, it suffices to ensure
\[
  q(v_{\gamma}(y) + \delta b) - q(v_{\gamma}(y)) \ge q(v_{\gamma}(y)) - q(v_{\gamma}(y) - (\ve + h)b).
\]
Since
\[
\begin{aligned}
  & q(v_{\gamma}(y) + \delta b) - q(v_{\gamma}(y))
  = \delta b (1 + 2a (v_{\gamma}(y) - m_{\gamma}) ) + a (\delta b)^2, \\
  & q(v_{\gamma}(y)) - q(v_{\gamma}(y) - (\ve + h)b)
  = (\ve + h) b (1 + 2a (v_{\gamma}(y) - m_{\gamma}) ) - a ((\ve+h)b)^2,
\end{aligned}
\]
it is enough to require
\[
\begin{aligned}
  0 &\le a (\delta b)^2 + a ((\ve+h)b)^2 -(\ve + h) b (1 + 2a (v_{\gamma}(y) - m_{\gamma}) ) +\delta b (1 + 2a (v_{\gamma}(y) - m_{\gamma}) ) \\
  &= ab^2 (\delta^2 + (\ve + h)^2 ) - (2h + \ve \theta) b (1 + 2a (v_{\gamma}(y) - m_{\gamma}) ).
\end{aligned}
\]
Since $b \ge \gamma$ and
\[
  ab (\delta^2 + (\ve + h)^2 ) - (2h + \ve \theta) (1 + 2a (v_{\gamma}(y) - m_{\gamma}) )
  \ge ab \ve^2 - (2h + \ve \theta) (1 + 2a (M - m + \gamma) ),
\]
then we always have $-\Delta_{\infty, \frakh}^\diamond \uve^{-}(\vertex) \le 0$ provided that
\begin{equation}\label{eq:homo-error-proof2}
  \frac13 a \gamma \ve^2 \ge 2h + \ve \theta, \quad
  \frac13 a \gamma \ve^2 \ge (2h + \ve \theta) 2a (M- m), \quad
  \frac13 a \gamma \ve^2 \ge (2h + \ve \theta) 2a \gamma.
\end{equation}
This can be achieved by requiring that $\beta = (2h + \ve \theta)/\ve^2$ is small enough and satisfies
\[
  2(M -m) \sqrt{3\beta} \le 1, \quad \beta \le 1/6.
\]
Choosing $a = \gamma = \sqrt{3 \beta}$ guarantees \eqref{eq:homo-error-proof2} and thus $-\Delta_{\infty, \frakh}^\diamond \uve^{-}(\vertex) \le 0$. Consequently, $\uve^-$ is a subsolution and we can apply the comparison principle of \Cref{thm:DCP} to obtain
\begin{equation}\label{eq:homo-error-proof3}
\begin{aligned}
  \uve(\vertex) &\ge \uve^{-}(\vertex) + \min_{\othervertex \in \Nhb{\ve}} (\uve(\othervertex) - \uve^{-}(\othervertex) ), \quad \forall \vertex \in \Nhi{\ve}.
\end{aligned}
\end{equation}
Notice that, for any vertex $\vertex \in \Nh$,
\[
\begin{aligned}
  &|\uve^{-}(\vertex) - u(\vertex)| \le
  | q(v_{\gamma}^{\ve + h}(\vertex)) - v_{\gamma}^{\ve + h}(\vertex)|
  + |v_{\gamma}^{\ve + h}(\vertex) - u^{\ve + h}(\vertex)| + |u^{\ve + h}(\vertex) - u(\vertex)| \\
  &\le a(M-m + \gamma)^2 + \gamma \diam(\Omega)+ (\ve + h)^{\alpha} |u|_{C^{0,\alpha}(\Oc)}
  \lesssim a(M- m)^2 + \gamma\diam(\Omega) + (\ve + h)^{\alpha} |u|_{C^{0,\alpha}(\Oc)} \\
  &\lesssim \l( (M - m)^2 + \diam(\Omega) \r) \sqrt{\beta} + \ve^{\alpha} |u|_{C^{0,\alpha}(\Oc)}.
\end{aligned}
\]	
Now, if $\othervertex \in \Nhb{\ve}$, we recall \eqref{eq:bdry-gve-u} to obtain 
\[
\begin{aligned}
  &|\uve(\othervertex) - \uve^{-}(\othervertex)|
  = |\wt{g}_\ve(\othervertex) - \uve^{-}(\othervertex)|
  \le |\wt{g}_\ve(\othervertex) - u(\othervertex)|  + |u(\othervertex) - \uve^{-}(\othervertex)| \\
  &	\lesssim   \l( (M - m)^2 + \diam(\Omega)\r) \sqrt{\beta} + \ve^{\alpha} \l(  |\wt{g}_\ve|_{C^{0,\alpha}(\Oc)} + |u|_{C^{0,\alpha}(\Oc)} \r) + \Vert g - \wt{g}_{\ve} \Vert_{L^{\infty}(\Omega)}.
\end{aligned}
\]
Plugging the inequalities above into \eqref{eq:homo-error-proof3} implies
\[
  \uve(\vertex) \ge u(\vertex) -  C\l( (M - m)^2 + \diam(\Omega) \r) \sqrt{\beta} - \ve^{\alpha} \l(  |\wt{g}_{\ve}|_{C^{0,\alpha}(\Oc)} + |u|_{C^{0,\alpha}(\Oc)} \r) - \Vert g - \wt{g}_{\ve} \Vert_{L^{\infty}(\Omega)}.
\]
We can obtain a similar upper bound for $\uve$, and hence we conclude that
\[
\begin{aligned}
  \Vert u-\uve \Vert_{L^{\infty}(\Omega_h)} \lesssim &\ve^{\alpha} \l(  |g|_{C^{0,\alpha}(\Oc)} + |u|_{C^{0,\alpha}(\Oc)} \r) + 
  \Vert g - \wt{g}_{\ve} \Vert_{L^{\infty}(\Omega)} \\ 
  &+ \frac{\sqrt{2h + \ve \theta}}{\ve} \l( \l( \max_{x \in \Oc} u(x) - \min_{x \in \Oc} u(x) \r)^2 + \diam(\Omega) \r),
\end{aligned}
\]
as we intended to show.
\end{proof}

Once again, we can properly scale the parameters to obtain explicit rates of convergence.

\begin{Corollary}[convergence rates ($f\equiv0$)]\label{cor:rates.f.zero}
Under the same assumptions of \Cref{Thm:homo-error}, we can choose
\[
  \ve \approx h^{\frac1{2(\alpha+1)}}, \qquad \ve\theta \approx h, 
\]
to obtain
\[
  \Vert u - \uve \Vert_{L^\infty(\Omega_h)} \lesssim h^{\frac\alpha{2(\alpha+1)}}.
\]
In particular, if $u \in C^{0,1}(\Oc)$, we have the error estimate
\[
  \Vert u - \uve \Vert_{L^\infty(\Omega_h)} \lesssim h^{1/4}.
\]
\end{Corollary}
\begin{proof}
It immediately follows from \Cref{Thm:homo-error} when using the prescribed scalings.
\end{proof}

\section{Extensions and variations}\label{sec:Variants}

In this section we consider some possible extensions of our numerical scheme, as well as variations on the type of problems that our approach can handle. To avoid unnecessary repetition, and to keep the presentation short, we will mostly state the main results and sketch their proofs, as many of the arguments are repetitions or slight variants of what we have presented before.

\subsection{Meshless discretizations}\label{sub:PointClouds}

The avid reader may have already realized that we made very little use of the fact that our discrete functions came from $\Vh$, a space of piecewise linears subject to the mesh $\Th$. In fact, this was only used so that functions can be evaluated at arbitrary points, and not just the vertices of our grid.

This motivates the following meshless discretization. Let $\calC_h = \{\vertex\}$ be a finite cloud of points. We set 
\[
  h_\vertex = \min\{ |\vertex - \othervertex| : \othervertex \in \calC_h \setminus \{\vertex \} \}, \qquad h = \max_{\vertex \in \calC_h} h_\vertex,
\]
and assume that $\calC_h$ satisfies
\begin{equation}
\label{eq:CloudIsDense}
  \Oc \subset \bigcup_{\vertex \in \calC_h} \Bc_{h/2}(\vertex).
\end{equation}
Moreover, we assume the following quasiuniformity condition on the cloud points: There is $\sigma \in (0,1]$ such that, for all $\calC_h$, we have
\[
  \sigma h \leq \inf_{\vertex \in \calC_h} h_\vertex.
\]

Let $\ve>0$. We define $\Nhb{\ve}$ to be the points in $\calC_h$ that are at a distance at most $2\ve$ of $\partial\Omega$, and $\Nhi{\ve} = \calC_h \setminus \Nhb{\ve}$. For $\vertex \in \calC_h$ we set $\calC_\ve(\vertex) = \calC_h \cap \Bc_\ve(\vertex)$.

The set of functions on $\calC_h$ is
\[
  \Xh = \mR^{\calC_h} = \left\{ w : \calC_h \to \mR \right\}.
\]
Over $\Xh$ we then define our meshless discretization of the operator $\Delta_\infty^\diamond$ as
\begin{equation}
\label{eq:InFLapCloud}
  \Delta_{\infty,\calC_h,\ve}^\diamond w(\vertex) = \frac1{\ve^2} \left( \max_{\othervertex \in \calC_\ve(\vertex)} w(\othervertex) - 2 w(\vertex) + \min_{\othervertex \in \calC_\ve(\vertex)} w(\othervertex) \right), \qquad \vertex \in \Nhi{\ve}.
\end{equation}
The meshless variant of \eqref{eq:inf-laplace-discrete} is: find $u_{\calC_h,\ve} \in \Xh$ such that
\begin{equation}
  \label{eq:inf-laplace-cloud}
  -\Delta_{\infty,\calC_h,\ve}^\diamond u_{\calC_h,\ve}(\vertex) = f(\vertex), \quad \forall \vertex \in \Nhi{\ve}, \qquad u_{\calC_h,\ve}(\vertex) = \wt{g}_\ve(\vertex), \quad \forall \vertex \in \Nhb{\ve}.
\end{equation}

To provide an analysis of scheme \eqref{eq:inf-laplace-cloud} we must assume that the cloud $\calC_h$ is \emph{symmetric at every point}. By this we mean that there is $\ve_0>0$ such that, for all $\ve \leq \ve_0$ and $\vertex \in \Nhi{\ve}$ we must have
\[
  \othervertex \in \calC_\ve(\vertex) \qquad \implies \qquad  2 \vertex -\othervertex \in \calC_\ve(\vertex) \subset \calC_h.
\]
This is the surrogate to the requirement that the set $\mS_\theta$ is symmetric. Let $\mathcal{G}_{h,\ve}$ be the (undirected) graph with vertices $\calC_h$ and edges $\mathcal{E}_\ve$; where, for $\vertex,\othervertex \in \calC_h$, we have
\[
  (\vertex,\othervertex) \in \mathcal{E}_\ve \qquad \iff \qquad |\vertex - \othervertex| \le \ve.
\]
As a proxy for Assumption \textbf{M}.\ref{Ass.M.connect} we require then following condition.
\begin{enumerate}[\textbf{C}.1]
  \item \label{Ass.C.connect} The graph $\mathcal{G}_{h,\ve}$ is connected.
\end{enumerate}

For this variant of our scheme we have the following properties.

\begin{Lemma}[existence, uniqueness, stability]
Assume \textbf{RHS}.\ref{Ass.RHS.cont}--\ref{Ass.RHS.f<>0}, and \textbf{BC}.\ref{Ass.BC.cont}--\ref{Ass.BC.ext}. Moreover, if \textbf{RHS}.\ref{Ass.RHS.zero} holds, suppose that for all $\ve,h>0$ we have \textbf{C}.\ref{Ass.C.connect}. For any choice of parameters $(h,\ve)$, there exists a unique $u_{\calC_h,\ve} \in \Xh$ that solves \eqref{eq:inf-laplace-cloud}. Moreover, this solution satisfies
\[
  \max_{\vertex \in \calC_h} |u_{\calC_h,\ve}(\vertex)| \leq C \max_{\vertex \in \Nhi{\ve}} |f(\vertex)| + \max_{\vertex \in \Nhb{\ve}} | \wt{g}_\ve(\vertex)|.
\]
\end{Lemma}
\begin{proof}
We immediately see that:
\begin{enumerate}[$\bullet$]
  \item The analogue of \Cref{lemma:monotonicity} follows after little modification to the presented proof.
  
  \item The previous item implies that the needed variants of \Cref{thm:simple-DCP} and \Cref{thm:DCP} are valid as well. To prove \Cref{thm:DCP} one needs to use Assumption \textbf{C}.\ref{Ass.C.connect}.
  
  \item Having comparison immediately implies the claimed stability estimate. To obtain it, one repeats the proof of \Cref{lemma:stability}. We only need to replace the Lagrange interpolant with the projection onto the cloud $\calC_h$.
  
  \item Existence and uniqueness then follows as in \Cref{lemma:ExistUnique}. \qedhere
\end{enumerate}
\end{proof}

The consistency properties of $\Delta_{\infty,\calC_h,\ve}^\diamond$ are similar to \Cref{lemma:consistency} and \Cref{lemma:consistency-p0} with $\theta = h$. For that, we use the symmetry of $\calC_\ve(\vertex)$ and \eqref{eq:CloudIsDense}. This allows us to assert convergence for solutions of \eqref{eq:inf-laplace-cloud}. Define the envelopes in an analogous way to \eqref{eq:def-uo-uu}. We have the following result.

\begin{Theorem}[convergence]\label{theorem:cloud}
Assume that the right hand side $f$ satisfies Assumptions \textbf{RHS}.\ref{Ass.RHS.cont}--\ref{Ass.RHS.f<>0}. Assume that the boundary datum $g$ satisfies Assumptions \textbf{BC}.\ref{Ass.BC.cont}--\ref{Ass.BC.ext}. Let $\{ (h_j, \ve_j) \}_{j=1}^\infty$ be a sequence of discretization parameters such that $(h_j, \ve_j) \to (0,0)$ as $j \uparrow \infty $, it satisfies \eqref{eq:para-condition} with $\theta_j = h_j$; and if \textbf{RHS}.\ref{Ass.RHS.zero} holds, it additionally satisfies \textbf{C}.\ref{Ass.C.connect} for all $j \in \mathbb{N}$. In this framework we have that the sequence $\{u_{\calC_{h_j},\ve_j} \in \mathbb{X}_{h_j}\}_{j=1}^\infty$ of solutions to \eqref{eq:inf-laplace-cloud} converges uniformly to $u$, the solution of \eqref{eq:BVPNinfLaplace} as $j \uparrow \infty$.
\end{Theorem}
\begin{proof}
Repeat the proofs of \Cref{lemma:consistency-combined}, \Cref{lemma:bdry}, and \Cref{thm:convergence}.
\end{proof}

Rates of convergence, analogous to those of \Cref{cor:RatesInhom} and \Cref{cor:rates.f.zero}, can also be obtained. We omit the details.

\subsection{Obstacle problems}\label{sub:Obstacle}

In \cite{MR3299035}, see also \cite[Section 5.2]{RossiBook}, the following variant of the game described in \Cref{sub:TugOfWar} was proposed and analyzed. In addition to the already stated rules, Player I has the option, after every move, to stop the game. In this case, Player II will pay Player I the amount
\[
  Q = \chi(x_{k+1}).
\]
Here $\chi : \Oc \to \mR$ is a given function with $\chi < g$ on $\pO$. It is shown that, in this case, the game also has a value $u^\ve$ which satisfies
\[
  u^\ve(x) \geq \chi(x), \qquad 2u^\varepsilon(x) - \left( \sup_{y \in \Bc_\ve(x) \cap \bar\Omega} u^\varepsilon(y) + \inf_{y \in \Bc_\ve(x) \cap \bar\Omega} u^\varepsilon(y)\right) \geq \varepsilon^2 f(x),
\]
and, if the first inequality is strict, the second is an equality. In the limit $\ve \downarrow 0$ we obtain the following obstacle problem
\begin{equation}
\label{eq:ObstacleProblem}
  \bT[u;f,\chi](x) = 0, \text{ in } \Omega, \qquad u = g, \text{ on } \pO;
\end{equation}
where
\[
  \bT[w;f,\chi](x) = \min\left\{ -\Delta_\infty^\diamond w(x) - f(x), w(x) - \chi(x) \right\}.
\]

For convenience of presentation, given $w \in C(\Oc)$, we define the coincidence, non-coincidence sets, and free boundary to be
\[
  \mathcal{C}(w) = \l\{ x \in \Omega: w(x) = \chi(x) \r\}, \quad \Omega^+(w) = \Omega \setminus \mathcal{C}(w), \quad
  \Gamma(w) = \partial\Omega^+(w) \cap \Omega.
\]
Notice that, if $u$ is sufficiently smooth, \eqref{eq:ObstacleProblem} is equivalent to the following complementarity conditions: $u(x) \geq \chi(x)$ for all $x \in \Omega$, and,
\[
  -\Delta_\infty^\diamond u(x) = f(x), \text{ in } \Omega^+(u), \quad -\Delta_\infty^\diamond u(x) \geq  f(x), \text{ in } \mathcal{C}(u).
\]

To analyze the problem we introduce an assumption on the obstacle.
\begin{enumerate}[\textbf{OBS}.1]
  \item \label{Ass.OBS.cont} The obstacle $\chi \in C(\Oc)$ is such that $\chi < g$ on $\partial\Omega$.
\end{enumerate}

Under assumptions \textbf{RHS}.\ref{Ass.RHS.cont}, \textbf{BC}.\ref{Ass.BC.cont}, and \textbf{OBS}.\ref{Ass.OBS.cont} existence of viscosity solutions to \eqref{eq:ObstacleProblem} was established in \cite{MR3299035}. In addition, if \textbf{RHS}.\ref{Ass.RHS.f<>0}, then uniqueness is guaranteed. Finally, it is shown that if $f \in C^{0,1}(\Oc)$, $g \in C^{0,1}(\pO)$, and $\chi \in C^2(\Oc)$ then $u \in C^{0,1}(\Oc)$. Further properties of the solution to \eqref{eq:ObstacleProblem} were explored in \cite{MR3421912}.


\subsubsection{Comparison principle}\label{sub:CompareObstacle}
In addition to the properties described above, to analyze numerical schemes, we will need a comparison principle for semicontinuous sub- and supersolutions of \eqref{eq:ObstacleProblem}. Since we were not able to locate one in the form that is suited for our purposes, we present one here.

We first point out that a comparison principle with strict inequality, like the one we give below, follows directly from \cite[Section 5.C]{MR1118699}.

\begin{Lemma}[strict comparison]\label{lemma:obstacle-CP-strict}
Let $\delta_w, \delta_v \in \mR$. Assume that $w \in \USC(\Oc)$ and $ v \in \LSC(\Oc)$ satisfy, in the viscosity sense,
\[
  \bT[w; f, \chi] \le \delta_w < \delta_v \le \bT[v; f, \chi] \quad \text{ in } \Omega.
\]
Further assume that $f, \chi$ satisfy Assumptions \textbf{RHS}.\ref{Ass.RHS.cont} and \textbf{OBS}.\ref{Ass.OBS.cont}. If $w \le v$ on $\pO$, then
\[
  w \le v \quad \text{ in } \Omega.
\]
\end{Lemma}

With the aid of \Cref{lemma:obstacle-CP-strict}, we can now prove a comparison principle.

\begin{Theorem}[comparison]\label{thm:obstacle-CP}
Assume that $w \in \USC(\Oc), v \in \LSC(\Oc)$ satisfy, in the viscosity sense,
\[
  \bT[w; f, \chi] \le 0 \le \bT[v; f, \chi] \quad \text{ in } \Omega.
\]
Assume, in addition, that $f, \chi$ satisfy Assumptions \textbf{RHS}.\ref{Ass.RHS.cont}--\ref{Ass.RHS.f<>0} and \textbf{OBS}.\ref{Ass.OBS.cont}.
If $w \le v$ on $\pO$, then
\[
  w \le v \quad \text{ in } \Omega.
\]
\end{Theorem}
\begin{proof}
Based on how Assumption \textbf{RHS}.\ref{Ass.RHS.f<>0} is satisfied, we split the proof into three cases.

\emph{Case 1:} $\sup\{f(x): x \in \Omega\} \le -f_0 < 0$. Since $w \in \USC(\Oc)$, we have that $w$ is bounded from above. Let $M$ be an upper bound of $w$. For any $\beta > 0$, consider the function
\[
  w_{\beta}(x) = (1 + \beta) w(x) - (M + 1) \beta \le w(x) - \beta.
\]
From $\bT[w; f, \chi] \le 0$, we argue that $w_{\beta}$ satisfies, in the viscosity sense,
\[
  \bT[w_{\beta}; f, \chi] \le \max\{ -\beta f_0, -\beta \} < 0.
\]
Let us, at least formally, provide an explanation. If $x \in \mathcal{C}(w)$, we have
\[
  \bT[w_{\beta}; f, \chi](x) \le w_{\beta}(x) - \chi(x)
  \le w(x) - \beta - \chi(x) = -\beta.
\]
On the other hand, for $x \in \Omega^+(w)$, we have
\[
  \bT[w_{\beta}; f, \chi](x) \le -\Delta_\infty^\diamond w_{\beta}(x) - f(x)
  = -(1 + \beta) \Delta_\infty^\diamond w(x) - f(x)
  = -\beta \Delta_\infty^\diamond w(x) \le -\beta f_0.
\]
Now, since we have $\bT[w_{\beta}; f, \chi] \le \max\{ -\beta f_0, -\beta \} < 0 \le \bT[v; f, \chi]$, we can apply \Cref{lemma:obstacle-CP-strict} to obtain that
\[
w_{\beta}(x) = (1 + \beta) w(x) - (M + 1) \beta \le v(x), \qquad \forall x \in \Omega, \quad \forall \beta>0.
\]
Letting $\beta \downarrow 0$ finishes the proof.

\emph{Case 2:} $\inf\{f(x): x \in \Omega\} \ge f_0 > 0$. Similar to the Case 1, let $m$ be a lower bound of $v$ and, for $\beta>0$, define the function
\[
  v_{\beta}(x) = (1 + \beta) v(x) - (m - 1) \beta \ge v(x) + \beta.
\]  
It is easy to verify that 
\[
  \bT[v_{\beta}; f, \chi] \ge \min\{ \beta f_0, \beta \} > 0
\]
in the viscosity sense. Thus, we use \Cref{lemma:obstacle-CP-strict} to arrive at
\[
  v_{\beta}(x) = (1 + \beta) v(x) - (m - 1) \ge w(x), \qquad \forall x \in \Omega, \quad \forall \beta>0.
\]
Letting $\beta \downarrow 0$ finishes the proof.

\emph{Case 3:} $f \equiv 0$. As usual, this case is more delicate. Let $M = \max_{x \in \Oc} \chi(x)$, we first claim that it is enough to prove the result in the case that $v$ satisfies $v(x) \le M$. To see this, define
\[
  \check{v}(x) = \min\{ v(x), M \} \le M.
\]
This function verifies $\bT[\check{v}; 0, \chi] \ge 0$ if $\bT[v; 0, \chi] \ge 0$. Consequently, if we know the comparison holds between $w$ and $\check{v}$, then we can conclude that $w \le \check{v} \le v$.

Let us assume then that $v$ is bounded from above. Owing to the fact that $v \in \LSC(\Oc)$, it is also bounded from below and as a consequence $v \in L^{\infty}(\Omega)$. Now, since $0 \le \bT[v;0, \chi]$, we must have that $v(x) \ge \chi(x)$ and $-\Delta_\infty^\diamond v(x) \ge 0$ in the viscosity sense. In other words, we have that $v$ is bounded and infinity superharmonic. This implies that $v \in \LSC(\Oc) \cap C(\Omega) \cap W^{1,\infty}_{\mathrm{loc}}(\Omega)$; see \cite[Lemma 4.1]{MR1605682}.

We now wish to construct, in the spirit of \cite[Theorem 2.1]{MR2372480}, an approximation to this function which is not flat, i.e., its local Lipschitz constant (defined in \eqref{eq:local-Lip}) is bounded from below. The main difficulty here, and the reason why we cannot simply invoke \cite[Theorem 2.1]{MR2372480} is that $v$ may not be continuous near the boundary. Nevertheless, we claim that, for any $\gamma > 0$, there exists $v_{\gamma} \in \LSC(\Oc) \cap W^{1,\infty}_{\mathrm{loc}}(\Omega)$ which satisfies following properties:
\begin{equation}\label{eq:obstacle-CP-proof-vg}
  -\Delta_{\infty}^{\diamond} v_{\gamma} \ge 0, \text{ in } \Omega, \quad
  v_{\gamma} \ge v, \text{ on } \partial \Omega, \quad
  L(v_{\gamma},\cdot) \ge \gamma, \text{ in } \Omega, \quad
  v \le v_{\gamma} \le v + C\gamma , \text{ in } \Omega,
\end{equation}
where the constant $C>0$ depends only on $\Omega$.
%

Let us sketch the construction of this $v_\gamma$. For $\gamma>0$ define
\[
  V_{\gamma} = \{x \in \Omega: L(v, x) < \gamma \},
\]
which, by definition, is open. Write $V_{\gamma} = \bigcup_{i} \mathcal{O}_i$ where each $\mathcal{O}_i$ is a connected component of $V_{\gamma}$. On each $\mathcal{O}_i$ the function $v$ is uniformly Lipschitz and, consequently, we can define
\[
  \wt{v}_i(x) = \lim_{\mathcal{O}_i \ni y \to x} v(y), \qquad x \in \overline{\mathcal{O}_i}.
\]
Clearly, $\wt{v}_i = v$ in $\overline{\mathcal{O}_i} \cap \Omega$, and for $x \in \overline{\mathcal{O}_i} \cap \pO$, we have $\wt{v}_i(x) \ge v(x)$ because of $v \in \LSC(\Oc)$. We then define the function $v_{\gamma}$ through
\[
  v_{\gamma}(x) = 
  \begin{dcases}
    \inf_{y \in \partial \mathcal{O}_i} \l[ \wt{v}_i(y) + \gamma \dist_{\mathcal{O}_i}(x,y) \r],  & x \in \mathcal{O}_i, \\
    v(x), & x \in \Oc \setminus V_{\gamma},
  \end{dcases}
\]
where by $\dist_{\mathcal{O}_i}(x,y)$ we denote the path distance, relative to $\mathcal{O}_i$, between the points $x$ and $y$. Clearly $L(v_\gamma,x) \geq \gamma$ for all $x \in \Omega$. Following the proof in \cite{MR2372480} one can verify all remaining properties in \eqref{eq:obstacle-CP-proof-vg}.

Once we have $v_{\gamma}$, for any $a > 0$, we define the quadratic function
\[
  q_a(t) = a \l(\Vert v_{\gamma} \Vert_{L^{\infty}(\Omega)}^2 + 1\r) + t - a t^2.
\]
It is possible to choose $a$ small enough, so that $q_a'(t) = 1 - 2at > 0$ for all $t \in[0, M + C \gamma]$. Consider now the function $q_a(v_{\gamma}) \in \LSC(\Oc)$. We have that $q_a(v_{\gamma}) \ge v_{\gamma}$, and that
\[
  \bT[q_a(v_{\gamma}); 0, \chi] \ge \min\{ 2a \gamma^2 , a \} > 0.
\]
To see this, let us do a formal calculation, which can be justified with the usual arguments. If $v_{\gamma}$ is smooth and with $|\De v_{\gamma}(x)| = L(v_{\gamma}, x) \ge \gamma$, we have that
\[
  \De q_a(v_{\gamma}) = q_a'(v_{\gamma}) \De v_{\gamma}, \quad
  \De^2 q_a(v_{\gamma}) = q_a'(v_{\gamma}) \De^2 v_{\gamma} + q_a''(v_{\gamma}) \De v_{\gamma} \otimes \De v_{\gamma},
\]
and consequently
\[
  -\Delta_{\infty}^{\diamond} q_a(v_{\gamma}) = -q_a'(v_{\gamma}) \Delta_{\infty}^{\diamond} v_{\gamma} - q_a''(v_{\gamma}) |\De v_{\gamma}|^2
    \ge - q_a''(v_{\gamma}) |\De v_{\gamma}|^2 \ge 2a |\De v_{\gamma}|^2 \geq 2a \gamma^2,
\]
where we used $q_a'(v_{\gamma}) > 0$ and $-\Delta_{\infty}^{\diamond} v_{\gamma} \ge 0$. Therefore, we can apply \Cref{lemma:obstacle-CP-strict} between $w$ and $q_a(v_{\gamma})$ to conclude that, for $a > 0$ but small enough,
\[
  w(x) \le q_a(v_{\gamma}) = a \l(\Vert v_{\gamma} \Vert_{L^{\infty}(\Omega)}^2 + 1\r) + v_{\gamma}(x) - a v_{\gamma}(x)^2.
\]
Letting $a \downarrow 0$ implies that
\[
  w(x) \le v_{\gamma}(x),
\]
for any $\gamma > 0$. Finally, taking $\gamma \downarrow 0$ finishes our proof.
\end{proof}

\subsubsection{The numerical scheme}\label{sub:SchemeObstacle}

Following the notation and constructions of \Cref{sec:Notation}, the discretization simply reads
\begin{equation}\label{eq:obstacle-discrete}
  \bT_{\frakh}[u_\frakh;f,\chi](\vertex)  = 0, \quad \forall \vertex \in \Nhi{\ve}, \qquad u_\frakh(\vertex) = \wt{g}_\ve(\vertex), \quad \forall \vertex \in \Nhb{\ve},
\end{equation}
where
\begin{equation}
\label{eq:DefOfDiscreteTObstacle}
  \bT_{\frakh}[u_\frakh;f,\chi](\vertex) = \min\l\{ -\Delta_{\infty, \frakh}^\diamond u_\frakh (\vertex) - f(\vertex), u_\frakh (\vertex) - \chi(\vertex)  \r\}, \quad \vertex \in \Nhi{\ve}.
\end{equation}

Let us now prove a discrete comparison principle for this scheme.

\begin{Theorem}[discrete comparison]\label{thm:obstacle-DCP}
Assume that the obstacle $\chi$ satisfies \textbf{OBS}.\ref{Ass.OBS.cont}, and the right hand side $f$ satisfies \textbf{RHS}.\ref{Ass.RHS.f<>0}. Moreover, assume that if \textbf{RHS}.\ref{Ass.RHS.zero} holds, then the mesh $\Th$ satisfies \textbf{M}.\ref{Ass.M.connect}. Let $w_h, v_h \in \Vh$ be such that
\begin{equation}\label{eq:assump-thm-obstacle-DCP}
  \bT_{\frakh}[w_h;f,\chi](\vertex) \le 0 \le \bT_{\frakh}[v_h;f,\chi](\vertex), \quad \forall \vertex \in \Nhi{\ve},
\end{equation}
and $w_h(\vertex) \le v_h(\vertex)$ for any $\vertex \in \Nhb{\ve}$. Then we have
\[
  w_h(\vertex) \le v_h(\vertex) \quad \forall \vertex \in \Nh.
\]	
\end{Theorem}
\begin{proof}
The proof proceeds in a similar way to \Cref{thm:simple-DCP} or \Cref{thm:DCP}, depending on how \textbf{RHS}.\ref{Ass.RHS.f<>0} is satisfied. We assume, for the sake of contradiction, that
\[
  m = \max_{\vertex \in \Nh} \left[ w_h(\vertex) - v_h(\vertex) \right] = \max_{\vertex \in \Nhi{\ve}} \left[ w_h(\vertex) - v_h(\vertex) \right] > 0.
\]
Consider the set of points
\[
  E = \{ \vertex \in \Nh: (w_h - v_h)(\vertex) = m \} \subset \Nhi{\ve}.
\]
Clearly, the set $E$ is nonempty because of the definition of $m$. For any $\vertex \in E$, we have
\begin{equation}\label{eq:obstacle-DCP-proof0}
  w_h(\vertex) - v_h(\vertex) \ge w_h(\othervertex) - v_h(\othervertex)\quad \forall \othervertex \in \Nh. 
\end{equation}
By \Cref{lemma:monotonicity}, this implies that
\[
  S_{\frakh}^+ w_h(\vertex) \le S_{\frakh}^+ v_h(\vertex), \qquad
  S_{\frakh}^- w_h(\vertex) \ge S_{\frakh}^- v_h(\vertex).
\]
and thus $\bT_{\frakh}[w_h;f,\chi](\vertex) \ge \bT_{\frakh}[v_h;f,\chi](\vertex)$. Combining with \eqref{eq:assump-thm-obstacle-DCP} we see that
\[
  \bT_{\frakh}[w_h;f,\chi](\vertex) = \bT_{\frakh}[v_h;f,\chi](\vertex).
\]

Now we claim that
\begin{equation}\label{eq:obstacle-DCP-proof1}
  \bT_{\frakh}[w_h;f,\chi](\vertex) = -\Delta_{\infty, \frakh}^\diamond w_\frakh (\vertex) - f(\vertex).
\end{equation}
If this is not the case, then
\[
  \bT_{\frakh}[w_h;f,\chi](\vertex) = w_h(\vertex) - \chi(\vertex) = v_h(\vertex) + m - \chi(\vertex) > v_h(\vertex) - \chi(\vertex) \ge \bT_{\frakh}[v_h;f,\chi](\vertex),
\]	
which is a contradiction. From \eqref{eq:obstacle-DCP-proof1} then we have
\[
  -\Delta_{\infty, \frakh}^\diamond w_h (\vertex) \le f(\vertex) \le -\Delta_{\infty, \frakh}^\diamond v_h (\vertex), \quad \forall \vertex \in E.
\]

Applying \Cref{thm:simple-DCP} or \Cref{thm:DCP} where $\Nhi{\ve}$ is our set $E$ here we obtain the contradiction
\[
  m \le \max_{\vertex \in \Nh} \left[ w_h(\vertex) - v_h(\vertex) \right] = \max_{\vertex \in \Nh \setminus E}\left[ w_h(\vertex) - v_h(\vertex) \right] < m.
\]
We point out that the Assumption~\textbf{M}.\ref{Ass.M.connect} required in \Cref{thm:DCP} still holds since $E \subset \Nhi{\ve}$ . This finishes our proof.
\end{proof}

Similar to the discussions in \Cref{sub:ExistenceUniqueness}, using the discrete comparison principle we obtain the following properties for the obstacle problem.

\begin{Lemma}[existence, uniqueness, stability]
	Assume \textbf{RHS}.\ref{Ass.RHS.cont}--\ref{Ass.RHS.f<>0},  \textbf{BC}.\ref{Ass.BC.cont}--\ref{Ass.BC.ext} and \textbf{OBS}.\ref{Ass.OBS.cont}. Moreover, if \textbf{RHS}.\ref{Ass.RHS.zero} holds, suppose that for all $\frakh$ we have \textbf{M}.\ref{Ass.M.connect}. For any choice of parameters $\frakh$, there exists a unique $u_{\frakh} \in \Vh$ that solves \eqref{eq:obstacle-discrete}. Moreover, this solution satisfies
\begin{equation}\label{eq:obstacle-stability}
\Vert \uve \Vert_{L^{\infty}(\Oh)} \le C \Vert f \Vert_{L^{\infty}(\Oh)} + \Vert \chi \Vert_{L^{\infty}(\Oh)} + \max_{\vertex \in \Nhb{\ve}} |\wt{g}_\ve(\vertex)|,
\end{equation}
\end{Lemma}
\begin{proof}
We can prove the stability result \eqref{eq:obstacle-stability} which is an analogue of \Cref{lemma:monotonicity}:
\begin{enumerate}[$\bullet$]
  \item The lower bound of $\uve$ is automatically given by $\chi$. 

  \item The upper bound of $\uve$ is obtained using the discrete comparison principle of \Cref{thm:obstacle-DCP} and the barrier function $\varphi(x) = \bv{p}^\intercal x - \frac12 |x|^2 + A$ constructed in the proof of \Cref{lemma:monotonicity}. The only modification needed here is to choose the constant $A$ large enough so that $\varphi \ge \chi$.
\end{enumerate}
The existence and uniqueness then follows from the discrete comparison principle and the stability result as in \Cref{lemma:ExistUnique}.
\end{proof}

The consistency properties proved in \Cref{lemma:consistency} and \Cref{lemma:consistency-p0} imply the corresponding consistency results of the operator $\bT_{\frakh}$. Based on this, we have the following convergence result.

\begin{Theorem}[convergence]\label{thm:obstacle-convergence}
	Assume that the right hand side $f$ satisfies Assumptions \textbf{RHS}.\ref{Ass.RHS.cont}--\ref{Ass.RHS.f<>0}. Assume that the boundary datum $g$ satisfies Assumptions \textbf{BC}.\ref{Ass.BC.cont}--\ref{Ass.BC.ext}. 
	Assume the obstacle $\chi$ satisfies Assumption \textbf{OBS}.\ref{Ass.OBS.cont}.
	Let $\{\frakh_j\}_{j=1}^\infty$ be a sequence of discretization parameters that satisfies \eqref{eq:para-condition}; and, if \textbf{RHS}.\ref{Ass.RHS.zero} holds, it additionally satisfies \textbf{M}.\ref{Ass.M.connect} for all $j \in \mathbb{N}$. In this framework we have that the sequence $\{u_{\frakh_j} \in \mathbb{V}_{h_j}\}_{j=1}^\infty$ of solutions to \eqref{eq:obstacle-discrete} converges uniformly, as $j \uparrow \infty$, to $u$, the solution of \eqref{eq:ObstacleProblem}.
\end{Theorem}
\begin{proof}
	We could almost repeat the proofs of \Cref{lemma:bdry} and \Cref{thm:convergence}. The only minor modification needed is in the proof of the boundary behavior $\overline{u}(x) = g(x)$ for all $x \in \pO$. To be more specific, in the construction of the discrete supersolution $\varphi$, one has to further require that $\varphi(x) \ge \chi(x)$ for any $x \in \Omega$. 
	To this aim, we use a similar argument as in the proof of \Cref{lemma:bdry} and assume without loss of generality that
	$\chi$ is smooth. Then the same calculations we did to guarantee \eqref{eq:bdry-behavior-proof2} give a similar sufficient condition needed for $\varphi$ and thus concludes the proof.
\end{proof}

We obtain the following error estimates as analogues of \Cref{Thm:inhomo-error} and \Cref{Thm:homo-error}.

\begin{Theorem}[error estimate for the inhomogeneous problem]\label{Thm:obstacle-inhomo-error}
	Let $u$ be the viscosity solution of \eqref{eq:ObstacleProblem} and $\uve$ be the solution of \eqref{eq:obstacle-discrete}. Under our running assumptions suppose that \textbf{RHS}.\ref{Ass.RHS.sign} holds, and  that $\frakh$ is sufficiently small, and such that
	\[
	\frac{ (\ve\theta)^\alpha + h^\alpha}{\ve^2},
	\]
	is small enough. Then
	\[
\Vert u-\uve \Vert_{L^{\infty}(\Omega_h)} \lesssim \ve^{\alpha} + \frac{\theta^{\alpha}}{\ve^{2-\alpha}} + \frac{h^{\alpha}}{\ve^2},
	\]
	where the implied constant depends on the dimension $d$, the domain $\Omega$, the shape regularity of the mesh $\Th$, $\min_{x \in \Oc} |f(x)|$, and the $C^{0,\alpha}$ norms of the data $f$, $g$, $\chi$ and the solution $u$.
\end{Theorem}
\begin{proof}
For convenience, we assume that $\min_{x \in \Oc}f(x) > 0$. To get a lower bound for $\uve$, we consider
\[
u_h^- = (1 - I_1) u^{\ve} + I_1 \min_{\vertex \in \Nh } u(\vertex) 
- (2\ve)^{\alpha} |\wt{g}_\ve|_{C^{0,\alpha}(\Oc)} - \Vert g - \wt{g}_{\ve} \Vert_{L^{\infty}(\pO)}
- (3\ve)^{\alpha} |u|_{C^{0,\alpha}(\Oc)} - \ve^{\alpha} |\chi|_{C^{0,\alpha}(\Oc)},
\]
where $u^\ve$ is defined in \eqref{eq:DefOfUSuperEps}, and $I_1$ is defined in \eqref{eq:inhomo-error-proof-I1}. To show that $u_h^-$ is a discrete subsolution in the sense
\[
  \bT_{\frakh}[u_h^-;f,\chi](\vertex) \le 0, \quad \forall \vertex \in \Nhi{\ve}, \qquad u_h^-(\vertex) \le \wt{g}_\ve(\vertex), \quad \forall \vertex \in \Nhb{\ve},
\]
we discuss two different scenarios based on the distance between $\vertex$ and $ \mathcal{C}(u)$.

If $B_{2\ve}(\vertex) \subset \Omega^+(u)$, then $-\Delta_\infty^\diamond u(x) \le f(x)$ for any $x \in B_{2\ve}(\vertex)$. 
The proof of \Cref{Thm:inhomo-error} immediately implies that
\[
\bT_{\frakh}[u_h^-;f,\chi](\vertex) \le -\Delta_{\infty, \frakh}^\diamond u_h^-(\vertex) - f(\vertex) \le 0.
\]
	
If $B_{2\ve}(\vertex) \cap \mathcal{C}(u) \neq \emptyset$, then there exists $z \in B_{2\ve}(\vertex)$ such that $u(z) = \chi(z)$. Hence,
\[
\begin{aligned}
u^{\ve}(\vertex) - \chi(\vertex)
&\le u(y) - u(z) + \chi(z) - \chi(\vertex)
\le |y - z|^{\alpha} |u|_{C^{0,\alpha}(\Oc)} + |z - \vertex| |\chi|_{C^{0,\alpha}(\Oc)} \\
& \le (2 \ve)^{\alpha} |u|_{C^{0,\alpha}(\Oc)}
+ \ve^{\alpha} |\chi|_{C^{0,\alpha}(\Oc)},
\end{aligned}
\]
where $y \in \overline{B}_\ve(\vertex)$ such that $u(y) = u^{\ve}(\vertex)$. From this we simply see that
\[
\bT_{\frakh}[u_h^-;f,\chi](\vertex) \le u_h^-(\vertex) - \chi(\vertex)
\le u^{\ve}(\vertex) -  (2 \ve)^{\alpha} |u|_{C^{0,\alpha}(\Oc)}
- \ve^{\alpha} |\chi|_{C^{0,\alpha}(\Oc)} - \chi(\vertex)
 \le 0.
\]
To show the inequality on the boundary nodes, we recall \eqref{eq:bdry-gve-u} and get
\[
\begin{aligned}
u_h^-(\vertex) - \wt{g}_{\ve}(\vertex)
&\le u^{\ve}(\vertex) - \wt{g}_{\ve}(\vertex) - (2\ve)^{\alpha} |\wt{g}_\ve|_{C^{0,\alpha}(\Oc)} - \Vert g - \wt{g}_{\ve} \Vert_{L^{\infty}(\pO)}
- (3\ve)^{\alpha} |u|_{C^{0,\alpha}(\Oc)}
\le 0,
\end{aligned}
\]
The discussions above prove that $u_h^-$ is a discrete subsolution and thus gives a lower bound of $u_\frakh$ by \Cref{thm:obstacle-DCP}.

To obtain the upper bound of $u_h$, we consider
\[
u_h^+ = \frac{1}{1 - I_1} u_{\ve} - \frac{I_1}{1 - I_1} \min_{\vertex \in \Nh } u(\vertex) + (2\ve)^{\alpha} |\wt{g}_\ve|_{C^{0,\alpha}(\Oc)} + \Vert g - \wt{g}_{\ve} \Vert_{L^{\infty}(\pO)}
+ (3\ve)^{\alpha} |u|_{C^{0,\alpha}(\Oc)}
\]
where, in full analogy to \eqref{eq:DefOfUSuperEps}, we defined $u_{\ve}(x) = \min_{y \in \overline{B}_\ve(x)} u(y)$. A similar calculation shows that $u_h^+(\vertex) \ge \wt{g}_{\ve}(\vertex)$ for any $\vertex \in \Nhb{\ve}$.
From the facts that $-\Delta_\infty^\diamond u(x) \ge f(x)$ and $u(x) \ge \chi(x)$ for any $x\in \Omega$ we similarly derive that
\[
-\Delta_{\infty, \frakh}^\diamond u_h^+(\vertex) - f(\vertex) \ge 0, \quad u_h^+(\vertex) - \chi(\vertex) \ge 0
\]
for any $\vertex \in \Nhi{\ve}$. These inequalities prove that $u_h^+$ is a discrete supersolution and thus it provides an upper bound for $\uve$. Combining the lower and upper bound together we have
\[
\Vert u-\uve \Vert_{L^{\infty}(\Omega_h)} \lesssim \ve^{\alpha} + \frac{\theta^{\alpha}}{\ve^{2-\alpha}} + \frac{h^{\alpha}}{\ve^2}. \qedhere
\]
\end{proof}

\begin{Theorem}[error estimate for the homogeneous problem]\label{Thm:obstacle-homo-error}
	Let $u$ be the viscosity solution of \eqref{eq:ObstacleProblem} and $\uve$ be the solution of \eqref{eq:obstacle-discrete}. Under our running assumptions suppose that \textbf{RHS}.\ref{Ass.RHS.zero} and, for all $h>0$, \textbf{M}.\ref{Ass.M.connect} holds. If $\frakh$ is sufficiently small, such that $\ve - \ve \theta - h > 0$, and
	\[
	(2h + \ve \theta)/\ve^2,
	\]
	can be made sufficiently small, then we have
	\[
	\Vert u- \uve \Vert_{L^{\infty}(\Omega_h)} \lesssim \ve^{\alpha}  + \frac{\sqrt{2h + \ve \theta}}{\ve},
	\]
	where the implied constant depends on the dimension $d$, the domain $\Omega$, the shape regularity of the mesh $\Th$, and the $C^{0,\alpha}$ norms of the data $g$, $\chi$ and the solution $u$.
\end{Theorem}
\begin{proof}
We have already shown in the proof of \Cref{Thm:obstacle-inhomo-error}
how to modify the error estimates in \Cref{sec:Rates} to obtain the error estimates for the obstacle problem. One different thing here is that in the construction of the discrete subsolution $u_h^-$, one needs to employ the approximation in $\Omega^+(u)$ instead of all of $\Omega$. To be more specific, we introduce $v_{\gamma} \in C(\overline{\Omega}) \cap W^{1,\infty}_{\mathrm{loc}}(\Omega^+(u))$ with the following properties:
\begin{align*}
  -\Delta_{\infty}^{\diamond} v_{\gamma} &\le 0, \text{ in } \Omega^+(u), 
  &v_{\gamma} &= u, \text{ on } \partial \Omega^+(u), \\
  L(v_{\gamma},\cdot) &\ge \gamma, \text{ in } \Omega^+(u),
  &u - \gamma \diam(\Omega) &\le v_{\gamma} \le u, \text{ in } \Omega^+(u).
\end{align*}
This is because we only have $-\Delta_{\infty}^{\diamond} u \le 0$ in $\Omega^+(u)$. However, since $-\Delta_{\infty}^{\diamond} u \ge 0$ in $\Omega$, the approximation needed in the construction of the supersolution is still in $\Omega$ as before.
Repeating some arguments in the proof of \Cref{Thm:homo-error} and \Cref{Thm:obstacle-inhomo-error} we know that
\[
\begin{aligned}
u_h^- &= q(v_{\gamma}^{\ve + h}) - C\l( (M - m)^2 + \diam(\Omega)\r) \sqrt{\beta} \\
&- C\ve^{\alpha} \l(  |\wt{g}_\ve|_{C^{0,\alpha}(\Oc)} + |u|_{C^{0,\alpha}(\Oc)} + |\chi|_{C^{0,\alpha}(\Oc)} \r)
- C\Vert g - \wt{g}_{\ve} \Vert_{L^{\infty}(\Omega)}
\end{aligned}
\]
is a discrete subsolution where the definitions of $m, M, q, \beta$ can be found in the proof of \Cref{Thm:homo-error}. This gives a lower bound of $\uve$. Combining it with the upper bound obtained in a similar fashion as before proves the desired error estimate.
\end{proof}

Rates of convergence, analogous to those of \Cref{cor:RatesInhom} and \Cref{cor:rates.f.zero}, can also be obtained. We omit the details. 

\subsection{Symmetric Finsler norms}\label{sub:Finsler}

Let us consider one last variation of the game described in \Cref{sub:TugOfWar}. Namely, the points where the token can be moved to do not need to constitute an (Euclidean) ball. Indeed, we can more generally consider rescaled translations of a convex set $B$ with $0\in B$. In this case, the game value should satisfy
\begin{equation}
\label{eq:FinslerGame}
  2u^\ve(x) -  \left( \sup_{y \in x + \ve B} u^\ve(y) + \inf_{y \in x + \ve B} u^\ve(y) \r) = \ve^2 f(x),
\end{equation}
where
\[
  x+\ve B = \l\{ x + \ve \bv{v}: \bv{v} \in B \r\}.
\]
It is then of interest to understand what is the ensuing differential equation, and its properties.

To proceed further with the discussion of what the limiting differential equation is, we must restrict the class of admissible convex sets to those that can be characterized as the unit ball of a (symmetric) Finsler--Minkowski norm.

\begin{Definition}[Minkowski norm]\label{def:MinkNorm}
A (symmetric) Finsler--Minkowski norm on $\mRd$ is a function $F: \mRd \to \mR$ that satisfies
\begin{enumerate}[$\bullet$]
  \item \textbf{Smoothness}: $F \in C^2(\mR^d\setminus\{\boldsymbol0\})$.
  
  \item \textbf{Absolute homogeneity}: For all $\bv{v} \in \mRd$ and $\lambda \in \mR$ we have $F(\lambda\bv{v}) = |\lambda| F(\bv{v})$.
  
  \item \textbf{Strong convexity:} For all $\boldsymbol0 \neq \bv{v} \in \mR^d$ the matrix $\De^2 F(\bv{v})$ is positive definite.
\end{enumerate}
\end{Definition}

\begin{Remark}[symmetry]
In \Cref{def:MinkNorm}, the qualifier symmetric comes from the fact that, as a consequence of absolute homogeneity we have $F(-\bv{v})= F(\bv{v})$. A general Minkowski norm needs not to satisfy absolute homogeneity, but only positive homogeneity, i.e., $F(t\bv{v}) = t F(\bv{v})$ for all $\bv{v} \in \mR^d$ and $t>0$. Notice that, in this case, the function $F $ is not necessarily symmetric. While to a certain extent it is possible, see \cite{MR3433957,MR3992461,MR3995808,MR4095548,MR4279064}, to develop a theory for nonsymmetric Finsler norms we will not consider this case here. The reason for this restriction is detailed below.
\end{Remark}

We refer the reader to \cite{MR1747675} for a thorough treatise regarding Finsler structures and Finsler geometry. Here we will just mention, in addition to \Cref{def:MinkNorm}, the so--called dual function of $F$, which is given by
\[
  F^*(\bv{w}) = \sup_{\boldsymbol{0} \neq \bv{v}  \in\mRd} \frac{ \bv{w}^\intercal \bv{v}}{F(\bv{v})}.
\]
Notice that the dual is also a symmetric Finsler--Minkowski norm. Moreover, by positive homogeneity
\[
  \bv{v}^\intercal \De F(\bv{v}) = F(\bv{v}), \qquad \forall \bv{v} \in \mRd.
\]
From this, several important relations between $F$ and its dual $F^*$ follow; see, for instance, \cite[Lemma 3.5]{MR3995808}.

With these definitions at hand we can now describe the equation that ensues from \eqref{eq:FinslerGame} in, at least, a particular case. Given a symmetric Finsler--Minkowski norm $F$, assume that the set of admissible moves is given by
\begin{equation}
\label{eq:FinslerBall}
  B_F = \left\{ \bv{v} \in \mRd: F^*(\bv{v}) \leq 1 \right\},
\end{equation}
where $F^*$ is the dual of $F$. Then, as \cite[Theorem 1.3 and Corollary 1.4]{MR4279064} show, 
equation \eqref{eq:FinslerGame}, in the formal limit $\ve \downarrow0$, becomes
\begin{equation}
\label{eq:DirProblFinslerInfLap}
  -\Delta_{F,\infty}^\diamond u = f, \text{ in } \Omega, \qquad u_{|\pO} = g.
\end{equation}
The operator $-\Delta_{F,\infty}^\diamond$, called for obvious reasons the (normalized) Finsler infinity Laplacian, is
\[
  -\Delta_{F,\infty}^\diamond \varphi(x) =-\De F(\De \varphi(x) ) \otimes \De F(\De \varphi(x) ): \De^2 \varphi(x) .
\]
Notice that, once again, this operator is singular in the case that $\De \varphi(x) = \boldsymbol0$. For completeness we present the upper and lower semicontinuous envelopes of this operator:
\begin{align*}
  \Delta_{F,\infty}^+ \varphi(x) &= \begin{dcases}
                                      \De F(\De \varphi(x) ) \otimes \De F(\De \varphi(x) ): \De^2 \varphi(x), & \De \varphi(x) \neq \boldsymbol0, \\
                                      \max\l\{ \bv{v}^\intercal \De^2 \varphi(x) \bv{v} : F^*(\bv{v}) = 1 \r\}, & \De \varphi(x) = \boldsymbol0,
                                    \end{dcases}
   \\
  \Delta_{F,\infty}^- \varphi(x) &= \begin{dcases}
                                      \De F(\De \varphi(x) ) \otimes \De F(\De \varphi(x) ): \De^2 \varphi(x), & \De \varphi(x) \neq \boldsymbol0, \\
                                      \min\l\{ \bv{v}^\intercal \De^2 \varphi(x) \bv{v} : F^*(\bv{v}) = 1 \r\}, & \De \varphi(x) = \boldsymbol0.
                                    \end{dcases}
\end{align*}

Problem \eqref{eq:DirProblFinslerInfLap} has been studied in \cite{MR3433957,MR3992461,MR3995808,MR4095548,MR4279064} where, under assumptions \textbf{RHS}.\ref{Ass.RHS.cont}--\ref{Ass.RHS.f<>0} and \textbf{BC}.\ref{Ass.BC.cont}, several results have been obtained. In particular, a comparison principle for semicontinuous functions is obtained  in \cite[Theorem 6.4]{MR3995808}. Existence of solutions is shown in \cite[Theorem 6.1]{MR4095548}. Uniqueness, under assumption \textbf{RHS}.\ref{Ass.RHS.f<>0} is also obtained. In addition to the game theoretic interpretation of \eqref{eq:DirProblFinslerInfLap}, we comment that a Finsler variant of the minimization problem of \Cref{sub:MotivateLinf} was studied in \cite{MR3433957}. Namely, the authors study the minimal Lipschitz extension problem, where the Lipschitz constant is measured with respect to the Finsler norm $F$, i.e.,
\[
  \Lip_F(w,E) = \sup\l\{ \frac{|w(x) - w(y)|}{F(x-y)} : x,y \in E, x\neq y \r\}.
\]
Reference \cite{MR3433957} shows the existence and uniqueness of solutions. A comparison principle for continuous functions is also obtained. Interestingly, the proofs of all the results mentioned here follow by adapting the techniques of \cite{MR2846345}. One needs to use, as semidiscrete scheme, relation \eqref{eq:FinslerGame} with $B$ given as in \eqref{eq:FinslerBall}. Finally, some applications of \eqref{eq:DirProblFinslerInfLap} to global geometry are explored in \cite{MR4216961}.

It is at this point that we are able to describe why we must assume that our Finsler--Minkowski norm is symmetric. The reason is because one of the main tools in the analysis is a comparison principle with quadratic Finsler cones which, in the nonsymmetric case, only holds with a restriction on the coefficients that define the cone; see \cite[Theorem 4.1]{MR4095548}. However, as \cite[Remark 4.3 item (2)]{MR4095548} shows, such a restriction is not needed for the symmetric case. This comparison principle, essentially, allows one to translate all the results of \cite{MR2846345} to the present case.

The investigations detailed above motivate us to present a numerical scheme for \eqref{eq:DirProblFinslerInfLap}. To do so, we follow the ideas of \Cref{sub:TwoScale}. We need then to discretize $\partial B_F$, given in \eqref{eq:FinslerBall}. In full analogy to $\mS_\theta$, we let $\mS_{F,\theta} = \{\bv{v}_\theta\}$ be a finite and symmetric set of vectors such that
\[
  \bv{v}_\theta \in \mS_{F,\theta} \quad \implies \quad F^*(\bv{v}_\theta ) =1 ,
\]
and, for every $\bv{v} \in \mRd$ such that $F^*(\bv{v}) = 1$, there is $\bv{v}_\theta \in \mS_{F,\theta}$ with
\[
  F^*(\bv{v}- \bv{v}_\theta) \leq \theta.
\]
For $\vertex \in \Nhi{\ve}$ we then define
\[
  \mathcal{N}_{F,\frakh}(\vertex) = \{\vertex\}\cup \l\{ \vertex + \ve \bv{v}_\theta : \bv{v}_\theta \in \mS_{F,\theta} \r\},
\]
and, for $w\in C(\Oc)$,
\[
  S_{F,\frakh}^+ w(\vertex) = \frac{1}{\ve} \l( \max_{x \in \mathcal{N}_{F,\frakh}(\vertex)} w(x) - w(\vertex) \r),  \quad 
  S_{F,\frakh}^- w(\vertex) = \frac{1}{\ve} \l( w(\vertex) - \min_{x \in \mathcal{N}_{F,\frakh}(\vertex)} w(x)  \r).
\]
Our fully discrete Finsler infinity Laplacian is defined, for $w \in C(\Oc)$, by
\begin{equation}\label{eq:def-fullydiscrete-op1-Finsler}
  -\Delta_{F,\infty, \frakh}^\diamond w(\vertex) = -\frac1\ve \left( S_{F,\frakh}^+ \interp w(\vertex) - S_{F,\frakh}^- \interp w(\vertex) \right), \qquad \forall \vertex \in \Nhi{\ve}.
\end{equation}
Under Assumption~\textbf{BC}.\ref{Ass.BC.ext} the numerical scheme to approximate the solution to \eqref{eq:DirProblFinslerInfLap} is: Find $u_\frakh \in \Vh$ such that
\begin{equation}\label{eq:inf-laplace-discrete-Finsler}
  -\Delta_{F,\infty, \frakh}^\diamond u_\frakh (\vertex) = 
  f(\vertex), \quad \forall \vertex \in \Nhi{\ve}, \qquad u_\frakh(\vertex) = \wt{g}_\ve(\vertex), \quad \forall \vertex \in \Nhb{\ve}.
\end{equation}

We will not dwell on the properties of solutions to \eqref{eq:inf-laplace-discrete-Finsler}, as these merely repeat all of the arguments we presented in \Cref{sec:ExistenceAndSuch}. We only comment that, during the analysis, one can use in several places the fact that $F$ and $F^*$ are indeed norms on $\mRd$. Norm equivalence with the standard Euclidean norm is useful in obtaining many estimates. In short, the proof of convergence repeats \Cref{thm:convergence}.

\section{Solution schemes}\label{sec:Solution}

Let us now make some comments regarding the practical solution of \eqref{eq:inf-laplace-discrete}. Since this involves at every point computing maxima and minima, it is only natural then to, as a first attempt, apply one of the variants of Howard's algorithm presented in \cite{MR2551155}, or equivalently active set strategies \cite{MR1972219}. All these approaches are related to the fact that computing maxima and minima are slantly differentiable operations \cite[Lemma 3.1]{MR1972219} and, thus, a semismooth Newton method can be formulated. If this is to converge we then have, at least locally, that this convergence is superlinear \cite[Theorem 5.31]{MR3653852}.

However, convergence of such schemes can only be guaranteed if, at every iteration step $n \in \mathbb{N}_0$, the ensuing matrices are nonsingular and monotone. While monotonicity is not an issue, for the problem at hand it is possible to end with singular matrices. The reason behind this is as follows. At each step one needs to solve a linear system of equations which reads
\begin{align*}
  2 u_\frakh^{n+1}(\vertex) - \sum_{\othervertex \in \sigma_+^n(\vertex)} \alpha_\othervertex^n u_\frakh^{n+1}(\othervertex) - \sum_{\othervertex \in \sigma_-^n(\vertex)} \beta_\othervertex^n u_\frakh^{n+1}(\othervertex) &= \ve^2 f(\vertex), & \vertex &\in \Nhi{\ve}, \\
  u_\frakh^{n+1}(\vertex) &= \wt{g}_\ve(\vertex), & \vertex &\in \Nhb{\ve}.
\end{align*}
Here $\sigma_+^n(\vertex),\sigma_-^n(\vertex) \subset \Nh$ are such that
\[
  u_\frakh^n(x^+) = \sum_{\othervertex \in \sigma_+^n(\vertex)} \alpha_\othervertex^n u_\frakh^{n}(\othervertex) = \max_{x \in \mathcal{N}_\frakh(\vertex)} u_\frakh^n(x),
\]
i.e., $\widehat{\varphi}_{\othervertex}(x^+) = \alpha_{\othervertex}^n$ for $\othervertex \in \wt{\mathcal N}_\frakh(\vertex)$ and $x^+ \in \Oc$ is the point where this local maximum is attained. A similar characterization can be made for $\sigma_-^n$. These considerations show that
\[
  \alpha_\othervertex^n, \beta_\othervertex^n \in [0,1], \qquad \sum_{\othervertex \in \sigma_+^n(\vertex)} \alpha_\othervertex^n = \sum_{\othervertex \in \sigma_+^n(\vertex)} \beta_\othervertex^n =1.
\]
In other words, the ensuing system matrix is diagonally dominant, but not strictly diagonally dominant, and thus it may be singular. To conclude nonsingularity the usual argument in this scenario involves invoking boundary nodes: For $\vertex \in \Nhi{\ve}$ that is close to $\pO$ the stencil must contain boundary nodes. However, there is no guarantee that this will be the case in our setting. As a consequence, we cannot guarantee that the matrices will be nonsingular.

\begin{algorithm}
\caption{Fixed point iteration to solve \eqref{eq:inf-laplace-discrete}.}
\label{alg:fixedpt}
  \KwIn{The initial guess: $u_\frakh^0 \in \Vh$ such that $u_\frakh^0 = \wt{g}_\ve$ in $\Nhb{\ve}$.}
  \KwOut{$u_\frakh$ the solution to \eqref{eq:inf-laplace-discrete}.}
  $n=0$\;
  \Repeat{\texttt{false}}{
    \ForEach{ $\vertex \in \Nhi{\ve}$}{
      \[
        u_\frakh^{n+1}(\vertex) = \frac12 \left[ \ve^2 f(\vertex) + \max_{x \in \mathcal{N}_\frakh(\vertex)} u_\frakh^n(x) + \min_{x \in \mathcal{N}_\frakh(\vertex)} u_\frakh^n(x)\right];
      \]
    }
    \ForEach{ $\vertex \in \Nhb{\ve}$}{
      \[
        u_\frakh^{n+1}(\vertex) = \wt{g}_\ve(\vertex);
      \]
    }
    $n\rightarrow n+1$\;
  }
\end{algorithm}

To circumvent this difficulty we, instead, propose the fixed point iteration presented in Algorithm~\ref{alg:fixedpt}. As the following result shows, this fixed point iteration is globally convergent.

\begin{Theorem}[convergence]\label{thm:ConvergenceFixedPoint}
For every $u_\frakh^0 \in \Vh$ such that $u_\frakh^0 = \wt{g}_\ve$ in $\Nhb{\ve}$ the sequence $\{u_\frakh^n\}_{n \in \mathbb{N}_0}$ generated by Algorithm~\ref{alg:fixedpt} converges to $u_\frakh$, the solution to \eqref{eq:inf-laplace-discrete}.
\end{Theorem}
\begin{proof}
We divide the proof in three steps.

First, let us assume that $u_\frakh^0$ is a supersolution, i.e.,
\[
  -\Delta_{\infty,\frakh}^\diamond u_h^0(\vertex) \geq f(\vertex), \quad \forall \vertex \in \Nhi{\ve}, \qquad u_h^0(\vertex) = \wt{g}_\ve(\vertex), \quad \forall \vertex \in \Nhb{\ve}.
\]
Then, by induction, we see that, for all $n \in \mathbb{N}_0$, we have $u_\frakh \leq u_\frakh^{n+1} \leq u_\frakh^n$, and that $\{u_\frakh^n\}_{n \in \mathbb{N}_0}$ is a family of supersolutions. We observe first that if $u_\frakh^n$ is a supersolution, by comparison we must have $u_\frakh \leq u_\frakh^n$. Next, since $u_\frakh^n$ is a supersolution
\[
  u_\frakh^n(\vertex) \geq \frac12 \left[ \ve^2 f(\vertex) + \max_{x \in \mathcal{N}_\frakh(\vertex)} u_\frakh^n(x) + \min_{x \in \mathcal{N}_\frakh(\vertex)} u_\frakh^n(x)\right] = u_\frakh^{n+1}(\vertex), \qquad \forall \vertex \in \Nhi{\ve},
\]
with equality on $\Nhb{\ve}$. Notice that this implies that
\begin{align*}
  u_\frakh^{n+1}(\vertex) &=\frac12 \left[ \ve^2 f(\vertex) + \max_{x \in \mathcal{N}_\frakh(\vertex)} u_\frakh^n(x) + \min_{x \in \mathcal{N}_\frakh(\vertex)} u_\frakh^n(x)\right] \\
    &\geq \frac12 \left[ \ve^2 f(\vertex) + \max_{x \in \mathcal{N}_\frakh(\vertex)} u_\frakh^{n+1}(x) + \min_{x \in \mathcal{N}_\frakh(\vertex)} u_\frakh^{n+1}(x)\right],
\end{align*}
so that $-\Delta_{\infty,\frakh}^\diamond u_\frakh^{n+1}(\vertex) \geq f(\vertex)$, and $u_\frakh^{n+1}$ is a supersolution as well. Consequently, $u_\frakh^{n+1} \geq u_\frakh$. Finally, a standard Perron--like reasoning involving monotonicity and comparison yields then that the limit must be a solution.

Next, we assume that $u_\frakh^0$ is a subsolution. In a similar manner we obtain that the sequence $\{u_\frakh^n\}_{n \in \mathbb{N}_0}$ is a family of subsolutions, that $u_\frakh^n \leq u_\frakh^{n+1} \leq u_\frakh$, and that $u_\frakh^n \to u_\frakh$.

Consider now the general case, i.e., we only assume that $u_\frakh^0 = \wt{g}_\ve$ in $\Nhb{\ve}$. With an argument similar to that of the proof of \Cref{lemma:stability} we can construct $w_h^\pm \in \Vh$ such that $w_h^\pm = \wt{g}_\ve$ in $\Nhb{\ve}$, 
\[
  w_h^- \leq u_\frakh^0 \leq w_h^+,
\]
$w_h^-$ is a subsolution, and $w_h^+$ is a supersolution. The monotonicity arguments of the previous two steps show that
\[
  w_h^{-,n} \leq u_\frakh^n \leq w_h^{+,n},
\]
where $\{w_h^{\pm,n}\}_{n \in \mathbb{N}_0}$ are the sequences obtained by applying Algorithm~\ref{alg:fixedpt} with initial guess $w_h^{\pm,0} = w_h^\pm$, respectively. The previous steps then show that $w_h^{\pm,n} \to u_\frakh$, and this shows convergence in the general case.
\end{proof}

\section{Numerical experiments}\label{sec:Numerics}

In this section we present some simple numerical examples to validate our analysis. All the computations were done with an in-house code that was written in MATLAB$^\copyright$. In practice, we choose a tolerance $\text{TOL} > 0$ and stop the iteration in Algorithm \ref{alg:fixedpt} when
\[
\Vert \{ -\Delta^\diamond_{\infty, \frakh} u_{\frakh}^n(\vertex) - f(\vertex) \} _{\vertex \in \Nhi{\ve} } \Vert_{\ell^{\infty}} < \text{TOL}.
\] 

\begin{Example}[$f \equiv 0$, Aronsson's example]\label{Ex-Aronsson}
Let $\Omega = (-1, 1)^2$, $f \equiv 0$ and $u(x, y) = |x|^{4/3} - |y|^{4/3}$. Notice that, in this example, the right hand side $f$ is smooth and $u \in C^{1,1/3}(\Oc)$.

We first wish to study the error induced by the discretization of the operator. To do so, we use an exact boundary condition, i.e., $\wt{g}_{\ve} = u$. \Cref{fig:Ex-Aronsson-uh} shows the computed solution $u_\frakh$, together with the error $u_\frakh - u$ for $h = 2^{-8}, \ve = 2^{-4.75}$, and $ \theta = 2^{-3.25}$. From the pictures we see that a larger error appears near the coordinate axes, where the solution $u$ is not smooth. One can also observe the radial pattern of the error, which might be a result from the discretization, $\St$, of the unit sphere $\mathbb{S}$.


\begin{figure}[!htb]
	\includegraphics[width=0.45\linewidth]{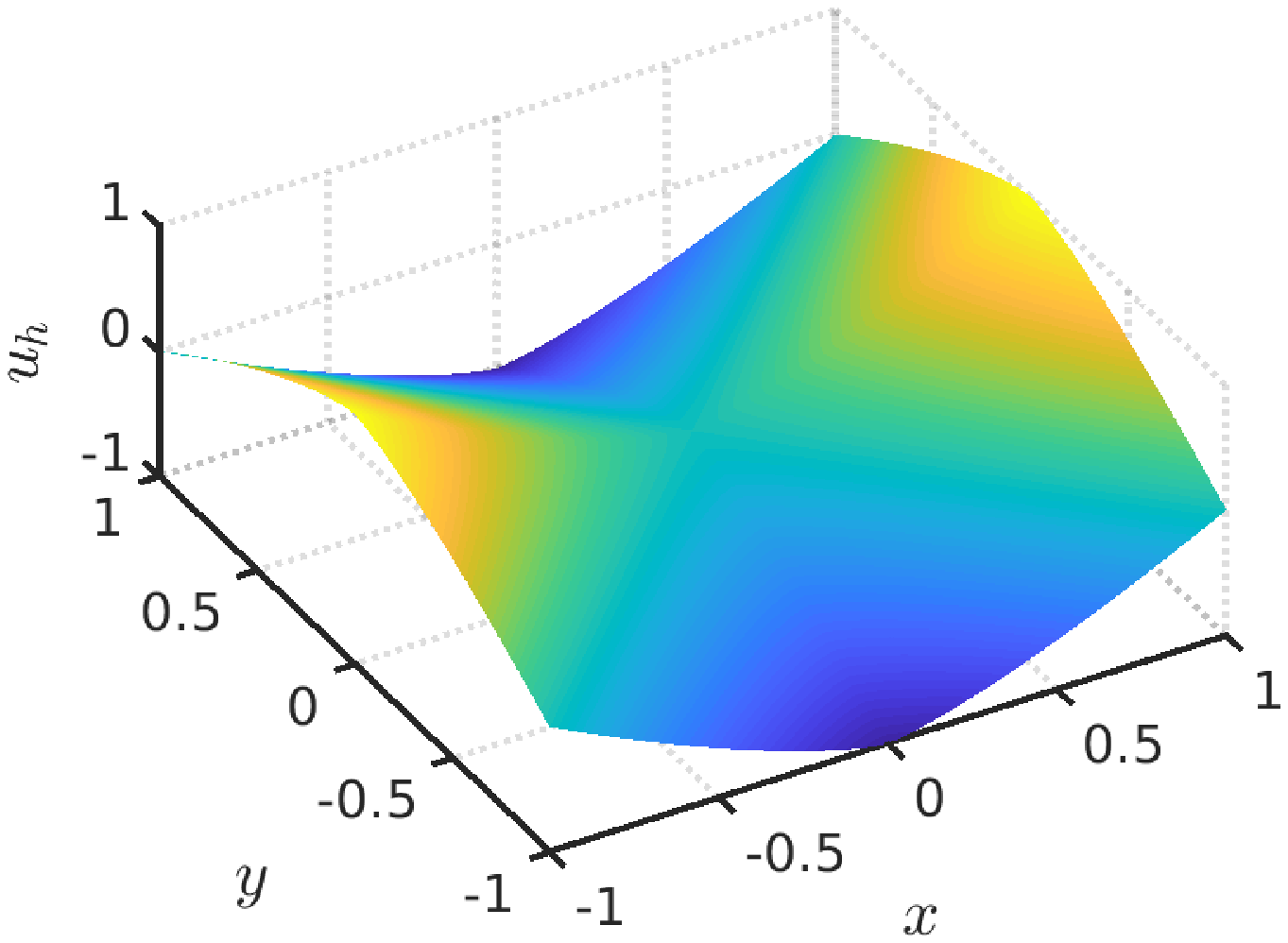}
	\includegraphics[width=0.45\linewidth]{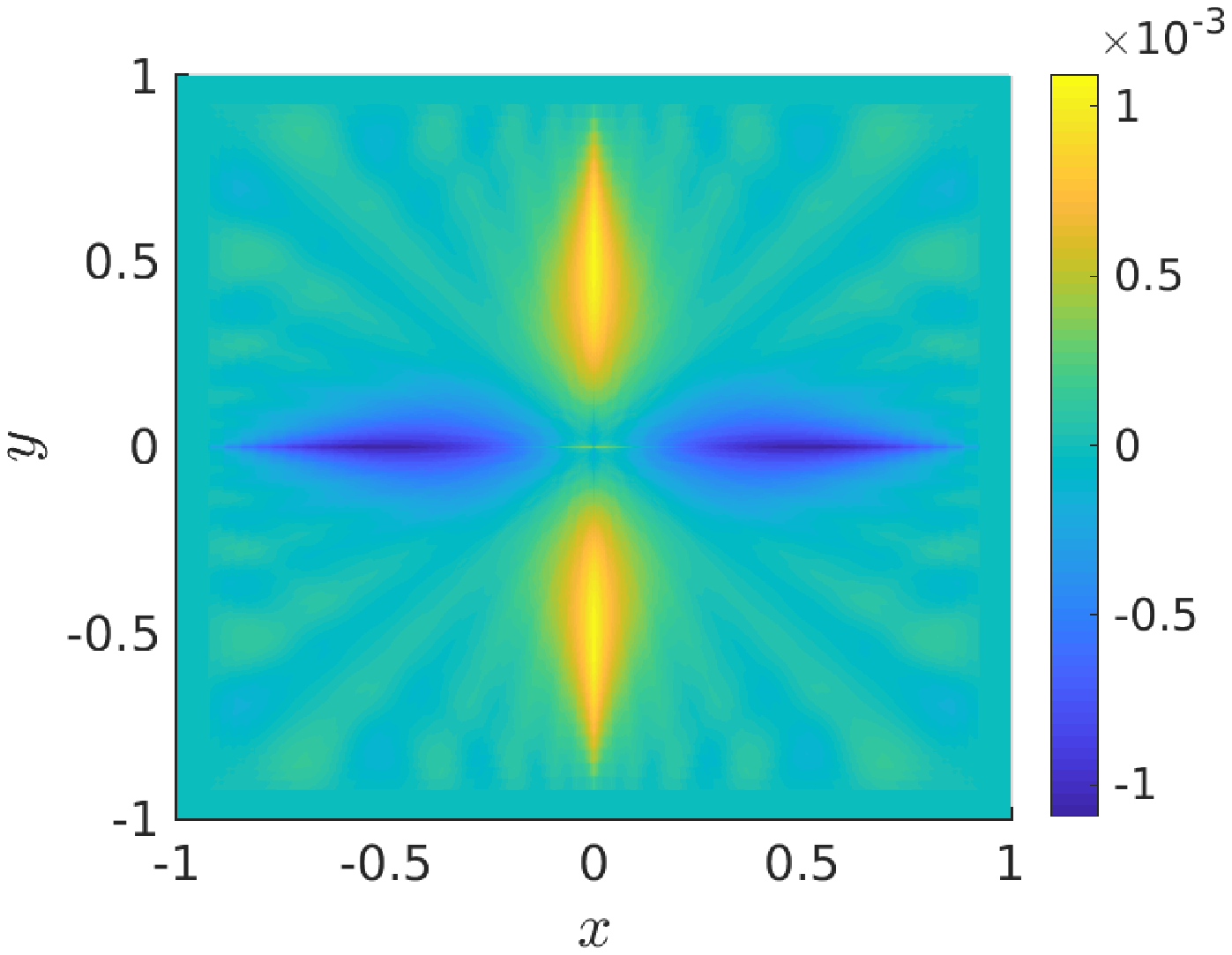}
	\caption{\small \Cref{Ex-Aronsson}. Left: Plot of $u_\frakh$. Right: Plot of $u_\frakh - u$.}
	\label{fig:Ex-Aronsson-uh}
\end{figure}

We now turn our attention to rates of convergence. We compute the approximate solution for a sequence of discretization parameters $\frakh =(h, \ve, \theta)$ satisfying
\begin{equation}\label{eq:exp-parameters}
\ve = C_{\ve} h^{\beta}, \quad \theta = C_{\theta} \frac{h}\ve,
\end{equation}
where the positive constants $C_{\ve}, C_{\theta}$ and the power $\beta$ are chosen manually. We set $C_{\theta} = 1$ and display, in \Cref{fig:Ex-Aronsson-rates}, plots of the $L^{\infty}$ errors vs.~meshsizes $h$ for $\beta \in \{ 0, \tfrac14, \tfrac13, \tfrac12, 1\}$. The corresponding values of $C_{\ve}$ are chosen to guarantee that when $h = 2^{-5}$ we have $\ve = 2^{-4}$ for all the choices of $\beta$. From the plot, we observe that the errors get stuck for small $h$ when $\beta = 0$ or $1$. This is consistent with our theory because in either case the consistency error of our discretization may not tend to zero as $h \to 0$; see \Cref{lemma:consistency}. For $\beta \in \{ \tfrac14, \tfrac13, \tfrac12 \}$, we also measure the convergence orders using a least squares fit for the errors of $h = 2^{-7}, 2^{-8}, 2^{-9}$. The orders are about $0.83, 1.45, 1.02$ respectively, which are much better than the theoretical error estimates \Cref{Thm:homo-error}. This may be caused by the fact that, although $u \in C^{1,1/3}(\Oc)$ only, it is smooth away from the coordinate axes.

\begin{figure}[!htb]
	\includegraphics[width=0.75\linewidth]{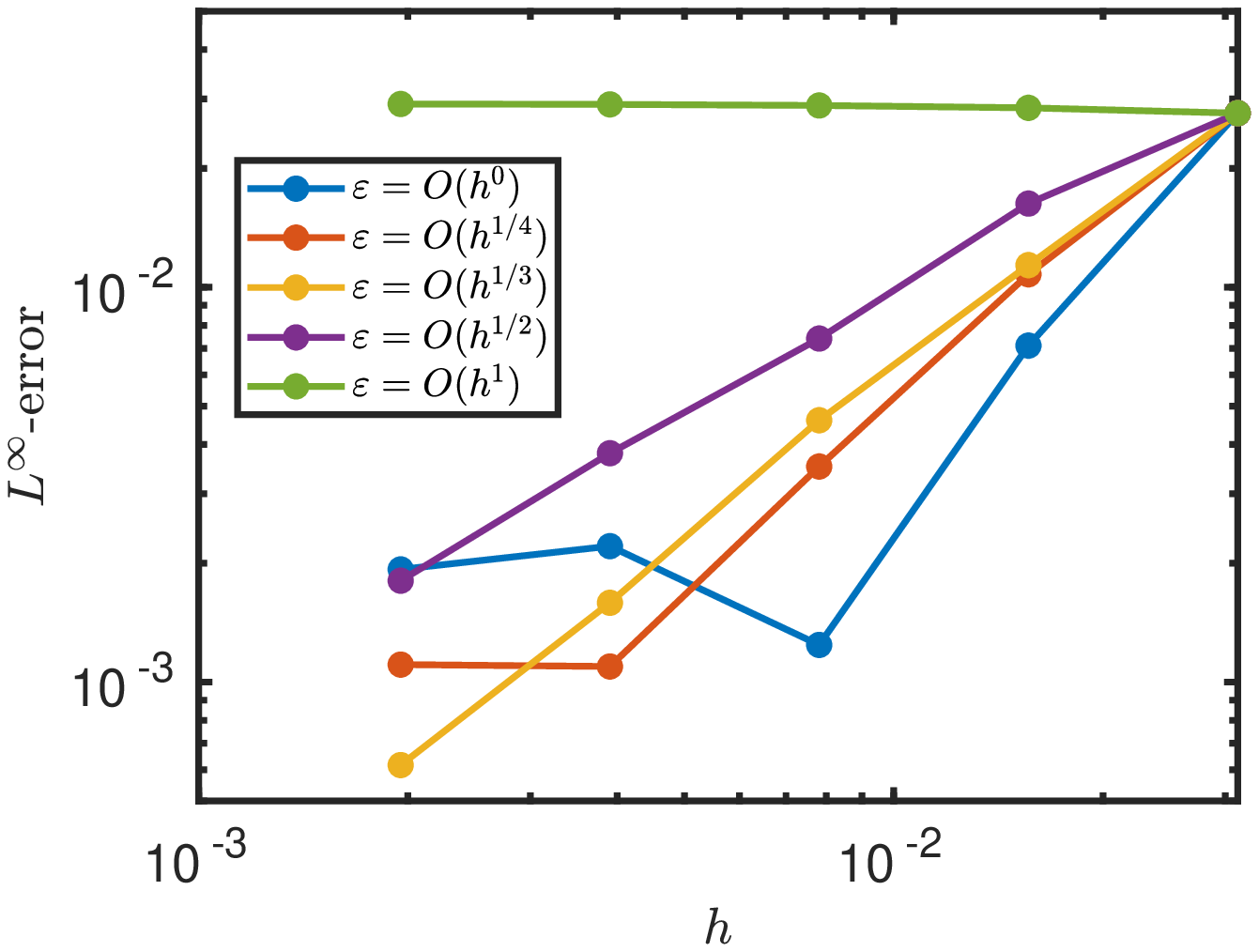}
	\caption{\small \Cref{Ex-Aronsson}. Experimental rates of convergence for $L^{\infty}$ errors with $\beta \in \{ 0, \tfrac14, \tfrac13, \tfrac12, 1\}$.
	}
	\label{fig:Ex-Aronsson-rates}
\end{figure}



Finally, we investigate the error induced by the extension of the boundary datum $\wt{g}_\ve$, we define
\[
\wt{g}_{\ve}(x,y) = g(T(x,y)), \qquad T(x,y) = \frac{(x,y)}{\|(x,y) \|_{\ell^\infty} } \in \partial \Omega.
\]
Clearly, this choice of function $\wt{g}_{\ve}$ satisfies Assumption~\textbf{BC}.\ref{Ass.BC.ext}. For this inexact boundary condition, in \Cref{fig:Ex-Aronsson-inexact}, we plot the error $u_\frakh - u$  for $h = 2^{-8}, \ve = 2^{-4.75}$, and $ \theta = 2^{-3.25}$. Due to the inexact boundary, the error is larger
than the one shown in \Cref{fig:Ex-Aronsson-uh} (left). We also plot the $L^{\infty}$ errors for the same choices of parameters with $\beta \in \{ 0, \tfrac14, \tfrac13, \tfrac12, 1\}$. We observe again that, in accordance to the theory, the solutions do not seem to converge for $\beta = 0$ or $1$. For $\beta \in \{ \tfrac14, \tfrac13, \tfrac12\}$, we measure the convergence orders using least square fit for the errors of $h = 2^{-7}, 2^{-8}, 2^{-9}$. The orders are about $0.55, 0.76, 0.30$ respectively, which are worse than the ones obtained for exact boundary, but still better than the theoretical ones in \Cref{Thm:homo-error}. Notice that, in fact, \Cref{Thm:homo-error} does not guarantee convergence for $\beta = \tfrac12$.


\begin{figure}[!htb]
	\includegraphics[width=0.45\linewidth]{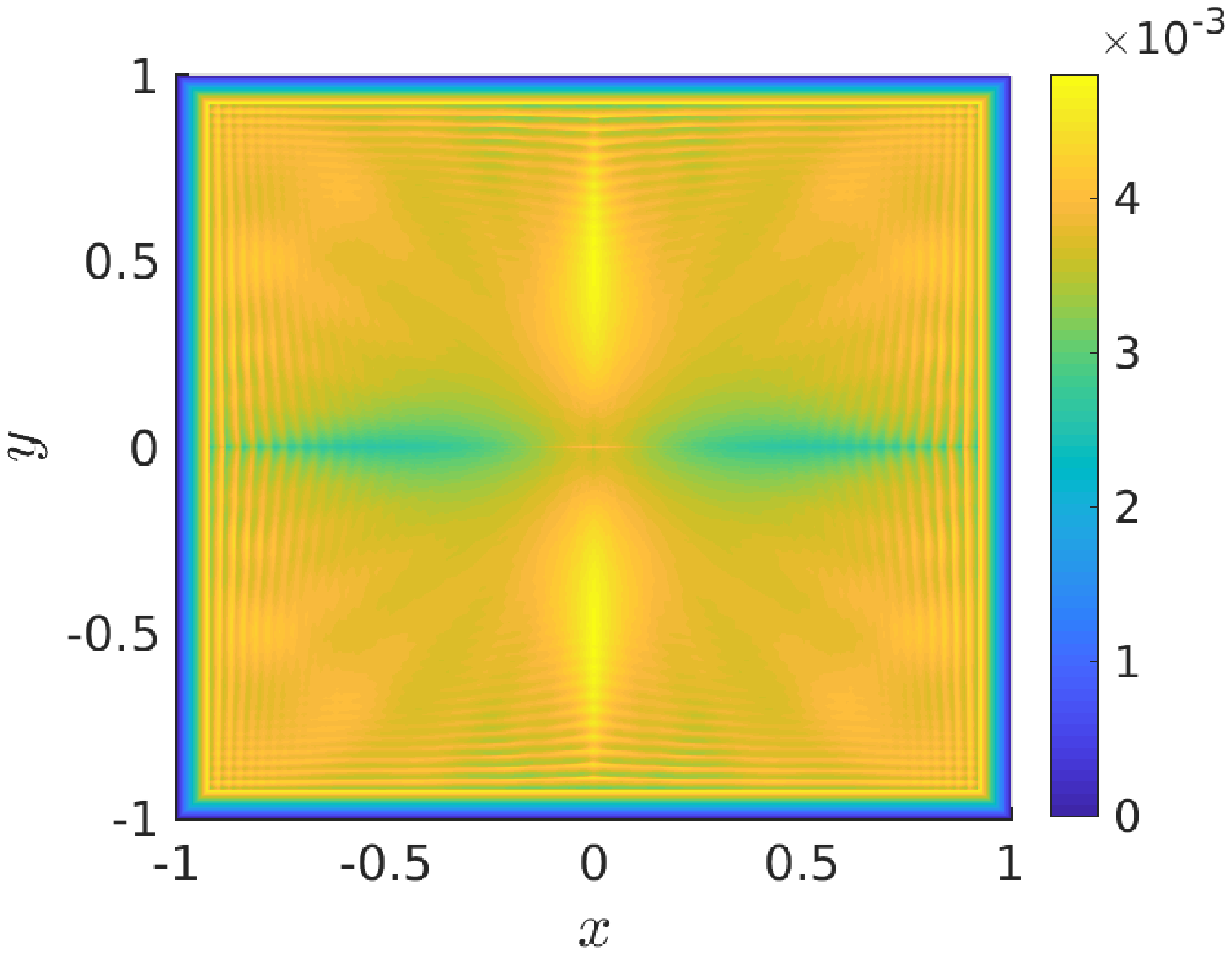}
	\includegraphics[width=0.45\linewidth]{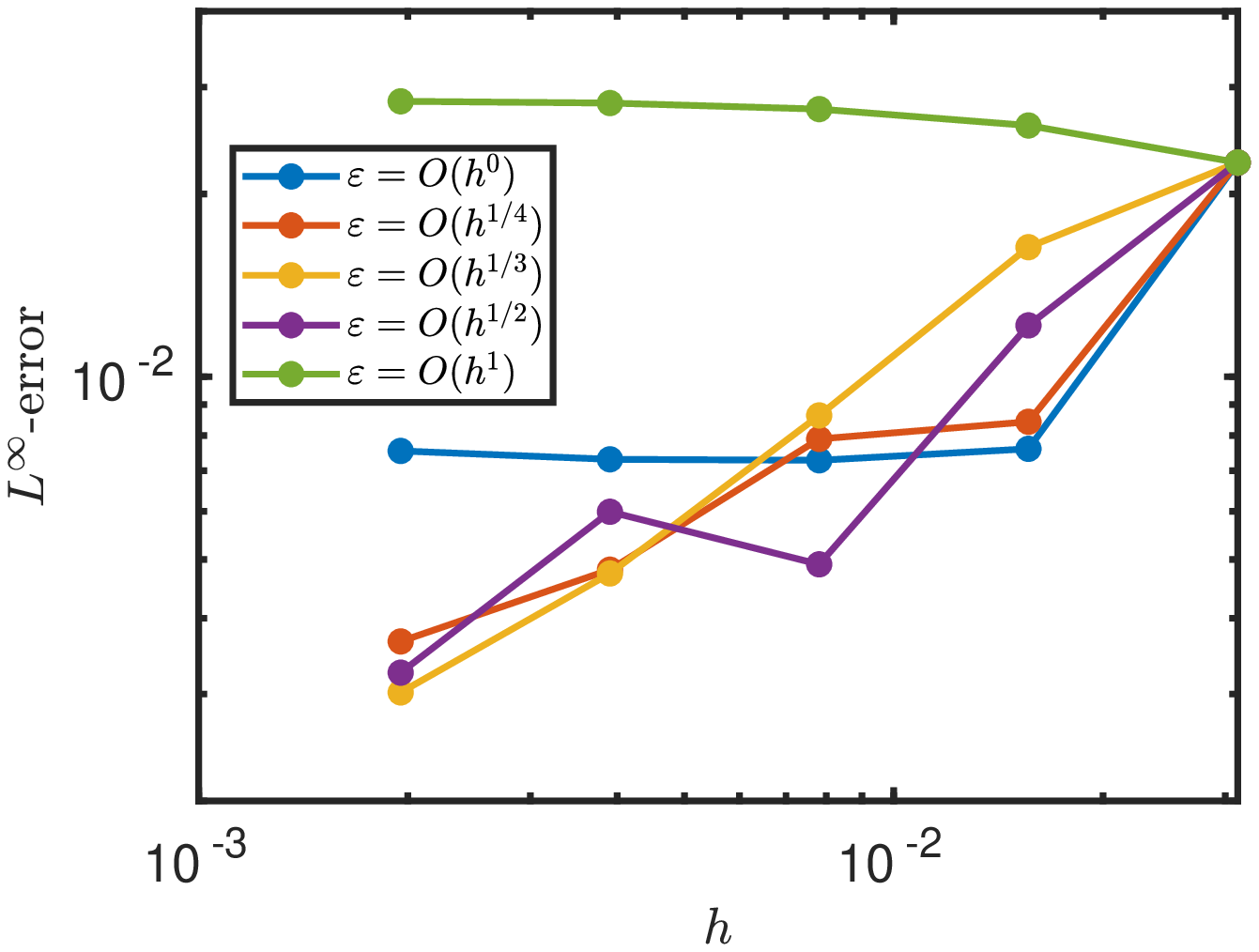}
	\caption{\small \Cref{Ex-Aronsson} with inexact boundary condition $\wt{g}_{\ve}$. Left: Plot of $u_\frakh - u$. Right: Experimental rates of convergence for $L^{\infty}$ errors with $\beta \in \{ 0, \tfrac14, \tfrac13, \tfrac12, 1\}$.}
	\label{fig:Ex-Aronsson-inexact}
\end{figure}

\end{Example}

%
%
%
%

\begin{Example}[$f > 0$]\label{Ex-quad-x}
Let $\Omega = B_1$ be the unit ball, $f \equiv 1$ and $u(x, y) = (1 - x^2)/2$. In this example, we have a smooth right hand side $f$ and solution  $u$.

	\begin{figure}[!htb]
		\includegraphics[width=0.45\linewidth]{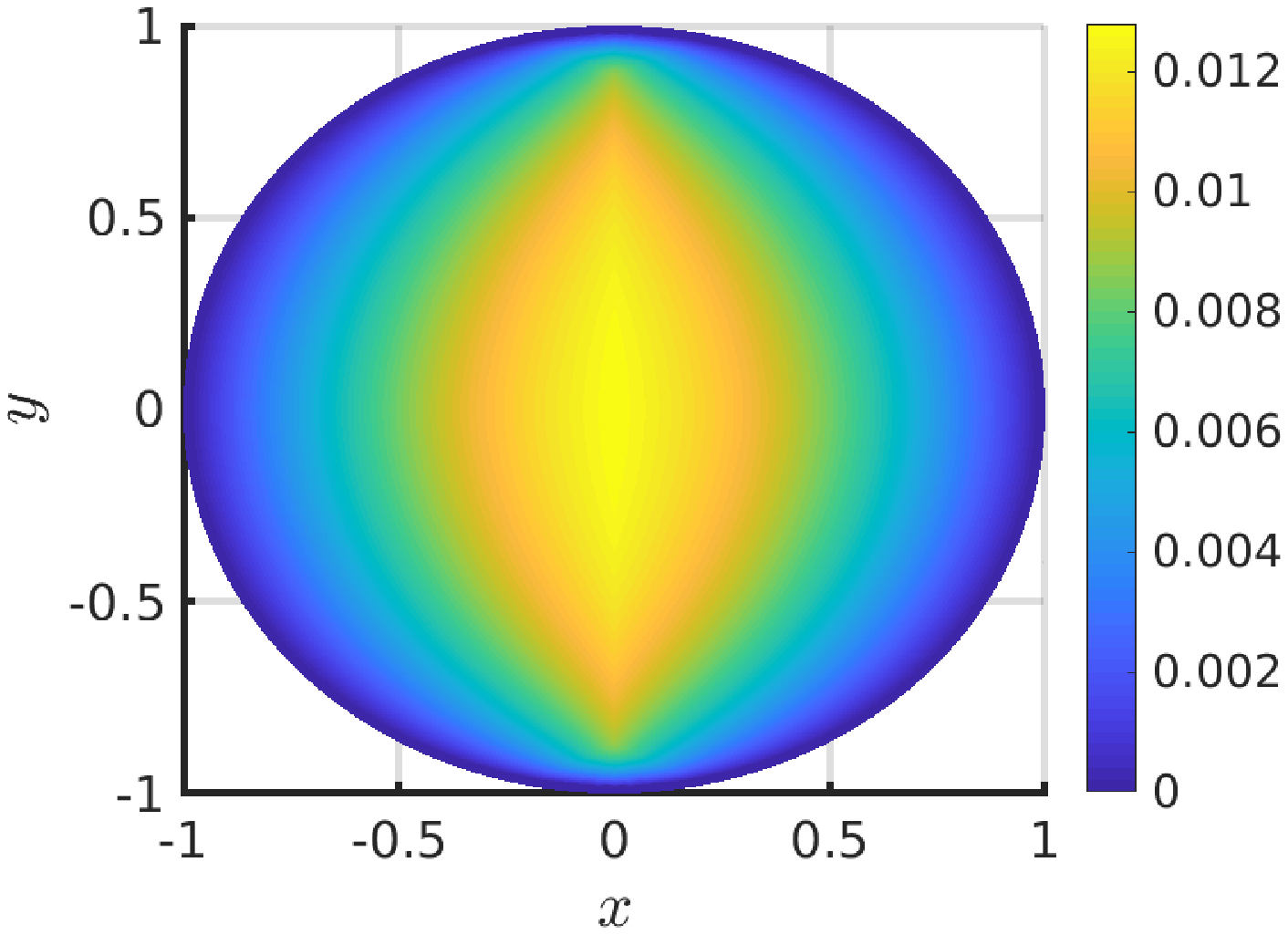}
		\includegraphics[width=0.45\linewidth]{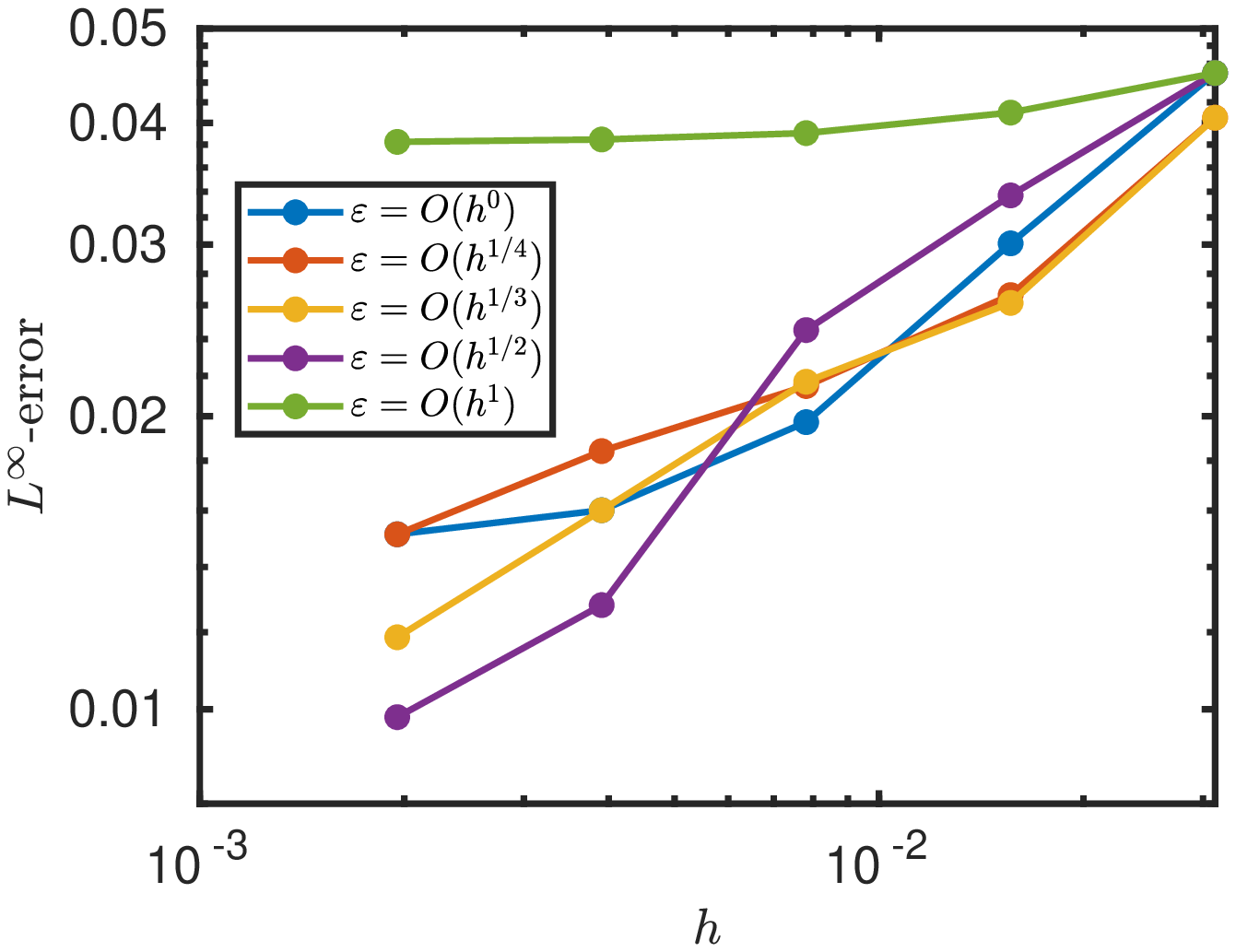}
		\caption{\small \Cref{Ex-quad-x} with exact boundary condition $\wt{g}_{\ve} = u$. Left: Plot of $u_\frakh - u$. Right: Experimental rates of convergence for $L^{\infty}$ errors with $\beta \in \{ 0, \tfrac14, \tfrac13, \tfrac12, 1\}$.}
		\label{fig:Ex-quad-x}
	\end{figure}

	We choose $\wt{g}_{\ve} = u$ and compute the numerical solution $u_\frakh$ for $h = 2^{-8}, \ve = 2^{-6.5}$, and $ \theta = 2^{-1.5}$. The error $u_\frakh - u$ is displayed in \Cref{fig:Ex-quad-x} (left), where the largest error appears near the $x$--axis. To measure the orders of convergence of the $L^{\infty}$ error, we let  $\frakh = (h, \ve, \theta)$ satisfy \eqref{eq:exp-parameters} with $\beta \in \{ 0, \tfrac14, \tfrac13, \tfrac12, 1\}$. We observe again that the solutions do not seem to converge for $\beta = 0$ 
	or $1$ as expected from \Cref{Thm:inhomo-error}. For $\beta \in \{ \tfrac14, \tfrac13, \tfrac12\}$, the convergence orders measured for errors of $h = 2^{-7}, 2^{-8}, 2^{-9}$ using least squares are about $0.25, 0.44, 0.66$ respectively. The orders for $\beta \in \{ \tfrac13, \tfrac12\}$ are better than the theoretical ones in \Cref{Thm:inhomo-error}. In fact, \Cref{Thm:inhomo-error} does not guarantee convergence for $\beta = \tfrac12$.
\end{Example}

%
%
%
%

In this section we have shown several numerical examples of the numerical solution to \eqref{eq:BVPNinfLaplace}, both with homogeneous, and positive right hand side. In all of the considered examples, the assumptions of either \Cref{Thm:inhomo-error} or \Cref{Thm:homo-error} are satisfied. We observe the convergence of the numerical scheme when $\ve$ and $\theta$ are properly scaled with respect to $h$. From the experiments, the rates of convergences are usually better than the predicted rates in \Cref{Thm:inhomo-error} and \Cref{Thm:homo-error}. This may be due to the fact that, in our analysis, we only exploited up to Lipschitz regularity of the solution. How to account for a smoother solution, like the ones we presented is a matter for future investigation.

\section*{Acknowledgement}

The work of the authors is partially supported by NSF grant DMS-2111228.

\bibliographystyle{amsplain}
\bibliography{inf_laplacian}

\end{document}